\documentclass[10pt, oneside]{article}

\usepackage{fancyhdr}
\fancyheadoffset{ 0.5 cm}
\fancyfootoffset{ 0.5 cm}

\usepackage{titlesec}
\titlelabel{\thetitle.\quad}

\usepackage{times}
\usepackage{geometry}
\usepackage{graphicx}
\usepackage{amssymb, amscd}
\usepackage{amsmath}
\makeatletter
\@addtoreset{equation}{section}
\makeatother
\numberwithin{equation}{section}
\usepackage{amsthm} 
\usepackage{mathrsfs}
\usepackage{epstopdf}
\usepackage{enumerate}
\usepackage{url}
\usepackage{graphicx}
\usepackage{mathtools}
\usepackage{tikz}							
\usepackage{setspace}
\usepackage{amsfonts, amsmath, wasysym}

\usepackage{appendix}

\usepackage{esint}

\usepackage{relsize}

\usepackage{shuffle}

\usepackage{float}
\restylefloat{table}

\usepackage{hyperref}
\usepackage{cite}

\usepackage{bbm}

\usepackage[draft]{optional}

\usepackage{color}

\usepackage{xcolor}
\newtagform{red}{\color{red}(}{)}

\usepackage{enumitem}

\usepackage{mathtools}

\usepackage{amsmath}
\makeatletter
\@addtoreset{equation}{section}
\makeatother
\numberwithin{equation}{section}

\theoremstyle{plain}
\newtheorem{theorem}{Theorem}[section]
\newtheorem{lemma}[theorem]{Lemma}

\theoremstyle{definition}

\theoremstyle{remark}

\newtheorem{remark}[theorem]{\bf{Remark}}

\def\DataSig{DataS{\i}g }

\newcommand{\cO}{\ensuremath{{\cal O}}}
\newcommand{\cU}{\ensuremath{{\cal U}}}

\newcommand{\m}{\ensuremath{{\cal M}}}
\newcommand{\n}{\ensuremath{{\cal N}}}

\newcommand{\cc}{\ensuremath{{\cal C}}}
\newcommand{\cd}{\ensuremath{{\cal D}}}
\newcommand{\cf}{\ensuremath{{\cal F}}}

\newcommand{\cj}{\ensuremath{{\cal J}}}
\newcommand{\cl}{\ensuremath{{\cal L}}}

\newcommand{\cH}{\ensuremath{{\cal H}}}

\newcommand{\cX}{\ensuremath{{\cal X}}}

\newcommand{\al}{\alpha}
\newcommand{\be}{\beta}

\newcommand{\de}{\delta}

\renewcommand{\th}{\theta}

\newcommand{\vph}{\varphi}
\newcommand{\ep}{\varepsilon}

\newcommand{\R}{\ensuremath{{\mathbb R}}}
\newcommand{\N}{\ensuremath{{\mathbb N}}}
\newcommand{\B}{\ensuremath{{\mathbb B}}}

\newcommand{\Z}{\ensuremath{{\mathbb Z}}}

\newcommand{\bbP}{\ensuremath{{\mathbb P}}}

\newcommand{\ovB}{\ensuremath{{ \overline{ \mathbb B }}}}

\newcommand{\bx}{{\bf x}}
\newcommand{\by}{{\bf y}}

\newcommand{\bolde}{{\bf e}}

\newcommand{\bA}{{\bf A }}

\newcommand{\bI}{{\bf I}}
\newcommand{\bII}{{\bf II}}

\DeclareMathOperator{\dist}{dist}

\DeclareMathOperator{\Span}{Span}
\DeclareMathOperator{\ovSpan}{\overline{Span}}
\DeclareMathOperator{\Reach}{Reach}

\DeclareMathOperator{\argmax}{argmax}
\DeclareMathOperator{\argmin}{argmin}

\DeclareMathOperator{\pack}{pack}
\DeclareMathOperator{\cov}{cov}

\newcommand{\support}{{\rm support}}

\newcommand{\beq}{\begin{equation}}
\newcommand{\eeq}{\end{equation}}
\newcommand{\beqa}{\begin{equation}\begin{aligned}}
\newcommand{\eeqa}{\end{aligned}\end{equation}}
\newcommand{\brmk}{\begin{rmk}}
\newcommand{\ermk}{\end{rmk}}
\newcommand{\partref}[1]{\hbox{(\csname @roman\endcsname{\ref{#1}})}}

\newcommand{\conv}{{\mathrm{conv}}}

\newcommand{\train}{{\mathrm{train}}}
\newcommand{\test}{{\mathrm{test}}}

\newcommand{\f}{\ensuremath{{\cal F}}}

\usepackage{soul}

\newcommand{\twopartdef}[4]
{\setstretch{1.5}
	\left\{
		\begin{array}{ll}
			#1 & \mbox{if } #2 \\ 
			#3 & \mbox{if } #4
		\end{array}
	\right.
}

\makeatletter
\def\dual#1{\expandafter\dual@aux#1\@nil}
\def\dual@aux#1/#2\@nil{\begin{tabular}{@{}c@{}}#1\\#2\end{tabular}}
\makeatother

\makeatletter
\def\three#1{\expandafter\three@aux#1\@nil}
\def\three@aux#1/#2/#3\@nil{\begin{tabular}{@{}c@{}c@{}}#1\\#2\\#3\end{tabular}}
\makeatother

\makeatletter
\def\four#1{\expandafter\four@aux#1\@nil}
\def\four@aux#1/#2/#3/#4\@nil{\begin{tabular}{@{}c@{}c@{}c@{}}#1\\#2\\#3\\#4\end{tabular}}
\makeatother

\title{Greedy Recombination Interpolation Method (GRIM)}
\author{Terry Lyons and Andrew D. McLeod}
\date{\today}

\begin{document}
\usetagform{red}
\maketitle

\begin{abstract}
In this paper we develop the \emph{Greedy 
Recombination Interpolation Method} (GRIM) for 
finding sparse approximations of functions initially 
given as linear combinations of some (large) number
of simpler functions. In a similar spirit to the 
\emph{CoSaMP} algorithm, GRIM combines
dynamic growth-based interpolation techniques and 
thinning-based reduction techniques.
The dynamic growth-based aspect is a modification of the 
greedy growth utilised in the 
\emph{Generalised Empirical Interpolation Method} (GEIM). 
A consequence of the modification is that our growth is not
restricted to being one-per-step as it is in GEIM.
The thinning-based aspect is carried out by \emph{recombination},
which is the crucial component of the recent ground-breaking 
\emph{convex kernel quadrature} method.
GRIM provides the first use of recombination outside the 
setting of reducing the support of a measure.
The sparsity of the approximation found by GRIM is controlled
by the geometric concentration of the data in a sense that is 
related to a particular \emph{packing number} of the data.
We apply GRIM to a kernel quadrature task 
for the radial basis function kernel, 
and verify that its performance matches that of other
contemporary kernel quadrature techniques.
\end{abstract}

{\small \tableofcontents}

\section{Introduction}
\label{intro}
This article considers finding sparse approximations of functions with the aim of reducing computational 
complexity. Applications of sparse representations are wide ranging and include the theory of compressed 
sensing \cite{DE03,CT05,Don06,CRT06}, image processing \cite{EMS08,BMPSZ08,BMPSZ09}, facial recognition 
\cite{GMSWY09}, data assimilation \cite{MM13}, explainability \cite{FMMT14,DF15}, 
sensor placement in nuclear reactors \cite{ABGMM16,ABCGMM18}, reinforcement learning 
\cite{BS18,HS19}, DNA denoising \cite{KK20}, and inference acceleration within machine learning \cite{ABDHP21,NPS22}.

Loosely speaking, the commonly shared aim of sparse
approximation problems is to approximate a 
complex system using only a few elementary features.
In this generality, the problem is commonly tackled 
via \emph{Least Absolute Shrinkage and 
Selection Operator (LASSO)} regression, 
which involves solving a minimisation problem under 
$l^1$-norm constraints. 
Constraining the $l^1$-norm is a convex 
relaxation of constraining the $l^0$-pseudo-norm, which 
counts the number of non-zero components of a vector.
Since $p=1$ is the smallest positive real number for which
the $l^p$-norm is convex, constraining the $l^1$-norm can be
viewed as the best convex approximation of 
constraining the $l^0$-pseudo-norm (i.e. constraring the 
number of non-zero components).
Implementations of LASSO techniques
can lead to near optimally 
sparse solutions in certain contexts \cite{DET06,Tro06,Ela10}.
The terminology "LASSO" 
originates in \cite{Tib96}, though imposing $l^1$-norm
constraints is considered in the earlier works
\cite{CM73,SS86}. More recent LASSO-type techniques 
may be found in, for example,
\cite{LY07,DGOY08,GO09,XZ16,TW19}.

Alternatives to LASSO include the
\emph{Matching Pursuit} (MP) sparse approximation 
algorithm introduced in \cite{MZ93} and the 
\emph{CoSaMP} algorithm introduced in \cite{NT09}. 
The MP algorithm greedily 
grows a collection of non-zero weights one-by-one 
that are used to construct the approximation.
The CoSaMP algorithm operates in a similar greedy manner
but with the addition of two key properties. 
The first is that the non-zero weights are no longer 
found one-by-one; groups of several non-zero weights
may be added at each step. 
Secondly, CoSaMP incorporates a naive prunning procedure
at each step. After adding in the new weights at a step, 
the collection is then pruned down by retaining only the
$m$-largest weights for a chosen integer $m \in \Z_{\geq 1}$.
The ethos of CoSaMP is particularly notworthy for our
purposes.

In this paper we assume that the system 
of interest is known to be a linear combination of some 
(large) number of features. Within this setting the 
goal is to identify a linear combination of a strict 
sub-collection of the features (i.e. not \emph{all} 
the features) that gives, in some sense, a good 
approximation of the system (i.e. of the original
given linear combination of all the features).
Depending on the particular context considered, 
techniques for finding such sparse approximations 
include the pioneering \emph{Empirical Interpolation Method} 
\cite{BMNP04,GMNP07,MNPP09}, its subsequent generalisation 
the \emph{Generalised Empirical Interpolation Method} (GEIM)
\cite{MM13,MMT14}, \emph{Pruning} 
\cite{Ree93,AK13,CHXZ20,GLSWZ21}, \emph{Kernel Herding}
\cite{Wel09a,Wel09b,CSW10,BLL15,BCGMO18,TT21,PTT22}, 
\emph{Convex Kernel Quadrature} \cite{HLO21}, and
\emph{Kernel Thinning} \cite{DM21a,DM21b,DMS21}.
The LASSO approach based on $l^1$-regularisation can
still be used within this framework.

It is convenient for our purposes to consider the 
following loose categorisation of techniques for finding 
sparse approximations in this setting;
those that are \emph{growth-based}, and those that 
are \emph{thinning-based}.  
Growth-based methods seek to inductively increase
the size of a sub-collection of features, until 
the sub-collection is rich enough to well-approximate
the entire collection of features.
Thinning-based methods seek to inductively identify
features that may be discarded without significantly 
affecting how well the remaining features can 
approximate the original entire collection.
Of the techniques mentioned above, MP, EIM, GEIM and 
Kernel Herding are growth-based, whilst LASSO, 
Convex Kernel Quadrature and 
Kernel Thinning are thinning-based.
An important observation is that the CoSaMP algorithm 
combines both growth-based and thinning-based techniques.

In this paper we work in the setting considered by Maday et 
al. when introducing GEIM \cite{MM13,MMT14,MMPY15}.
Namely, we let $X$ be a real Banach space and 
$\n \in \Z_{\geq 1}$ be a (large) positive integer. Assume that 
$\cf = \{ f_1 , \ldots , f_{\n} \} \subset X$ is a collection
of non-zero elements, and that $a_1 , \ldots , a_{\n} \in \R 
\setminus \{0\}$. Consider the element $\vph \in X$ defined 
by $\vph := \sum_{i=1}^{\n} a_i f_i$. 
Let $X^{\ast}$ denote the dual of $X$ and suppose that
$\Sigma \subset X^{\ast}$ is a finite subset with
cardinality $\Lambda \in \Z_{\geq 1}$. 
We use the terminology that the set $\cf$ consists of the
\emph{features} whilst the set $\Sigma$ consists of \emph{data}.
Then we consider the following sparse approximation problem.
Given $\ep > 0$, find an element 
$u = \sum_{i=1}^{\n} b_i f_i\in \Span(\cf) \subset X$ 
such that the cardinality of the set 
$\{ i \in \{1, \ldots , \n \} : b_i \neq 0 \}$ is 
\emph{less} than $\n$ and that $u$ is close to $\vph$ 
throughout $\Sigma$ in the sense that, for every 
$\sigma \in \Sigma$, we have $|\sigma(\vph-u)| \leq \ep$.

The tasks of approximating 
sums of continuous functions on finite subsets of Euclidean 
space, cubature for empirical measures, and kernel quadrature 
for kernels defined on finite sets are included as 
particular examples within this general mathematical 
framework (see Section \ref{Motivation}) for full details).
The inclusion of approximating a linear combination of continuous
functions within this general framework ensures that several 
machine learning related tasks are included.
For example, each layer of a neural network is typically 
given by a linear combination of continuous functions.
Hence the task of approximating a layer within a neural network
by a sparse layer is covered within this general framework.
Observe that there is no requirement that the original layer
is fully connected; in particular, it may itself already be 
a sparse layer.
Consequently, any approach to the approximation problem 
covered in the general framework above could be used to carry 
out the reduction step (in which a fixed proportion of a models 
weights are reduced to zero) in the recently proposed 
\emph{sparse training} algorithms 
\cite{GLMNS18,EJOPR20,LMPY21,JLPST21,FHLLMMPSWY23}.

The GEIM approach developed by Maday et al. 
\cite{MM13,MMT14,MMPY15} involves dynamically growing 
a subset $F \subset \cf$ of the features and a subset 
$L \subset \Sigma$ of the date. At each step a new 
feature from $\cf$ is added to $F$, and a new piece of 
data from $\Sigma$ is added to $L$. An approximation of 
$\vph$ is then constructed via the following interpolation 
problem. Find a linear combination of the 
elements in $F$ that coincides with $\vph$
throughout the subset $L \subset \Sigma$.
This dynamic growth is \emph{feature-driven} in the following 
sense. The new feature to be added to $F$ is determined, 
and this new feature is subsequently used to determine 
how to extend the collection $L$ of linear functionals on 
which an approximation will be required to match the target.

The growth is greedy in the sense that the element chosen
to be added to $F$ is, in some sense, the one worst 
approximated by the current collection $F$. 
If we let $f \in \cf$ be the newly selected feature 
and $J[f]$ be the linear combination of the previously 
selected features in $F$ that coincides with $f$ on $L$, 
then the element to be added to $L$ is the one achieving the 
maximum absolute value when acting on $f - J[f]$.
A more detailed overview of GEIM can be found in 
Section \ref{Motivation} of this paper; full details of 
GEIM may be found in \cite{MM13,MMT14,MMPY15}.

Momentarily restricting to the particular task of kernel 
quadrature, the recent thinning-based 
\emph{Convex Kernel Quadrature} approach proposed by 
Satoshi Hayakawa, the first author, and Harald Oberhauser 
in \cite{HLO21} achieves out performs existing techniques 
such as Monte Carlo, Kernel Herding 
\cite{Wel09a,Wel09b,CSW10,BLL15,BCGMO18,PTT22}, 
Kernel Thinning \cite{DM21a,DM21b,DMS21}
to obtain new state-of-the-art results.
Central to this approach is the \emph{recombination}
algorithm \cite{LL12,Tch15,ACO20}.
Originating in \cite{LL12} as a technique 
for reducing the support of a probability 
measure whilst preserving a specified list
of moments, at its core recombination is a method of
reducing the number of non-zero components in a 
solution of a system of linear equations whilst preserving 
convexity. 

An improved implementation of recombination is 
given in \cite{Tch15} that is significantly more efficient
than the original implementation in \cite{LL12}.
A novel implementation of the method from \cite{Tch15} was 
provided in \cite{ACO20}. A modified variant of the 
implementation from \cite{ACO20} is used in the convex
kernel quadrature approach developed in \cite{HLO21}.
A method with the same complexity as the original 
implementation in \cite{LL12} was introduced in 
the works \cite{FJM19,FJM22}.

Returning to our general framework, we develop the 
\emph{Greedy Recombination Interpolation Method} (GRIM)
which is a novel hybrid combination of the dynamic growth of 
a greedy selection algorithm, in a similar spirit to GEIM 
\cite{MM13,MMT14,MMPY15},
with the thinning reduction of recombination that underpins 
the successful convex kernel quadrature approach of \cite{HLO21}.
GRIM dynamically grows a collection of linear functionals 
$L \subset \Sigma$. After each extension of $L$, we apply 
recombination to find an approximation of $\vph$ that 
coincides with $\vph$ throughout $L$ (cf. 
the \emph{recombination thinning} Lemma 
\ref{Banach_recombination_lemma}). 
Subsequently, for a chosen integer $m \in \Z_{\geq 1}$, 
we extend $L$ by adding the $m$ linear functionals from 
$\Sigma$ achieving the $m$ largest absolute values 
when applied to the difference between $\vph$ and the 
current approximation of $\vph$ 
(cf. Section \ref{Banach_GRIM_alg}).
We inductively repeat these steps a set number of times.
Evidently GRIM is 
a hybrid of growth-based and thinning-based techniques 
in a similar spirit to the CoSaMP algorithm \cite{NT09}. 

Whilst the dynamic growth of GRIM is in a similar spirit to 
that of GEIM, there are several important distinctions between
GRIM and GEIM.
The growth in GRIM is \emph{data-driven} rather than
\emph{feature-driven}. 
The extension of the data to be interpolated with respect to 
in GRIM does \emph{not} involve making any choices of features
from $\cf$. The new information to be matched is determined
by examining where in $\Sigma$ the current approximation is 
furthest from the target $\vph$
(cf. Section \ref{Banach_GRIM_alg}).

Expanding on this point, 
we only dynamically grow a subset $L \subset \Sigma$ 
of the data and do \emph{not} grow a subset $F \subset \cf$ of 
the features. 
In particular, we do \emph{not} pre-determine the features 
that an approximation will be a linear combination of.
Instead, the features to be used are determined by 
recombination (cf. the \emph{recombination thinning}
Lemma \ref{Banach_recombination_lemma}).
Indeed, besides an upper bound on the number of features 
being used to construct an approximation 
(cf. Section \ref{Banach_GRIM_alg}), we have \emph{no}
control over the features used. 
Allowing recombination and the data $\Sigma$ to determine 
which features are used means there is no requirement to 
use the features from one step at subsequent steps.
There is no requirement that \emph{any} of the features
used at one specific step must be used in any of the subsequent
steps.

GRIM is \emph{not} 
limited to extending the subset $L \subset \Sigma$ by a single
linear functional at each step. 
GRIM is capable of extending $L$ by $m$ linear functionals 
for any given integer $m \in \Z_{\geq 1}$ (modulo 
restrictions to avoid adding more linear functionals 
than there are in the original subset $\Sigma \subset X^{\ast}$
itself).
The number of new linear functionals to be added at a step 
is a hyperparameter that can be optimised during numerical
implementations.

Unlike \cite{HLO21}, our use of recombination is not 
restricted to the setting of reducing the support of a 
discrete measure. 
After each extension of $L \subset \Sigma$, 
we use recombination 
\cite{LL12,Tch15,ACO20} to find an element $u \in \Span(\cf)$ 
satisfying that $u \equiv \vph$ throughout $L$.
Recombination is applied to the linear system determined 
by the combination of the 
set $\{ \sigma(\vph) : \sigma \in L \}$, for a given 
$\sigma \in L$ we get the equation
$\sum_{i=1}^{\n} a_i \sigma(f_i) = \sigma(\vph)$, 
and the sum of the coefficients $a_1 , \ldots , a_{\n}$ 
(i.e. the trivial equation 
$a_1 + \ldots + a_{\n} = \sum_{i=1}^{\n} a_i$) 
(cf. the \emph{recombination thinning}
Lemma \ref{Banach_recombination_lemma}).

Our use of recombination means, in particular, that 
GRIM can be used for cubature and kernel quadrature 
(cf. Sections \ref{Motivation} and \ref{Banach_GRIM_alg}).
Since recombination preserves convexity \cite{LL12},
the benefits of convex weights enjoyed by the convex kernel 
quadrature approach in \cite{HLO21} are inherited 
by GRIM (cf. Section \ref{Motivation}).

Moreover, at each step we optimise our use of recombination
over multiple permutations of the orderings of the 
equations determining the linear system to which 
recombination is applied (cf. the 
\textbf{Banach Recombination Step} in Section 
\ref{Banach_GRIM_alg}). 
The number of permutations to be considered at each step 
gives a parameter that may be optimised during applications 
of GRIM.

Whilst we analyse the complexity cost of the 
\textbf{Banach GRIM} algorithm 
(cf. Section \ref{complexity_cost_subseq}), computational 
efficiency is not our top priorities. 
GRIM is designed to be a one-time tool; it is applied a 
single time to try and find a sparse approximation of the
target system $\vph \in X$.
The cost of its implementation is then recouped through 
the repeated use of the resulting approximation for 
inference via new inputs.
Thus GRIM is ideally suited for use in cases where the 
model will be repeatedly computed on new inputs for the 
purpose of inference/prediction.
Examples of such models include, in particular, those 
trained for character recognition, medical diagnosis, 
and action recognition.

With this in mind, the primary aim of our complexity cost 
considerations is to verify that implementing GRIM is 
feasible. We verify this by proving that, at worst, the 
complexity cost of GRIM is 
$\cO(\Lambda^2 \n + s\Lambda^4 \log(\n) )$ where 
$\n$ is the number of features in $\cf$, $\Lambda$ is the 
number of linear functionals forming the data $\Sigma$, 
and $s$ is the maximum number of shuffles considered during 
each application of recombination
(cf. Lemma \ref{GRIM_cost_lemma}).

The remainder of the paper is structured as follows. 
In Section \ref{Motivation} we fix the mathematical 
framework that will be used throughout the 
article and motivate its consideration. In particular, 
we illustrate some specific examples covered by our framework.
Additionally, we summarise the GEIM approach of Maday et al. 
\cite{MM13,MMT14,MMPY15} and highlight the fundamental 
differences in the philosophies underlying GEIM and GRIM.

An explanation of the recombination algorithm and its use 
within our setting is provided in subsection \ref{recomb_thin}.
In particular, we prove the 
\emph{Recombination Thinning} Lemma 
\ref{Banach_recombination_lemma} detailing our use of 
recombination to find $u \in \Span(\cf)$ coinciding with 
$\vph$ on a given finite subset $L \subset X^{\ast}$.

The \textbf{Banach GRIM} algorithm is both presented and 
discussed in Section \ref{Banach_GRIM_alg}. 
We formulate the \textbf{Banach Extension Step}
governing how we extend an existing collection of linear
functionals $L \subset \Sigma$, 
the \textbf{Banach Recombination Step} specifying how we 
use the \emph{recombination thinning} Lemma 
\ref{Banach_recombination_lemma} in the \textbf{Banach GRIM} 
algorithm, 
and provide an upper bounds for the number of 
features used to construct the approximation at each step.

The complexity cost of the \textbf{Banach GRIM} algorithm 
is considered in Section \ref{complexity_cost_subseq}. 
We prove Lemma \ref{GRIM_cost_lemma} establishing the 
complexity cost of any implementation of the 
\textbf{Banach GRIM} algorithm, and subsequently establish 
an upper bound complexity cost for the most expensive 
implementation.

The theoretical performance of the \textbf{Banach GRIM} 
algorithm is considered in Section \ref{conv_gen}.
Theoretical guarantees in terms of a specific 
geometric property
of $\Sigma$ are established for the \textbf{Banach GRIM} 
algorithm in which a single new linear functional is chosen 
at each step (cf. the \textbf{Banach GRIM Convergence} 
Theorem \ref{banach_conv}).
The specific geometric property is related to a particular
\emph{packing number} of $\Sigma$ in $X^{\ast}$ 
(cf. Subsection \ref{Main_Theoretical_Result}).
The packing number of a subset of a Banach space is closely
related to the covering number of the subset.
Covering and packing numbers, first studied by Kolmogorov
\cite{Kol56}, arise in a variety of contexts including 
eigenvalue estimation \cite{Car81,CS90,ET96}, 
Gaussian Processes \cite{LL99,LP04}, and machine learning 
\cite{EPP00,SSW01,Zho02,Ste03,SS07,Kuh11,MRT12,FS21}

In Section \ref{numerical} we compare the performance of GRIM
against other reduction techniques on three tasks.
The first is an $L^2$-approximation 
task motivated by an example in \cite{MMPY15}. 
The second is a kernel 
quadrature task using machine learning datasets as 
considered in \cite{HLO21}.
In particular, we illustrate that GRIM matches the 
performance of the tailor-made \emph{convex kernel quadrature}
technique developed in \cite{HLO21}.
The third task is approximating the action recognition model from 
\cite{JLNSY17} for the purpose of inference acceleration.
In particular, this task involves approximating a function 
in a pointwise sense that is outside the Hilbert space 
framework of the proceeding two examples.
\vskip 4pt
\noindent 
\emph{Acknowledgements}: This work was supported by 
the \DataSig Program under the EPSRC grant 
ES/S026347/1, the Alan Turing Institute under the
EPSRC grant EP/N510129/1, the Data Centric Engineering 
Programme (under Lloyd's Register Foundation grant G0095),
the Defence and Security Programme (funded by the UK 
Government) and the Hong Kong Innovation and Technology 
Commission (InnoHK Project CIMDA). 
This work was funded by the Defence and Security
Programme (funded by the UK Government).

\section{Mathematical Framework \& Motivation}
\label{Motivation}
In this section we rigorously formulate the sparse 
approximation problem that GRIM will be designed to tackle. 
Further, we briefly summarise the 
\emph{Generalised Empirical Interpolation Method} (GEIM) 
of Maday et al. \cite{MM13,MMT14,MMPY15} 
to both highlight some of the ideas we utilise in GRIM, 
and to additionally highlight the key novel properties 
\emph{not} satisfied by GEIM that \emph{will} be 
satisfied by GRIM. 
We first fix the mathematical framework in which 
we will work for the remainder of this paper.
 
Let $X$ be a Banach space and 
$\n \in \Z_{>0}$ be a (large) positive integer.
Assume that $\cf = \{ f_1 , \ldots , f_{\n} \}
\subset X$ is a collection of non-zero elements, 
and that $a_1 , \ldots , a_{\n} \in \R \setminus \{0\}$.
Consider the element $\vph \in X$ defined by
\beq
    \label{intro_varphi_def}
        \vph = \sum_{i=1}^{\n} a_i f_i.
\eeq
Let $X^{\ast}$ denote the dual of $X$ and  
suppose that $\Sigma \subset X^{\ast}$ is a finite 
subset of cardinality $\Lambda \in \Z_{\geq 1}$. 
We use the terminology that 
the set $\cf$ consists of \emph{features} whilst
the set $\Sigma$ consists of \emph{data}.
We consider the following sparse 
approximation problem. Given $\ep > 0$, find an 
element $u = \sum_{i=1}^{\n} b_i f_i \in \Span(\cf)$ 
such that the cardinality of the set 
$\left\{ i \in \{1 ,\ldots ,\n\} ~:~ b_i \neq 0\right\}$
is \emph{less} than $\n$ and that $u$ 
is close to $\vph$ throughout $\Sigma$ in the sense 
that, for every $\sigma \in \Sigma$, we have
$|\sigma(\vph-u)| \leq \ep$.

The task of finding a sparse approximation 
of a sum of continuous functions, defined on a  finite set 
of Euclidean space, is within this framework.
More precisely, let $N , d , E \in \Z_{\geq 1}$,
$a_1 , \ldots , a_N \in \R$,
$\Omega \subset \R^d$ be a finite subset,
and, for each $i \in \{1 , \ldots , N\}$, 
$f_i \in C^{0}\left(\Omega;\R^E\right)$ be a continuous 
function $\Omega \to \R^E$.
Then finding a sparse approximation of the continuous 
function $F := \sum_{i=1}^N a_i f_i$ is within this framework.
To see this, first let $e_1 , \ldots , e_E \in \R^E$ be the standard basis of 
$\R^E$ that is orthonormal with respect to the Euclidean dot product 
$\left< \cdot , \cdot \right>_{\R^E}$ on $\R^E$.
Then note that for each point $p \in \Omega$ and every $j \in \{1, \ldots , E\}$ 
the mapping $f \mapsto \left< f(p) , e_j \right>_{\R^E}$ determines a linear 
functional $C^0(\Omega;\R^E) \to \R$ that is in the dual space $C^0(\Omega;\R^E)^{\ast}$.
Denote this linear functional by $\de_{p,j} : C^0(\Omega;\R^E) \to \R$.
Therefore by choosing $X := C^0(\Omega;\R^E)$, $\n := N$, and 
$\Sigma := \left\{ \de_{p,j} : p \in \Omega ~\text{and}~ j \in \{1, \ldots , E\} \right\} \subset X^{\ast}$,
we see that this problem is within our framework.
Here we are also using the observation that if $f,h \in C^0(\Omega;\R^E)$ satisfy, for every 
$p \in \Omega$ and every $j \in \{1, \ldots , E\}$, that $|\de_{p,j}[f-h]| \leq \ep$, then 
we have $||f-h||_{C^0(\Omega;\R^E)} \leq C \ep$ for a constant $C \geq 1$ depending  
on the particular norm chosen on $\R^E$.
Thus finding an approximation $u$ of $F$ that satisfies, for every $\sigma \in \Sigma$, that 
$|\sigma(F-u)| \leq \ep/C$ allows us to conclude that $||F-u||_{C^0(\Omega;\R^E)} \leq \ep$.

The \emph{cubature problem} \cite{Str71} for empirical measures,
which may be combined with sampling to offer an approach to 
the \emph{cubature problem} for general signed measures,
is within this framework. 
More precisely, let $N \in \Z_{\geq 1}$, $a_1 , \ldots , a_N > 0$, 
$\Omega = \{x_1 , \ldots , x_N\} \subset \R^d$, 
and $\m[\Omega]$ denote the collection of finite 
signed measures on $\Omega$. 
Recall that $\m[\Omega]$ can be viewed as a subset of 
$C^0(\Omega)^{\ast}$ by defining, for $\nu \in \m[\Omega]$
and $\psi \in C^0(\Omega)$, 
$\nu[\psi] := \int_{\Omega} \psi(x) d\nu(x)$.
Consider the empirical measure
$\mu := \sum_{i=1}^N a_i \de_{x_i}$ and, for 
$e = (e_1 , \ldots , e_d) \in \Z_{\geq 0}^d$, 
define $p_{e} : \R^d \to \R$ by 
$p_{e}(x_1 , \ldots , x_d) := x_1^{e_1} \ldots x_d^{e_d}$.
Then the choices that $X := \m[\Omega]$, 
$\n := N$, and, for a given $K \in \Z_{\geq 0}$, 
$\Sigma := \left\{ p_e : e = (e_1 , \ldots , e_d) \text{ with } 
e_1 + \ldots + e_d \leq K \right\} 
\subset C^{0}(\Omega) \subset \m[\Omega]^{\ast}$ 
illustrate that the \emph{cubature problem} of 
reducing the support of $\mu$
whilst preserving its moments of order no greater than $K$ 
is within our framework. 

Moreover the \emph{Kernel Quadrature} problem for empirical
probability measures is within our framework.
To be more precise, let $\cX$ be a finite set and $\cH_{k}$ is 
a \emph{Reproducing Kernel Hilbert Space} (RKHS) associated 
to a positive semi-definite symmetric kernel function
$k : \cX \times \cX \to \R$
(appropriate definitions can be found, for example, in 
\cite{BT11}).
In this case $\cH_k = \Span \left( 
\left\{ k_x : x \in \cX \right\} \right)$ 
where, for $x \in \cX$, the function $k_x : \cX \to \R$ 
is defined by $k_x(z) := k(x,z)$ (see, for example, 
\cite{BT11}). 

In this context one can consider the \emph{Kernel Quadrature}
problem, for which the Kernel Herding 
\cite{Wel09a,Wel09b,CSW10,BLL15,BCGMO18,TT21,PTT22}, 
Convex Kernel Quadrature \cite{HLO21} and Kernel Thinning 
\cite{DM21a,DM21b,DMS21} methods have been developed.
Given a probability measure $\mu \in \bbP[\cX]$ the 
\emph{Kernel Quadrature} problem 
involves finding, for some $n \in \Z_{\geq 1}$, 
points $z_1 , \ldots , z_n \in \cX$ and weights 
$w_1 , \ldots , w_n \in \R$ and such that the measure 
$\mu_n := \sum_{j=1}^n w_j \de_{z_j}$ approximates 
$\mu$ in the sense that, for every $f \in \cH_k$, we have
$\mu_n(f) \approx \mu(f)$. 
Additionally requiring the weights $w_1 , \ldots , w_n$
to be \emph{convex} in the sense that they are all positive 
and sum to one 
(i.e. $w_1 , \ldots , w_n > 0$ and $w_1 + \ldots + w_n = 1$)
ensures both better robustness properties for the 
approximation $\mu_n \approx \mu$ and better estimates 
when the $m$-fold product of quadrature formulas is 
used to approximate $\mu^{\otimes m}$ on $\cX^{\otimes m}$; 
see \cite{HLO21}.

A consequence of 
$\cH_k = \Span \left( \left\{ k_x : x \in \cX \right\} \right)$ 
is that
linearity ensures that this approximate equality will be 
valid for all $f \in \cH_k$ provided it is true for every
$f \in \{ k_x : x \in \cX \}$.
The inclusion
$\cH_k \subset C^0(\cX)$ means that the subset 
$\{ k_x : x \in \cX \} \subset \cH_k$ can be viewed as a 
finite subset of the dual space $\m[\cX]^{\ast}$. 
The choice of $X := \m[\cX]$ and
$\Sigma := \{ k_x : x \in \cX \}$ illustrates that the 
\emph{Kernel Quadrature} problem for empirical probability
distributions $\mu \in \bbP[\cX]$ is within the framework 
we consider.

The framework we consider is the same as the setup for which 
GEIM was developed by Maday et al. \cite{MM13,MMT14,MMPY15}. 
It is useful to briefly recall this method.
\vskip 4pt
\noindent
\textbf{GEIM}

\begin{enumerate}[label=(\Alph*)]
    \item\label{GEIM_step_1} 
        \begin{itemize}
            \item Find $h_1 := \argmax \left\{ 
            ||f||_X : f \in \cf \right\}$ and 
            $\sigma_1 := \argmax \left\{ 
            | \sigma(h_1) | : \sigma \in \Sigma \right\}$.
            \item Define 
            $q_1 := h_1 / \sigma_1(h_1)$, 
            $S_1 := \{ q_1 \}$, 
            $L_1 := \{\sigma_1 \}$, and an operator
            $\cj_1 : \Span(\cf) \to \Span(S_1)$ by setting
            $\cj_1[w] := \sigma_1(w)q_1$ for every 
            $w \in \Span(\cf)$.
            \item Observe that, given any $w \in \Span(\cf)$, 
            we have $\sigma_1 ( w - \cj_1[w]) = 0$.
        \end{itemize}
    \item\label{GEIM_step_2} 
    Proceed recursively via the follwing inductive 
    step for $n \geq 2$.
        \begin{itemize}
            \item Suppose we have found subsets
            $S_{n-1} = \{ q_1 , \ldots , q_{n-1} \} 
            \subset \Span(\cf)$ and 
            $L_{n-1} = \{ \sigma_1 , \ldots , \sigma_{n-1} \} 
            \subset \Sigma$, and an operator 
            $\cj_{n-1} : \Span(\cf) \to \Span(S_{n-1})$
            satisfying, for every $w \in \Span(\cf)$, that
            $\sigma(w - \cj_{n-1}[w]) = 0$
            for every $\sigma \in L_{n-1}$.
            \item Find 
            $h_n := \argmax \left\{ ||f - \cj_{n-1}[f]||_X 
            : f \in \cf \right\}$ and 
            $\sigma_n := \argmax \left\{ 
            | \sigma(h_n - \cj_{n-1}[h_n]) | : 
            \sigma \in \Sigma \right\}$.
            \item Define
            $$ q_n := \frac{h_n - \cj_{n-1}[h_n]}
            {\sigma_n(h_n - \cj_{n-1}[h_n])}, $$
            $S_n := S_{n-1} \cup \{q_n\} \subset \Span(\cf)$ 
            and 
            $L_n := L_{n-1} \cup \{\sigma_n\} \subset \Sigma$.
            \item Construct operator 
            $\cj_n : \Span(\cf) \to \Span(S_n)$ by defining
            $$ \cj_n[w] := \cj_{n-1}[w] + 
            \sigma_n( w - \cj_{n-1}[w])q_n $$
            for $w \in \Span(\cf)$.
            Direct computation verifies that, 
            whenever $w \in \Span(\cf)$ and 
            $\sigma \in L_{n}$, we have 
            $\sigma(w - \cj_n[w]) = 0$.
        \end{itemize}
\end{enumerate}
\noindent 
The algorithm provides a sequence 
$\cj_1[\vph]$, $\cj_2[\vph], \ldots$ of approximations of $\vph$.
However, GEIM is intended to control the $X$-norm of the 
difference between 
$\vph$ and its approximation, i.e. to have 
$||\vph - \cj_n[\vph]||_X$ be small for a suitable integer $n$.
This aim requires additional assumptions to be made 
regarding the data $\Sigma \subset X^{\ast}$ which we 
do \emph{not} impose (see \cite{MMT14} for details).
Recall that we aim only to approximate $\vph$ over 
the data $\Sigma$,
i.e. we seek an approximation $u$ such that 
$|\sigma(\vph-u)|$ is small for every $\sigma \in \Sigma$. 
A consequence of this difference is that 
certain aspects of GEIM are 
not necessarily ideal for our task.

Firstly, the growth of the subset $L_{n-1}$ to $L_n$ is 
\emph{feature-driven}. That is, a new feature from $\cf$ 
to be used by the next approximation is selected, 
and then this new feature is used to 
determine the new functional from $\Sigma$ to be added
to the collection on which we require the approximation 
to coincide with the target $\vph$
(cf. \textbf{GEIM} \ref{GEIM_step_2}).
Since we only seek to approximate $\vph$ over the 
data $\Sigma$, \emph{data-driven} growth would be 
preferable. That is, we would rather select the new 
information from $\Sigma$ to be matched by the next 
approximation \emph{before} any consideration is 
given to determining which features from $\cf$ will 
be used to construct the new approximation.

Secondly, related to the first aspect, the features 
from $\cf$ to be used to construct $\cj_n[\vph]$ are
predetermined. Further, the features
used to construct $\cj_n[\vph]$ are forced to be included
in the features used
to construct $\cj_m[\vph]$ for any $m > n$.
This restriction is \emph{not} guaranteed to be sensible; 
it is conceivable that the features that work well at one 
stage are disjoint from the features that work well at 
another.
We would prefer that the features used to construct an 
approximation be determined by the target $\vph$ and the 
data selected from $\Sigma$ on which we require the 
approximation to match $\vph$.
This would avoid retaining features used at one 
step that become less effective at later steps, and could 
offer insight regarding which of the features in $\cf$ are
sufficient 
to capture the behaviour of $\vph$ on $\Sigma$.

Thirdly, at each step GEIM can provide an approximation for
\emph{any} $w \in \Span(\cf)$. That is, for $n \in \Z_{\geq 1}$,
the element $\cj_n[w]$ is a linear combination
of the elements in $S_n$ that coincides with $w$ on $L_n$.
Requiring the operator $\cj_n : \Span(\cf) \to \Span(S_n)$
to provide such an approximation for every $w \in \Span(\cf)$ 
stops the method from exploiting any advantages
available for the particular choice of $\vph \in \Span(\cf)$.
As we only aim to approximate $\vph$ itself, we would prefer
to remove this requirement and allow the method the 
opportunity to exploit advantages resulting from the 
specific choice of $\vph \in \Span(\cf)$.

All three aspects are addressed in GRIM.
The greedy growth of a subset $L \subset \Sigma$, 
determining the functionals in $\Sigma$ at which an 
approximation is required to agree with $\vph$ 
is data-driven.
At each step the desired approximation is found using 
recombination \cite{LL12,Tch15}
so that the features used to construct the 
approximation are determined by $\vph$ and the subset 
$L$ and, in particular, are \emph{not} predetermined.
This use of recombination ensures both that GRIM only 
produces approximations of $\vph$, and that GRIM can 
exploit advantages resulting from the specific choice of 
$\vph \in \Span(\cf)$.

\section{Recombination Thinning}
\label{recomb_thin}
In this section we illustrate how, 
given a finite collection $L \subset X^{\ast}$ of 
linear functionals, recombination \cite{LL12,Tch15} can be 
used to find an element
$u \in \Span(\cf) \subset X$ that coincides with 
$\vph$ throughout $L$, provided we can compute 
the values $\sigma(f_i)$ for every 
$i \in \{1, \ldots , \n\}$ and every linear functional 
$\sigma \in L$.
The recombination algorithm was initially introduced by 
Christian Litterer and the first author in 
\cite{LL12}; a substantially improved implementation was 
provided in the PhD thesis of Maria Tchernychova \cite{Tch15}.
A novel implementation of the method from \cite{Tch15} 
was provided in \cite{ACO20}.
A method with the same complexity as the original 
implementation in \cite{LL12} was introduced in 
the works \cite{FJM19,FJM22}.
Recombination has been applied in a number of contexts 
including particle filtering \cite{LL16}, 
kernel quadrature \cite{HLO21}, and mathematical 
finance \cite{NS21}.

For the readers convenience we briefly overview the ideas 
involved in the recombination algorithm. For this illustrative 
purpose consider a linear system of equations 
$\bA \bx = \by$ where $\bA \in \R^{m \times k}$, 
$\bx \in \R^m$, $\by \in \R^k$, and $k,m \in \Z_{\geq 1}$ 
with $m \geq k$.
We assume, for every $i \in \{1 , \ldots , k\}$, that $\bx_i > 0$.
Recombination relies on the simple observation that this linear
system is preserved under translations of $\bx$ 
by elements in the kernel of the matrix $\bA$.
To be more precise, if $\bx$ satisfies that $\bA \bx = \by$, 
and if $\bolde \in \ker(\bA)$ 
and $\theta \in \R$, then $\bx + \theta \bolde$ also satisfies
that $\bA \left( \bx + \theta \bolde \right) = \by$. 
The recombination algorithm makes use of linearly independent
elements in the $\ker(\bA)$ to reduce the number of 
non-zero entries in the solution vector $\bx$.

As outlined in \cite{LL12}, this could in principle be done
as follows. First, 
we choose a basis for the kernel of $\bA$. Computing the
\emph{Singular Value Decomposition} (SVD) of $\bA$ gives a method
of finding such a basis that is well-suited to dealing with 
numerical instabilities. 
Supposing that $\ker(\bA) \neq \{0\}$, 
let $\bolde_1 , \ldots , \bolde_l$ be a basis of $\ker(\bA)$
found via SVD.
For each $j \in \{1 , \ldots , l\}$ denote the coefficients of 
$\bolde_j$ by 
$\bolde_{j,1} , \ldots , \bolde_{j,m} \in \R$.

Consider the element $\bolde_1 \in \R^m$.
Choose $i \in \{1 , \ldots , m\}$ such that
\beq
    \label{recomb_scale_fac}
        \frac{\bx_i}{\bolde_{1,i}} 
        =
        \min \left\{ \frac{\bx_j}{\bolde_{1,j}} :
        \bolde_{1,j} > 0 \right\}.
\eeq
Replace $\bx$ by the vector
$\bx - \left( \bx_i / \bolde_{1,i} \right) \bolde_1$.
The new $\bx$ remains a solution to $\bA \bx = \by$, and now 
additionally satisfies that $\bx_i = 0$.
Moreover, for every $j \in \{1, \ldots , m\}$ such that
$j \neq i$, our choice of $i$ in \eqref{recomb_scale_fac} 
ensures that $\bx_j \geq 0$.
Finally, for $j \in \{2 , \ldots , l\}$, replace
$\bolde_j$ by $\bolde_j - ( \bolde_{j,i}/\bolde_{1,i} ) \bolde_1$
to ensure that $\bolde_{j,i} = 0$.
This final alteration allows us to repeat the process for the new 
$\bx$ using the new $\bolde_2$ in place of $\bolde_1$ since 
the addition of any scalar multiple of $\bolde_2$ to $\bx$ 
will \emph{not} change the fact that $\bx_i =0$.

After iteratively repeating this procedure for $j=2, \ldots , l$,
we obtain a vector $\bx' \in \R^m$ whose coefficients are all 
non-negative, with at most $m-l$ of the coefficients being 
non-zero, and still satisfying that $\bA \bx' = \by$.
That is, the original solution $\bx$ has been reduced to a new
solution $\bx'$ with at least $l$ fewer non-zero coefficients 
than the original solution $\bx$.

Observe that the positivity of the original coefficients
is weakly preserved. That is, if we let 
$x'_1 , \ldots , x'_m \in \R$ denote the coefficients of the
vector $\bx' \in \R^m$, 
then for each $i \in \{1,\ldots,m\}$
we have that $x'_i \geq 0$.
One consequence of this preservation is that recombination
can be used to reduce the support of an empirical measure
whilst preserving a given finite collection of moments
\cite{LL12}. Moreover, this property is essential to 
the ground-breaking state-of-the-art
\emph{convex kernel quadrature} method developed by 
Satoshi Hayakawa, the first author, and Harald Oberhauser 
in \cite{HLO21}.

The implementation of recombination proposed in 
\cite{LL12} iteratively makes use of the method outlined above,
for $m = 2k$, applied to linear systems arising via sub-dividing 
the original system. 
The improved \emph{tree-based} method developed by 
Maria Tchernychova in \cite{Tch15} provides a significantly 
more efficient implementation.
A \emph{geometrically greedy algorithm} is proposed in 
\cite{ACO20} to implement the algorithm of \cite{Tch15}.
In the setting of our example above, it follows from 
\cite{Tch15} that the complexity of the recombination algorithm 
is $O( mk + k^3 \log ( m/k ) )$.

Having outlined the recombination method, we turn our attention 
to establishing the following result regarding the use of 
recombination in our setting.

\begin{lemma}[\textbf{Recombination Thinning}]
\label{Banach_recombination_lemma}
Assume $X$ is a Banach space with dual space 
$X^{\ast}$, that $\n \in \Z_{\geq 1}$, and that 
$m \in \Z_{\geq 0}$.
Define $\m := \min \{ \n , m + 1 \}$.
Let $\cf := \{ f_1 , \ldots , f_{\n} \} \subset X$
be a collection of non-zero elements and 
$L = \{ \sigma_1 , \ldots , \sigma_m \} \subset X^{\ast}$. 
Suppose $a_1 , \ldots , a_{\n} > 0$ and consider the 
element $\vph \in \Span(\cf) \subset X$ defined
by $\vph := \sum_{i=1}^{\n} a_i f_i$.
Then recombination 
can be applied to find non-negative coefficients
$b_1 , \ldots , b_{\m} \geq 0$ and indices
$e(1) , \ldots , e(\m) \in \{1 , \ldots , \n \}$
satisfying that 
\beq
    \label{Banach_recomb_coeff_sum}
        \sum_{j=1}^{\m} b_j 
        =
        \sum_{i=1}^{\n} a_i,
\eeq
and such that the element $u \in \Span(\cf) \subset X$
defined by
\beq
    \label{Banach_recomb_approx}
        u := \sum_{j=1}^{\m} b_j f_{e(j)}
        \quad \text{satisfies, for every }
        \sigma \in L, \text{ that} \quad
        \sigma(\vph - u) = 0.
\eeq
Finally, the complexity of this use of recombination is
$O( \n m + m^3 \log (\n/m))$. 
\end{lemma}

\begin{proof}[Proof of Lemma
\ref{Banach_recombination_lemma}]
Assume $X$ is a Banach space with dual space 
$X^{\ast}$, that $\n \in \Z_{\geq 1}$, and that 
$m \in \Z_{\geq 0}$.
Define $\m := \min \{ \n , m + 1 \}$.
Let $\cf := \{ f_1 , \ldots , f_{\n} \} \subset X$
be a collection of non-zero elements and 
$L = \{ \sigma_1 , \ldots , \sigma_m \} \subset X^{\ast}$. 
Suppose $a_1 , \ldots , a_{\n} > 0$ and Consider the 
element $\vph \in \Span(\cf) \subset X$ defined
by $\vph := \sum_{i=1}^{\n} a_i f_i$.

The values 
$\sigma_1 (\varphi) , \ldots , \sigma_m(\varphi)$
and the sum of the coefficients 
$\sum_{i=1}^{\n} a_i$ give rise to
the linear system of equations 
\beq
	\label{Banach_recomb_lin_sys}
		\begin{pmatrix}
		    1 & 1 & \dots & 1 \\
			\sigma_1 (f_1) & \sigma_1 (f_2) 
			& \dots & \sigma_1 (f_{\n}) \\
			\vdots & \vdots & \ddots & \vdots \\
			\sigma_{m} (f_1) & \sigma_{m} (f_2) 
			& \dots & \sigma_{m} (f_{\n})
		\end{pmatrix}
		\begin{pmatrix} 
			a_1 \\ a_2 \\ \vdots \\ \vdots  \\ a_{\n} 
		\end{pmatrix}
		=
		\begin{pmatrix}
		    \sum_{i=1}^{\n} a_i \\
			\sigma_1(\vph)  \\ \vdots \\ 
			\sigma_{m} (\vph) 
		\end{pmatrix}
\eeq
which we denote more succinctly by $\bA \bx = \by$.

Since the coefficients $a_1 , \ldots , a_{\n} > 0$ are 
positive, we are able to apply recombination 
\cite{LL12,Tch15} to this linear system.
Combined with the observation that 
$\dim \left( \ker (\bA) \right) \geq \n - \m$, 
an application of recombination returns an element
$\bx' = \left( \bx'_1 , \ldots , \bx'_{\n} \right) \in \R^{\n}$ 
satisfying the following properties. 
Firstly, for each $j \in \{1, \ldots , \n\}$ the coefficient 
$\bx'_j \geq 0$ is non-negative.
Secondly, there are at most $\m$ indices $i \in \{1,\ldots ,\n\}$
for which $\bx'_i > 0$.

Let $D := \dim \left( \ker (\bA) \right) \geq \n - \m$.
Take 
$e(1) , \ldots , e(\n-D) \in \{1, \ldots , \n\}$ 
to be the indices 
$i \in \{1,\ldots ,\n\}$ for which $\bx'_i > 0$.
Then, for each $j \in \{1 , \ldots , \n - D\}$, we set 
$b_j := \bx'_{e(j)} > 0$.
Define an element $u \in \Span(\cf)$ by
$u:= \sum_{j=1}^{\n - D} b_j f_{e(j)}$ 
(cf. \eqref{Banach_recomb_approx}).
Recall that recombination ensures that $\bx'$ is
a solution to the linear system \eqref{Banach_recomb_lin_sys}.
Hence the equation corresponding to the top row of the matrix
$\bA$ tells us that
\beq
    \label{Banach_recomb_D_coeff_sum}
        \sum_{j=1}^{\n - D} b_j 
        = 
        \sum_{j=1}^{\n} \bx'_j 
        = 
        \sum_{j=1}^{\n} a_j.
\eeq
Moreover, given any $l \in \{1, \ldots , m\}$, the 
equation corresponding to row $l+1$ of the matrix $\bA$
tells us that
\beq    
    \label{Banach_recomb_D_lin_funcs_equal}
        \sigma(u) =
        \sum_{j=1}^{\n - D} b_j \sigma_l 
        \left( f_{e(j)} \right) 
        = 
        \sum_{j=1}^{\n} \bx'_j \sigma_l 
        \left( f_{e(j)} \right) 
        =
        \sigma_l(\vph) 
\eeq
If $D = \n - \m$ then \eqref{Banach_recomb_D_coeff_sum} and 
\eqref{Banach_recomb_D_lin_funcs_equal} 
are precisely the equalities claimed in 
\eqref{Banach_recomb_coeff_sum} and 
\eqref{Banach_recomb_approx}. 
If $D > \n - \m$ then $\n - D < \m$. Set
$b_{\n - D + 1} = \ldots = b_{\m} = 0$ and choose
\emph{any} indices 
$e(\n-D+1) , \ldots , e(\m) \in \{1, \ldots , \n\}$.
Evidently we have that 
$\sum_{j=1}^{\m} b_j = \sum_{j=1}^{\n - D} b_j$ and, 
for each $l \in \{1, \ldots , m\}$, that 
$\sum_{j=1}^{\m} b_j \sigma_l \left( f_{e(j)} \right) 
= \sum_{j=1}^{\n - D} b_j \sigma_l \left( f_{e(j)} \right) $.
Consequently, \eqref{Banach_recomb_D_coeff_sum} and 
\eqref{Banach_recomb_D_lin_funcs_equal} 
are once again precisely the equalities claimed in 
\eqref{Banach_recomb_coeff_sum} and 
\eqref{Banach_recomb_approx}.

It remains only to verify the claimed complexity of this 
application of recombination.
For this purpose, we observe that recombination is applied
to a linear system of $m+1$ equations in $\n$ variables. 
Hence from \cite{Tch15} we have that the complexity is 
$O \left( \n(m+1) + (m+1)^3 \log( \n / m+1) \right) 
= O \left( \n m + m^3 \log( \n / m ) \right)$ as claimed.
This completes the proof of Lemma 
\ref{Banach_recombination_lemma}.
\end{proof}

\section{The Banach GRIM Algorithm}
\label{Banach_GRIM_alg}
In this section we detail the \textbf{Banach GRIM} algorithm. 
The following \textbf{Banach Extension Step} is used
to grow a collection of linear functionals from $\Sigma$ 
at which we require our next approximation of $\vph$ to 
agree with $\vph$.
\vskip 4pt
\noindent
\textbf{Banach Extension Step}

\noindent
Assume $L' \subset \Sigma$. Let $u \in \Span(\cf)$.
Let $m \in \Z_{\geq 1}$ such that 
$\# \left( L' \right) + m \leq \Lambda := 
\# \left( \Sigma \right)$.
Take
\beq
    \label{banach_ext_step_one}
        \sigma_1 := \argmax \left\{ 
        |\sigma(\varphi-u)| ~:~ 
        \sigma \in \Sigma \right\}.
\eeq
Inductively for $j=2,3,\ldots,m$ take
\beq
    \label{banach_ext_step_j}
        \sigma_j := \argmax \left\{
        |\sigma(\varphi - u)| ~:~
        \sigma \in \Sigma \setminus 
        \{ \sigma_1 , \ldots , \sigma_{j-1}\}
        \right\}. 
\eeq
Once $\sigma_1 , \ldots , \sigma_m \in \Sigma$
have been defined, we extend $L'$ to 
$L := L' \cup \{ \sigma_1 , \ldots , \sigma_m\}$.
\vskip 4pt
\noindent
For each choice of subset $L \subset \Sigma$, we
use recombination
(cf. Lemma \ref{Banach_recombination_lemma})
to find $u \in \Span(\cf)$ 
agreeing with $\vph$ throughout $L$. 
Theoretically, Lemma \ref{Banach_recombination_lemma}
verifies that recombination can be used to 
find an approximation $u$ of $\varphi$ for 
which $\sigma(\varphi - u) = 0$ for all linear
functionals $\sigma \in L$ for a given
subset $L\subset \Sigma \subset X^{\ast}$.
However, implementations of recombination
inevitably result in numerical errors. 
That is, the returned coefficients will only solve the 
equations modulo some (ideally) small error
term. To account for this in our analysis, 
whenever we apply Lemma 
\ref{Banach_recombination_lemma} we will only 
assume that the resulting approximation
$u \in \Span(\cf)$ is \emph{close} to $\varphi$ at 
each functional $\sigma \in L$. That is, for
each $\sigma \in L$, we have that
$|\sigma(\varphi - u)| \leq \ep_0$ for some
(small) constant $\ep_0 \geq 0$.

Recall (cf. Section \ref{recomb_thin}) that if recombination
is applied to a linear system corresponding to a matrix $A$,
then a \emph{Singular Value Decomposition} SVD of the matrix 
$A$ is used to find a basis for $\ker(A)$.
Re-ordering the rows of the matrix $A$ (i.e. changing the 
order in which the equations are considered) can potentially
result in a different basis for $\ker(A)$ being selected.
Thus shuffling the order of the equations can affect the 
approximation returned by recombination via the 
\emph{recombination thinning} Lemma
\ref{Banach_recombination_lemma}.
We exploit this by optimising the approximation returned
by recombination over a chosen number of \emph{shuffles} of 
the equations forming the linear system. 
This is made precise in the following 
\textbf{Banach Recombination Step} detailing how, for a given 
subset $L \subset \Sigma$, we use 
recombination to find an element $u \in \Span(\cf)$ that is 
within $\ep_0$ of $\vph$ throughout $\Sigma$.
\vskip 4pt
\noindent
\textbf{Banach Recombination Step}

\noindent
Assume $L \subset \Sigma$. Let $s \in \Z_{\geq 1}$. 
For each $j \in \{1, \ldots , s\}$ we do the following. 
\begin{enumerate}[label=(\Alph*)]
    \item\label{banach_recomb_step_A} 
    Let $L_j \subset \Sigma$ be the subset resulting
    from applying a random permutation to the ordering of 
    the elements in $L$.
    \item\label{banach_recomb_step_B}  
    Apply recombination 
    (cf. the \emph{recombination thinning} Lemma 
    \ref{Banach_recombination_lemma}) to find 
    an element $u_j \in \Span(\cf)$ satisfying, for 
    every $\sigma \in L_j$, that 
    $|\sigma(\vph-u_j)| \leq \ep_0$.
    \item\label{banach_recomb_step_C}  
    Compute $E[u_j] := 
    \max \left\{ |\sigma(\vph-u_j)| 
    : \sigma \in \Sigma \right\}$.
\end{enumerate}
After obtaining the elements $u_1 , \ldots , u_s \in \Span(\cf)$,
define $u \in \Span(\cf)$ by 
\beq
    \label{banach_recomb_step_output}
        u := \argmin \left\{ 
        E[w] : w \in \{u_1 , \ldots , u_s\} \right\}.
\eeq
Then $u \in \Span(\cf)$ is returned as our approximation of 
$\vph$ that satisfies, for every $\sigma \in \Sigma$, that 
$|\sigma(\vph-u)| \leq \ep_0$.
\vskip 4pt
\noindent
We now detail our proposed GRIM 
algorithm to find an approximation $u \in \Span(\cf)$
of $\vph \in \Span(\cf)$ that is close to $\varphi$ 
at every linear functional in $\Sigma$.
\vskip 4pt
\noindent
\textbf{Banach GRIM}
\begin{enumerate}[label=(\Alph*)]
    \item\label{Banach_GRIM_i} 
    Fix $\ep > 0$ as the \emph{target accuracy
    threshold}, $\ep_0 \in [0,\ep)$
    as the \emph{acceptable recombination error}, 
    and $M \in \Z_{\geq 1}$ as the 
    \emph{maximum number of steps}.
    Choose integers 
    $s_1 , \ldots , s_M \in \Z_{\geq 1}$ as the 
    \emph{shuffle numbers}, and 
    integers $k_1 , \ldots , k_M \in \Z_{\geq 1}$
    with 
    \beq
        \label{kappa_constraint}
            \kappa := k_1 + \ldots + k_M \leq 
            \min \left\{ \n - 1 , \Lambda \right\}.
    \eeq
    \item\label{Banach_GRIM_ii} 
    For each $i \in \{1, \ldots , \n\}$, 
    if $a_i < 0$ then replace $a_i$ and 
    $f_i$ by $-a_i$ and $-f_i$ respectively.
    This ensures that $a_1 , \ldots , a_{\n} > 0$
    whilst leaving the expansion 
    $\vph = \sum_{i=1}^{\n} a_i f_i$ unaltered.
    Additionally, rescale each $f_i$ to have unit 
    $X$ norm.
    That is, for each $i \in \N$ we replace 
    $f_i$ by $h_i := \frac{f_i}{||f_i||_X}$.
    Then $\vph = \sum_{i=1}^{\n} \al_i h_i$ 
    where, for each $i \in \{1 , \ldots , \n \}$, 
    we have $\al_i := a_i || f_i ||_X > 0$.
    \item\label{Banach_GRIM_iii} 
    Apply the 
    \textbf{Banach Extension Step}, with 
    $L' := \varnothing$, $u := 0$ and 
    $m := k_1$, to obtain a subset 
    $\Sigma_1 = \{ \sigma_{1,1} , \ldots , 
    \sigma_{1,k_1} \} \subset \Sigma$.
    Apply the \textbf{Banach Recombination Step}, 
    with $L := \Sigma_1$ and $s:=s_1$, to find 
    an element $u_1 \in \Span(\cf)$ satisfying, 
    for every $\sigma \in \Sigma_1$, that
    $|\sigma(\vph - u)| \leq \ep_0$.
    
    If $M = 1$ then the algorithm terminates here
    and returns $u_1$ as the final approximation 
    of $\vph$
    \item\label{Banach_GRIM_iv} 
    If $M \geq 2$ then we proceed inductively 
    for $t \geq 2$ as follows.
    If $|\sigma(\vph - u_{t-1}) | \leq \ep$
    for every $\sigma \in \Sigma$
    then we stop and $u_{t-1}$ is an 
    approximation of $\varphi$ possessing the 
    desired level of accuracy.
    Otherwise, we choose $k_t \in \Z_{\geq 1}$
    and apply the \textbf{Banach Extension Step},
    with $L' = \Sigma_t$, $u := u_{t-1}$ and 
    $m := k_t$, to obtain a subset 
    $\Sigma_t := \Sigma_{t-1} \cup 
    \left\{ \sigma_{t,1} , \ldots ,\sigma_{t,k_t} 
    \right\} \subset \Sigma$. Apply the 
    \textbf{Banach Recombination Step}, with $L := \Sigma_t$
    and $s := s_M$, to find an element $u \in \Span(\cf)$
    satisfying, for every $\sigma \in \Sigma_t$,
    that $|\sigma(\vph - u)| \leq \ep_0$.
    
    The algorithm ends either by returning 
    $u_{t-1}$ for the $t \in \{2, \ldots , M\}$ 
    for which the stopping criterion was triggered
    as the final approximation of $\vph$,
    or by returning $u_M$ as the final approximation of 
    $\vph$ if the stopping criterion 
    is never triggered.
\end{enumerate}
\vskip 2mm
\noindent
If the algorithm terminates as a result of one of the 
early stopping criterion being triggered then we are 
guaranteed that the returned approximation $u$ satisfies, 
for every $\sigma \in \Sigma$, that 
$| \sigma( \vph - u) | \leq \ep$. 
In Subsection \ref{conv_gen} we establish estimates for 
how large $M$ is required to be in order to guarantee 
that the algorithm returns an approximation of 
$\vph$ possessing this level of accuracy throughout
$\Sigma$ 
(cf. the \textbf{Banach GRIM Convergence Theorem}
\ref{banach_conv}).

GRIM uses \emph{data-driven} growth.
For each $m \in \{2 , \ldots , M\}$, GRIM first determines 
the new linear functionals in $\Sigma$ to be added to 
$\Sigma_{m-1}$ to form $\Sigma_m \subset \Sigma$ before
using recombination to find an approximation 
coinciding with $\vph$ on $\Sigma_m$. That is, 
we first choose the new information 
that we want our approximation to match \emph{before} 
using recombination to both select the elements from
$\cf$ and use them to construct our approximation.

Evidently we have the nesting property that
$\Sigma_{t_1} \subset \Sigma_{t_2}$ for integers
$t_1 \leq t_2$, ensuring that at each step we are 
increasing the amount of information that we 
require our approximation to match.
For each integer $m \in \{1, \ldots , M \}$ let 
$S_m \subset \cf$ denote the sub-collection of
elements from $\cf$ used to form the approximation
$u_m$.
Recombination is applied to a system of
$1+k_1 + \ldots + k_m$ linear equations when 
finding $u_m$, hence we may conclude that
$\#(S_m) \leq \min \left\{1 + k_1 + \ldots + k_m , \n \right\}$ 
(cf. the \emph{recombination thinning}
Lemma \ref{Banach_recombination_lemma}).
Besides this upper bound for $\#(S_m)$, 
we have \emph{no} control on the sets $S_m$.
We impose only that the linear functionals
are greedily chosen; the selection of the elements 
from $\cf$ to form the approximation $u_m$ is left 
up to recombination and determined by the data.
In contrast to GEIM, there is \emph{no} requirement that 
elements from $\cf$ used to form $u_m$ must also be 
used for $u_l$ for $l > m$.

The restriction on $\kappa := k_1 + \ldots + k_M $ 
in \eqref{kappa_constraint} is for the following reasons.
As observed above, for $m \in \{1, \ldots , M\}$, $k_s$ is 
the number of new linear functionals to be 
chosen at the $m^{\text{th}}$-step. 
Hence, at the $m^{\text{th}}$-step, recombination is 
used to find an approximation $u$ that is within 
$\ep_0$ of $\vph$ on a collection of 
$\kappa_s := k_1 + \ldots + k_m$ linear functionals from 
$\Sigma$. Evidently we have, for every $m \in \{1, \ldots , M\}$,
that $\kappa_s \leq \kappa$.

A first consequence of the restriction in 
\eqref{kappa_constraint} ensures is that, for every 
$m \in \{1, \ldots , M \}$, we have $\kappa_m \leq \n -1$.
Since the linear system that recombination is applied to at 
step $s$ consists of $1 + \kappa_m$ equations (cf. 
the \emph{recombination thinning} Lemma 
\ref{Banach_recombination_lemma}), this prevents 
the number of equations from exceeding $\n$.
When the number of equations is at least $\n$, recombination
returns the original coefficients without reducing the 
number that are non-zero (see Subsection \ref{recomb_thin}).
Consequently, $\kappa_m \geq \n - 1$ guarantees that step 
$s$ returns $\vph$ itself as the approximation of $\vph$. 
The restriction in \eqref{kappa_constraint} ensures that
the algorithm ends if this (essentially useless) stage is 
reached.

A second consequence of \eqref{kappa_constraint} is, 
for every $m \in \{1, \ldots , M \}$, that 
$\kappa_m \leq \Lambda$. Note the collection 
$\Sigma_m \subset \Lambda$ has cardinality 
$\kappa_m$.
If $\kappa_m = \Lambda$ then we necessarily have that 
$\Sigma_m = \Sigma$, and so recombination is used to 
find an approximation of $\vph$ that is within $\ep_0$
of $\vph$ at every $\sigma \in \Sigma$.
The restriction \eqref{kappa_constraint} ensures that
if this happens then the algorithm terminates without
carrying out additional steps.

\section{Complexity Cost}
\label{complexity_cost_subseq}
In this subsection we consider the complexity cost of the
\textbf{Banach GRIM} algorithm presented in Section 
\ref{Banach_GRIM_alg}. We begin by recording the 
complexity cost of the \textbf{Banach Extension Step}.

\begin{lemma}[\textbf{Banach Extension Step} Complexity Cost]
\label{banach_ext_step_cost_lemma}
Let $\n, \Lambda, m  , t \in \Z_{\geq 1}$ and $X$ be a 
Banach space with dual space $X^{\ast}$. 
Assume $\cf = \left\{ f_1 , \ldots , f_{\n} \right\} 
\subset X \setminus \{0\}$ 
and that $\Sigma \subset X^{\ast}$ has cardinality 
$\Lambda$.
Let $a_1 , \ldots , a_{\n} \in \R \setminus \{0\}$
and define $\vph := \sum_{i=1}^{\n} a_i f_i$.
Assume that the set $\{ \sigma(\vph) : \sigma \in \Sigma \} $
and, for every $i \in \{1, \ldots , \n\}$, the set 
$\{ \sigma(f_i) : \sigma \in \Sigma \}$ have already been
computed.
Then for any choices of $L' \subset \Sigma$ with 
$\# \left( L' \right) \leq \Lambda - m$ and any
$u \in \Span(\cf)$ with $\# \support(u) = t$, 
the complexity cost of applying the 
\textbf{Banach Extension Step}, with the $L'$, $u$ and 
$m$ there as the $L'$, $u$ and $m$ here respectively, is
$\cO \left( (m+t) \Lambda \right)$. 
\end{lemma}

\begin{proof}[Proof of Lemma \ref{banach_ext_step_cost_lemma}]
Let $\n, \Lambda, m  , t \in \Z_{\geq 1}$ and $X$ be a 
Banach space with dual space $X^{\ast}$. 
Suppose that $\cf = \left\{ f_1 , \ldots , f_{\n} \right\} 
\subset X \setminus \{0\}$ 
and that $\Sigma \subset X^{\ast}$ has cardinality 
$\Lambda$.
Let $a_1 , \ldots , a_{\n} \in \R \setminus \{0\}$
and define $\vph := \sum_{i=1}^{\n} a_i f_i$.
Assume that the set $\{ \sigma(\vph) : \sigma \in \Sigma \} $
and, for every $i \in \{1, \ldots , \n\}$, the set 
$\{ \sigma(f_i) : \sigma \in \Sigma \}$ have already been
computed.
Suppose that $L' \subset \Sigma$ with 
$\# \left( L' \right) \leq \Lambda - m$ and 
$u \in \Span(\cf)$ with $\# \support(u) = t$.
Recall our convention (cf. Section \ref{notation}) that 
$\support(u)$ is the set of the $f_i$ that correspond to 
the non-zero coefficients in the expansion of $u$ in 
terms of the $f_i$. That is, $u= \sum_{i=1}^{\n} u_i f_i$
for some coefficients $u_1 , \ldots , u_{\n} \in \R$, and 
\beq
    \label{support(u)_def}
        \support(u) := \left\{ f_j ~:~
        j \in \{ 1 , \ldots , \n \} \text{ and }
        u_j \neq 0 \right\}.
\eeq
Consider carrying out the the 
\textbf{Banach Extension Step} with the $L'$, $u$ and 
$m$ there as the $L'$, $u$ and $m$ here respectively.
Since we have access to 
$\{ \sigma(\vph) : \sigma \in \Sigma \} $ and 
$\{ \sigma(f_i) : \sigma \in \Sigma \}$ for every 
$i \in \{1 , \ldots , \n\}$ without additional computation, 
and since $\# \support(u) = t$, the complexity cost of 
computing the set 
$\left\{ |\sigma(\vph - u) | : \sigma \in \Sigma \right\}$
is no worse than 
$\cO \left( t \Lambda \right)$.
The complexity cost of subsequently extracting the top 
$m$ argmax values of this set is no worse than 
$\cO(m \Lambda )$. 
The complexity cost of appending the resulting $m$ 
linear functionals in $\Sigma$ to the collection $L'$
is $\cO(m)$. 
Therefore the entire \textbf{Banach Extension Step} 
has a complexity cost no worse than 
$\cO \left( (m+t)\Lambda \right)$ as claimed.
This completes the proof of Lemma 
\ref{banach_ext_step_cost_lemma}.
\end{proof}
\vskip 4pt
\noindent
We next record the complexity cost of the 
\textbf{Banach Recombination Step}.

\begin{lemma}
\label{Banach_recomb_step_cost_lemma}
Let $\n , \Lambda , m , s \in \Z_{\geq 1}$ and 
$X$ be a Banach space with dual space $X^{\ast}$. 
Assume $\cf = \left\{ f_1 , \ldots , f_{\n} \right\} 
\subset X \setminus \{0\}$ 
and that $\Sigma \subset X^{\ast}$ has cardinality 
$\Lambda$.
Let $a_1 , \ldots , a_{\n} \in \R \setminus \{0\}$
and define $\vph := \sum_{i=1}^{\n} a_i f_i$.
Assume that the set $\{ \sigma(\vph) : \sigma \in \Sigma \} $
and, for every $i \in \{1, \ldots , \n\}$, the set 
$\{ \sigma(f_i) : \sigma \in \Sigma \}$ have already been
computed. 
Then for any $L \subset \Sigma$ with cardinality $\#(L) = m$
the complexity cost of applying the 
\textbf{Banach Recombination Step}, with the subset $L$ and the 
integer $s$ there as $L$ and $s$ here respectively, is 
\beq
    \label{recomb_step_cost}
        \cO \left( ms( \n + \Lambda) + 
        m^3 s \log \left( \frac{\n}{m} \right) \right).
\eeq
\end{lemma}

\begin{proof}[Proof of Lemma \ref{Banach_recomb_step_cost_lemma}]
Let $\n , \Lambda , m , s \in \Z_{\geq 1}$ and 
$X$ be a Banach space with dual space $X^{\ast}$. 
Assume that $\cf = \left\{ f_1 , \ldots , f_{\n} \right\} 
\subset X \setminus \{0\}$ 
and that $\Sigma \subset X^{\ast}$ has cardinality 
$\Lambda$.
Let $a_1 , \ldots , a_{\n} \in \R \setminus \{0\}$
and define $\vph := \sum_{i=1}^{\n} a_i f_i$.
Assume that the set $\{ \sigma(\vph) : \sigma \in \Sigma \} $
and, for every $i \in \{1, \ldots , \n\}$, the set 
$\{ \sigma(f_i) : \sigma \in \Sigma \}$ have already been
computed.

Consider applying the 
\textbf{Banach Recombination Step} with the subset $L$ and the 
integer $s$ there as $L$ and $s$ here respectively.
Let $j \in \{1 , \ldots , s\}$.
The complexity cost of shuffling of the elements in 
$L$ to obtain $L_j$ is $\cO(s)$.
By appealing to the \emph{recombination thinning}
Lemma \ref{Banach_recombination_lemma}
we conclude that the complexity cost of applying recombination 
to find $u_j \in \Span(\cf)$ satisfying, for every 
$\sigma \in L_j$, that $|\sigma(\vph-u_j)| \leq \ep_0$ 
is $\cO \left( \n m + m^3 \log ( \n / m ) \right)$.
Further recall that, since $\#(L) = m$,  
the \emph{recombination thinning}
Lemma \ref{Banach_recombination_lemma}
ensures that $\# \support(u_j) \leq m + 1$. Thus, since
we already have access to
$\{ \sigma(\vph) : \sigma \in \Sigma \} $ and 
$\{ \sigma(f_i) : \sigma \in \Sigma \}$ for every 
$i \in \{1 , \ldots , \n\}$ without additional computation 
and we know from ,
the complexity cost of computing 
$E[u_j] := \max \left\{ 
|\sigma(\vph - u_j)| : \sigma \in \Sigma \right\}$
is $\cO(\Lambda m)$.

Therefore, the complexity cost of 
\textbf{Banach Recombination Step} \ref{banach_recomb_step_A}, 
\ref{banach_recomb_step_B}, and 
\ref{banach_recomb_step_C} is 
\beq
    \label{recomb_step_cost_v1}
        \cO \left( sm + s\Lambda m + 
        s \n m + m^3 s \log \left( \frac{\n}{m} \right) \right)
        =
        \cO \left( ms( \n + \Lambda) + 
        m^3 s \log \left( \frac{\n}{m} \right) \right).
\eeq
The complexity cost of the final selection of 
$u := \argmin \left\{ E[w] : w \in \{u_1 , \ldots , u_s\} \right\}$
is $\cO(s)$.
Combined with \eqref{recomb_step_cost_v1}, this yields 
that the complexity cost of the entire 
\textbf{Banach Recombination Step} is 
\beq
    \label{recomb_step_cost_v2}
        \cO \left( ms( \n + \Lambda) + 
        m^3 s \log \left( \frac{\n}{m} \right) \right).
\eeq
as claimed in \eqref{recomb_step_cost}.
This completes the proof of Lemma \ref{Banach_recomb_step_cost_lemma}.
\end{proof}
\vskip 2mm
\noindent
We now establish an upper bound for the complexity cost of
the \textbf{Banach GRIM} algorithm via repeated use of Lemmas 
\ref{banach_ext_step_cost_lemma} and 
\ref{Banach_recomb_step_cost_lemma}.
This is the content of the following result.

\begin{lemma}[\textbf{Banach GRIM Complexity Cost}]
\label{GRIM_cost_lemma}
Let $M , \n , \Lambda \in \Z_{\geq 1}$ and $\ep > \ep_0 \geq 0$. 
Take $s_1 , \ldots , s_M \in \Z_{\geq 1}$ and
$k_1 , \ldots , k_M \in \Z_{\geq 1}$ 
with $k_1 + \ldots + k_M \leq 
\min \left\{ \n - 1 , \Lambda \right\}$.
For $j \in \{1, \ldots , M \}$ let 
$\kappa_j := \sum_{i=1}^j k_i$.
Assume $X$ is a Banach space with dual-space $X^{\ast}$.
Suppose
$\cf := \{ f_1 , \ldots , f_{\n} \} \subset X \setminus \{0\}$ 
and that $\Sigma \subset X^{\ast}$ is finite with 
cardinality $\# \left( \Sigma \right) = \Lambda$.
Let $a_1 , \ldots , a_{\n} \in \R \setminus \{0\}$ and define 
$\vph \in \Span(\cf)$ by $\vph := \sum_{i=1}^{\n} a_i f_i$.
The complexity cost of completing $n$ steps of the 
\textbf{Banach GRIM} algorithm to approximate $\vph$, 
with $\ep$ as the target accuracy,
$\ep_0$ as the acceptable recombination error, 
$M$ as the maximum number of steps,
$s_1 , \ldots , s_M$ as the shuffle numbers, 
and the integers $k_1 , \ldots , k_M \in \Z_{\geq 1}$
as the integers $k_1 , \ldots , k_M$ chosen in Step 
\ref{Banach_GRIM_i} of the \textbf{Banach GRIM} algorithm, is 
\beq
    \label{GRIM_M_steps_cost}
        \cO \left( \n \Lambda + 
        \sum_{j=1}^M \kappa_j s_j 
        \left( \n + \Lambda \right)
        +
        \kappa_j^3 s_j
        \log \left( \frac{\n}{\kappa_j}
        \right)
        \right).
\eeq
\end{lemma}

\begin{proof}[Proof of Lemma \ref{GRIM_cost_lemma}]
Let $M , \n , \Lambda \in \Z_{\geq 1}$ and $\ep > \ep_0 \geq 0$. 
Take $s_1 , \ldots , s_M \in \Z_{\geq 1}$ and 
$k_1 , \ldots , k_M \in \Z_{\geq 1}$ 
with $k_1 + \ldots + k_M \leq 
\min \left\{ \n - 1 , \Lambda \right\}$.
For $j \in \{1, \ldots , M\}$ define 
$\kappa_j := \sum_{i=1}^j k_i$.
Assume $X$ is a Banach space with dual-space $X^{\ast}$.
Suppose
$\cf := \{ f_1 , \ldots , f_{\n} \} \subset X \setminus \{0\}$ 
and that $\Sigma \subset X^{\ast}$ is finite with 
cardinality $\# \left( \Sigma \right) = \Lambda$.
Let $a_1 , \ldots , a_{\n} \in \R \setminus \{0\}$ and define 
$\vph \in \Span(\cf)$ by $\vph := \sum_{i=1}^{\n} a_i f_i$.
Consider applying the 
\textbf{Banach GRIM} algorithm to approximate $\vph$, 
with $\ep$ as the target accuracy,
$\ep_0$ as the acceptable recombination error, 
$M$ as the maximum number of steps, 
$s_1 , \ldots , s_M$ as the shuffle numbers,
and the integers $k_1 , \ldots , k_M \in \Z_{\geq 1}$
as the integers $k_1 , \ldots , k_M$ chosen in Step 
\ref{Banach_GRIM_i} of the \textbf{Banach GRIM} algorithm.

Since the cardinality of $\cf$ is $\n$, the complexity cost of 
the rescaling and sign alterations required in Step 
\ref{Banach_GRIM_ii} of the \textbf{Banach GRIM} algorithm
is $\cO(\n)$.
The complexity cost of computing the sets 
$\left\{ \sigma(f_i) : \sigma \in \Sigma \right\}$
for $i \in \{ 1 , \ldots , \n \}$ is $\cO(\n \Lambda)$.
Subsequently, the complexity cost of computing the set
$\left\{ \sigma(\vph) : \sigma \in \Sigma \right\}$ is 
$\cO(\n \Lambda)$.
Consequently, the total complexity cost of performing these
computations is $\cO (\n \Lambda)$.

We appeal to Lemma \ref{banach_ext_step_cost_lemma}
to conclude that the complexity cost of performing the 
\textbf{Banach Extension Step} as in Step \ref{Banach_GRIM_iii}
of the \textbf{Banach GRIM} algorithm (i.e. with 
$L' := \emptyset$, $u := 0$, and $m := k_1$) is 
$\cO \left( k_1 \Lambda \right)$.
By appealing to Lemma \ref{Banach_recomb_step_cost_lemma}, 
we conclude that the complexity cost of the use of
the \textbf{Banach Recombination Step} in Step 
\ref{Banach_GRIM_iii}
of the \textbf{Banach GRIM} algorithm 
(i.e. with $L := \Sigma_1$ and $s:=s_1$) is 
$\cO \left( s_1k_1 \left( \n + \Lambda \right) + 
k_1^3 s_1 \log \left( \n / k_1 \right) \right)$.
Hence the complexity cost of 
Step \ref{Banach_GRIM_iii}
of the \textbf{Banach GRIM} algorithm is 
$\cO \left( s_1k_1 \left( \n + \Lambda \right) + 
k_1^3 s_1 \log \left( \n / k_1 \right) \right)$.

In the case that $M = 1$ we can already conclude that the 
total complexity cost of performing the 
\textbf{Banach GRIM} algorithm is 
$\cO \left( k_1 s_1 \left( \n + \Lambda \right) + \n \Lambda +
k_1^3 s_1 \log \left( \frac{\n}{k_1} \right) \right)$
as claimed in \eqref{GRIM_M_steps_cost}.
Now suppose that $M \geq 2$. We assume that all $M$ steps 
of the \textbf{Banach GRIM} algorithm are completed without
early termination since this is the case that will maximise
the complexity cost.
Under this assumption, for $j \in \{1, \ldots , M\}$ 
let $u_j \in \Span(\cf)$ denote the approximation of $\vph$ 
returned after step $j$ of the
\textbf{Banach GRIM} algorithm is completed.
Recall that $\kappa_j := \sum_{i=1}^j k_i$.
Examining the \textbf{Banach GRIM} algorithm, $u_j$ is 
obtained by applying recombination (cf. 
the \emph{recombination thinning} Lemma 
\ref{Banach_recombination_lemma}) to find an approximation 
that is within $\ep_0$ of $\vph$ on a subset of 
$\kappa_j$ linear functionals from $\Sigma$.
Thus, by appealing to the \emph{recombination thinning}
Lemma \ref{Banach_recombination_lemma},
we have that $\# \support(u_j) \leq 1 + \kappa_j$.
For the purpose of computing the maximal complexity cost 
we assume that $\# \support(u_j) = 1 + \kappa_j$.

Let $t \in \{2, \ldots , M\}$ and consider performing 
Step \ref{Banach_GRIM_iv} of the \textbf{Banach GRIM} 
algorithm for the $s$ there as $t$ here.
Since $\# \support ( u_{t-1}) = 1 + \kappa_{t-1}$,
Lemma \ref{banach_ext_step_cost_lemma} tells us that
the complexity cost of performing the 
\textbf{Banach Extension Step} as in Step \ref{Banach_GRIM_iv}
of the \textbf{Banach GRIM} algorithm (i.e. with 
$L' := \Sigma_{t-1}$, $u := u_{t-1}$, 
and $m := k_t$) is 
$\cO \left( \kappa_t \Lambda \right)$.
Lemma \ref{Banach_recomb_step_cost_lemma} yield that 
the complexity cost of the use of the
\textbf{Banach Recombination Step} in Step \ref{Banach_GRIM_iv}
of the \textbf{Banach GRIM} algorithm 
(i.e. with $L := \Sigma_{t}$ and $s:=s_t$) is 
$\cO \left( \kappa_t s_t \left( \n + \Lambda \right) + 
\kappa_t^3 s_t
\log \left( \frac{\n}{\kappa_t} \right) \right)$.
Consequently, the total complexity cost of performing 
Step \ref{Banach_GRIM_iv}
of the \textbf{Banach GRIM} algorithm 
(for $t$ here playing the role of $s$ there) is 
\beq
    \label{step_iv_t_cost}
        \cO \left( \kappa_t s_t 
        \left( \n + \Lambda \right) + 
        \kappa_t^3 s_t
        \log \left( \frac{\n}{\kappa_t } \right) 
        \right).
\eeq
By summing the complexity costs in \eqref{step_iv_t_cost}
over $t \in \{2, \ldots , M\}$ we deduce that the complexity 
cost of Step \ref{Banach_GRIM_iv} of the 
\textbf{Banach GRIM} algorithm is 
\beq
    \label{step_iv_cost}
        \cO \left(
        \sum_{j=2}^M \kappa_j s_j
        \left( \n + \Lambda \right)
        +
        \kappa_j^3 s_j
        \log \left( \frac{\n}{\kappa_j}
        \right)
        \right).
\eeq
Since we have previously observed that the complexity cost 
of performing Steps \ref{Banach_GRIM_i}, \ref{Banach_GRIM_ii},
and \ref{Banach_GRIM_iii} of the \textbf{Banach GRIM} 
algorithm is $\cO\left( k_1 s_1 \left( \n + \Lambda \right) + 
\n \Lambda +
k_1^3 s_1 \log \left( \frac{\n}{k_1} \right) \right)$, 
it follows from \eqref{step_iv_cost} that the complexity cost
of performing the entire \textbf{Banach GRIM} algorithm is
(recalling that $\kappa_1 := k_1$)
\beq
    \label{GRIM_cost_pf}
        \cO \left( \n \Lambda + 
        \sum_{j=1}^M \kappa_j s_j 
        \left( \n + \Lambda \right)
        +
        \kappa_j^3 s_j
        \log \left( \frac{\n}{\kappa_j}
        \right)
        \right)
\eeq
as claimed in \eqref{GRIM_M_steps_cost}.
This completes the proof of Lemma 
\ref{GRIM_cost_lemma}.
\end{proof}
\vskip 2mm
\noindent
We end this subsection by explicitly recording the 
complexity cost estimates resulting from Lemma 
\ref{GRIM_cost_lemma} for some particular choices of parameters.
We assume for both choices that we are in the situation 
that $\n >> \Lambda$.

First, consider the choices that $M:=1$, $k_1 := \Lambda$, 
and $s_1 := 1$.
This corresponds to making a single application of 
recombination to find an approximation of $\vph$ that is
within $\ep_0$ of $\vph$ at every linear functional in $\Sigma$.
Lemma \ref{GRIM_cost_lemma} tells us that the complexity cost
of doing this is 
$\cO \left( \n \Lambda + \Lambda^2 + 
\Lambda^3 \log \left( \frac{\n}{\Lambda} \right) \right)$.

Secondly, consider the choices $M := \Z_{\geq 1}$,
$k_1 = \ldots = k_M = 1$, and arbitrary fixed
$s_1 , \ldots , s_M \in \Z_{\geq 1}$.
Let $L \subset \Sigma$ denote the collection of linear 
functionals that is inductively grown during the 
\textbf{Banach GRIM} algorithm. 
These choices then correspond to adding a single new
linear functional to $L$ at each step. 
It follows, for each 
$j \in \{1, \ldots , M\}$, that 
$\kappa_j := \sum_{i=1}^j k_i = j$.
Lemma \ref{GRIM_cost_lemma} tells us that the complexity cost
of doing this is 
$\cO \left( \n \Lambda + \sum_{j=1}^M j s_j (\n+\Lambda)
+ j^3 s_j \log \left( \frac{\n}{j} \right)\right)$.
If we restrict to a single application 
of recombination at each step (i.e. choosing 
$s_1 = \ldots = s_M = 1$), then the complexity cost becomes
$\cO \left( M^2 \left( \n + \Lambda \right)
+ \sum_{j=1}^M j^3 \log \left( \frac{\n}{j} \right) \right)$.

In particular, if we take $M := \Lambda$ (which corresponds 
to allowing for the possibility that the collection $L$ 
may grow to be the entirety of $\Sigma$) then the 
complexity cost is 
\beq    
    \label{comp_cost_conv_anal_case}
        \cO \left( \Lambda^2 \left( \n + \Lambda \right)
        + \sum_{j=1}^{\Lambda} j^3 \log 
        \left( \frac{\n}{j} \right) \right).
\eeq
If $\n$ is large enough that $\Lambda < e^{-1/3}\n$, then 
the function $x \mapsto x^3 \log (\n/x)$ is increasing on 
the interval $[1,\Lambda]$.
Under these conditions, the complexity cost in 
\eqref{comp_cost_conv_anal_case} is no worse than
\beq
    \label{comp_cost_conv_anal_ideal_assumps}
        \cO \left( \Lambda^2 \n 
        + \Lambda^4 \log \left( \frac{\n}{\Lambda} \right) 
        \right).
\eeq

\section{Convergence Analysis}
\label{conv_gen}
In this section we establish a convergence result for the
\textbf{Banach GRIM} algortihm and discuss its 
consequences. The section is organised as follows. 
In Subsection \ref{notation} we fix notation and conventions
that will be used throughout the entirety of Section 
\ref{conv_gen}.
In Subsection \ref{Main_Theoretical_Result} we state the 
\textbf{Banach GRIM Convergence} theorem \ref{banach_conv}.
This result establishes, in particular, a non-trivial 
upper bound on the maximum number of steps the
\textbf{Banach GRIM} algorithm, under the 
choice $M := \min \left\{ \n - 1 , \Lambda \right\}$ and that 
for every $t \in \{1, \ldots , M \}$ we have $k_t := 1$, 
can run for before terminating. 
We do not consider the potentially beneficial impact of 
considering multiple shuffles at each step, and so we assume 
throughout this section that for every $t \in \{1 , \ldots , M\}$
we have $s_t := 1$.
We additionally discuss some performance 
guarantees available as a consequence of the 
\textbf{Banach GRIM Convergence} theorem \ref{banach_conv}.
In Subsection \ref{Supplementary_Lemmata} we record several 
lemmata recording, in particular, 
estimates for the approximations found at 
each completed step of the algorithm 
(cf. Lemma \ref{banach_robust_lemma}), 
and the properties 
regarding the distribution of the collection of linear 
functionals selected at each completed step of the algorithm
(cf. Lemma \ref{dist_between_interp_functionals}).
In Subsection \ref{Main_Theoretical_Result_Proof} we
use the supplementary lemmata from Subsection 
\ref{Supplementary_Lemmata} to prove the 
\textbf{Banach GRIM Convergence} theorem \ref{banach_conv}.

\subsection{Notation and Conventions}
\label{notation}
In this section we fix notation and conventions that 
will be used throughout the entirety of Section 
\ref{conv_gen}.

Given a real Banach space $Z$, an integer 
$m \in \Z_{\geq 1}$, a subset 
$L = \{ z_1 , \ldots , z_m \} \subset Z$, an element 
$r \in \R_{\geq 0}$, and a norm $\lambda$ on 
$\R^m$, we will refer to the set 
\beq
    \label{l1_dual_space_ball}
        \Span_{\lambda} \left( L, r \right)  
        :=
        \left\{ \sum_{j=1}^m \be_j z_j 
        ~:~ 
        \be := (\be_1 , \ldots , \be_m) \in \R^m 
        \text{ and } 
        \lambda (\be) \leq r \right\}
        \subset Z
\eeq
as the 
\emph{$\lambda$ distance $r$ span of 
$L = \{z_1 , \ldots , z_m\}$ in $Z$}.
If $r = \infty$ then we take 
$\Span_{\lambda} \left( L , \infty \right)$ 
to be the usual linear span of the set 
$L = \{z_1 , \ldots , z_m\}$ in $Z$. 
Further, given constants $c > 0$ and 
$\al > \be \geq 0$, we refer to the set 
\beq
    \label{l1_reach_L}
        \Reach_{\lambda} 
        \left( L , c , \al , \be \right)
        :=
        \twopartdef{\bigcup_{0 \leq \tau \leq \frac{\al}{\be}}
        \ovSpan_{\lambda} \left( L , \tau \right)_{
        \frac{\al - \tau \be}{c}
        }}{\be > 0}
        {\Span(L)_{
        \frac{\al}{c}
        }}{\be = 0}
\eeq
as the \emph{$c$ weighted $\be$ based $\lambda$ 
distance $\al$ reach of $L=\{z_1 , \ldots , z_m\}$ in $Z$}.
Given $V = \{v_1 , \ldots , v_m \} \subset Z$ we let
$\conv (V)$
denote the \emph{closed convex hull} of 
$\left\{ v_1 , \ldots , v_m \right\} $ in $Z$. That is
\beq
    \label{closed_conv_hull}
        \conv (V) 
        := 
        \left\{ \sum_{j=1}^m \be_j v_j ~:~
        \be_1 , \ldots , \be_m \in [0,1] \text{ and }
        \sum_{j=1}^m \be_j = 1 \right\}
        \subset Z.
\eeq
For future use, we record that 
the $||.||_{l^1(\R^m)}$ distance $1$ span of 
$L = \{z_1 , \ldots , z_m\}$ in $Z$ coincides with the closed
convex hull of the set 
$\{z_1 , \ldots , z_m , 0 , -z_1 , \ldots , -z_m \}$ in $Z$, 
and that consequently the $(\al - \be)/c$-fattening of 
this convex hull is contained 
within the $c$-weighted $\be$-based $||\cdot||_{l^1(\R^m)}$ 
distance $\al$-span of $L$, denoted
$\Reach_{||\cdot||_{l^1(\R^m)}} (L, c , \al , \be)$,
whenever $c > 0$ and $\al > \be \geq 0$.

\begin{lemma}
\label{conv_hull_equiv_form_lemma}
Let $Z$ be a Banach space and $m \in \Z_{\geq 1}$.
Assume that $L = \{ z_1 , \ldots , z_m \} \subset Z$ and 
define $V := \left\{
z_1 , \ldots , z_m , 0 , -z_1 , \ldots , -z_m \right\}$. 
Then we have that 
\beq
    \label{conv_hull_equiv_form}
        \conv (V)
        = 
        \Span_{||\cdot||_{l^1(\R^m)} } 
        \left(L , 1\right).
\eeq
Consequently, given any $c>0$ and any $\al > \be \geq 0$
we have that 
\beq
    \label{conv_hull_in_reach}
        \conv(V)_{\frac{\al - \be}{c}} \subset 
        \Reach_{||\cdot||_{l^1(\R^m)}} 
        ( L , c , \al , \be ).
\eeq
\end{lemma}

\begin{proof}[Proof of Lemma \ref{conv_hull_equiv_form_lemma}]
Let $Z$ be a Banach space, $m \in \Z_{\geq 1}$, and
assume that $L = \{z_1 , \ldots , z_m \} \subset Z$.
Define $z_0 := 0 \in M$ and, for each $j \in \{1 \ldots , m\}$ 
define $z_{-j} := -z_j$. Let
$\cc := \conv \left( \left\{ 
z_1 , \ldots , z_m , 0 , -z_1 , \ldots , -z_m 
\right\} \right)$ and 
$\cd := \Span_{||\cdot||_{l^1(\R^m)} } 
\left( L , 1 \right)$.
We establish \eqref{conv_hull_equiv_form} by proving 
both the inclusions $\cc \subset \cd$ and $\cd \subset \cc$.

We first prove that $\cc \subset \cd$. To do so, 
suppose that $z \in \cc$. Then there are coefficients 
$\be_{-m} , \ldots , \be_{m} \in [0,1]$ such that
$z = \sum_{j=-m}^m \be_j z_j$ and $\sum_{j=-m}^m \be_j = 1$.
Then we have that 
$z = \sum_{j=1}^m \left( \be_j - \be_{-j} \right) z_j$, 
and moreover we may observe via the triangle inequality that 
$\sum_{j=1}^m \left| \be_j - \be_{-j} \right| \leq 
\sum_{j=-m}^m \be_j = 1$.
Consequently, by taking $c_j := \be_j - \be_{-j}$ for 
$j \in \{1, \ldots , m\}$, we see that 
$z \in \cd$ by definition.
The arbitrariness of $z \in \cc$ allows us to conclude that
$\cc \subset \cd$.

In order to prove $\cd \subset \cc$, suppose now 
that $z \in \cd$. Hence there exist 
$c_1 ,\ldots , c_m \in \R$, with 
$\sum_{j=1}^m \left|c_j \right| \leq 1 $, and such that
$z = \sum_{j=1}^m c_j z_j$.
For each $j \in \{1 , \ldots , m\}$, if $c_j >0$ define 
$\be_j := c_j$ and $\be_{-j} := 0$, and if $c_j \leq 0$ define
$\be_j := 0$ and $\be_{-j} := -c_j$.
Observe that, for each $j \in \{1 , \ldots , m\}$, we have 
$\be_j , \be_{-j} \in [0,1]$.
Further, for each $j \in \{1, \ldots , m\}$,
we have $\be_j + \be_{-j} = |c_j|$, and so
$\sum_{j=1}^m \be_j + \be_{-j} = \sum_{j=1}^m |c_j|$.
Finally, define $\be_0 := 1 - \sum_{j=1}^m |c_j| \in [0,1]$ so 
that $\be_{-m} , \ldots , \be_m \in [0,1]$ with 
$\sum_{j=-m}^m \be_j = 1$.
We observe that
$z = \sum_{j=1}^m c_j z_j = \sum_{c_j \geq 0} c_j z_j 
+ \sum_{c_j < 0} (-c_j) (-z_j) = 
\sum_{j=1}^m \be_j z_j + \be_{-j}z_{-j} =
\sum_{j=-m}^m \be_j z_j$ 
where the last equality holds since $z_0 := 0$.
Consequently, $z \in \cc$.
The arbitrariness of $z \in D$ allows us to conclude
$D \subset \cc$.

Together, the inclusions $\cc \subset D$ and $D \subset \cc$
establish that $\cc = D$ as claimed in 
\eqref{conv_hull_equiv_form}.
We complete the proof of Lemma \ref{conv_hull_equiv_form_lemma}
by using \eqref{conv_hull_equiv_form} to 
establish \eqref{conv_hull_in_reach}.

Let $c > 0$ and $\al \geq \be \geq 0$. 
First suppose $\beta > 0$. Then $\al/\be \geq 1$ and 
so 
\beq
    \label{reach_tau=1_case_a}
        \Reach_{||\cdot||_{l^1(\R^m)}}
        (L, c , \al ,\be)
        \stackrel{\eqref{l1_reach_L}}{=}
        \bigcup_{0 \leq \tau \leq \frac{\al}{\be}}
        \ovSpan (L , \tau)_{\frac{\al - \tau \beta}{c}}
        \supset 
        \ovSpan (L , 1)_{\frac{\al - \beta}{c}}
        \stackrel{\eqref{conv_hull_equiv_form}}{=}
        \conv (V)_{\frac{\al - \beta}{c}}.
\eeq 
If $\be = 0$ then 
\beq
    \label{reach_tau=1_case_b}
        \Reach_{||\cdot||_{l^1(\R^m)}}
        (L, c , \al ,0)
        \stackrel{\eqref{l1_reach_L}}{=}
        \Span (L)_{\frac{\al}{c}}
        \supset 
        \conv (V)_{\frac{\al}{c}}
\eeq 
where the last inclusion follows from the observation 
that any element in $V$ is a linear combination of 
the elements in $L = \{z_1 , \ldots , z_m \}$.
Together, \eqref{reach_tau=1_case_a} and 
\eqref{reach_tau=1_case_b} establish that we always have 
$\Reach_{||\cdot||_{l^1(\R^m)}}
(L, c , \al ,\be) \supset \conv (V)$
as claimed in
\eqref{conv_hull_in_reach}.
This completes the proof of Lemma \ref{conv_hull_equiv_form_lemma}.
\end{proof}
\vskip 2mm
\noindent
Given a real Banach space $Z$, an integer $m \in \Z_{\geq 1}$,
a subset $S = \{ z_1 , \ldots , z_m \} \subset Z$, and 
$w = \sum_{j=1}^m w_j z_j \in \Span(S)$ for 
$w_1 , \ldots , w_m \in \R$, we refer to the set 
$\support (w) := \left\{ z_j ~:~
        j \in \{1 , \ldots , \n \} \text{ and } 
        w_j \neq 0 \right\}$
as the \emph{support} of $w$.
The cardinality of $\support(w)$ is equal to the number
of non-zero coefficients in the expansion of $w$ in terms 
of $z_1 , \ldots , z_m$.

Finally, given a real Banach space $Z$, a $r \in \R_{\geq 0}$
and a subset $A \subset Z$, we denote the 
\emph{$r$-fattening} of $A$ in $Z$ by $A_r$. That is, we define 
$A_r := \left\{ z \in Z ~:~ \exists a \in A \text{ with }
||z - a||_Z \leq r \right\}$.

\subsection{Main Theoretical Result}
\label{Main_Theoretical_Result}
In this subsection we state our main theoretical result 
and then discuss its consequences. Our main theoretical result 
is the following \textbf{Banach GRIM Convergence} 
theorem for the \textbf{Banach GRIM} algorithm under the 
choice $M := \min \left\{ \n - 1 , \Lambda \right\}$ and that 
for every $t \in \{1, \ldots , M \}$ we have $k_t := 1$ and 
$s_t := 1$.

\begin{theorem}[\textbf{Banach GRIM Convergence}]
\label{banach_conv}
Assume $X$ is a Banach space with dual-space $X^{\ast}$.
Let $\ep > \ep_0 \geq 0$.
Let $\n , \Lambda \in \Z_{\geq 1}$ and set 
$M := \min \left\{ \n - 1 , \Lambda \right\}$.
Let
$\cf := \{ f_1 , \ldots , f_{\n} \} \subset X \setminus \{0\}$ 
and $\Sigma \subset X^{\ast}$ be a finite subset with cardinality 
$\Lambda$.
Let $a_1 , \ldots , a_{\n} \in \R \setminus \{0\}$ and define 
$\vph \in \Span(\cf)$ and $C > 0$ by
\beq
    \label{banach_conv_thm_varphi_C}
        (\bI) \quad \vph := \sum_{i=1}^{\n} 
        a_i f_i
        \qquad \text{and} \qquad
        (\bII) \quad C := \sum_{i=1}^{\n}
        |a_i| || f_i ||_X > 0.
\eeq
Then there is a non-negative integer 
$N = N(\Sigma,C,\ep,\ep_0)\in \Z_{\geq 0}$, given by
\beq
    \label{banach_conv_thm_r1_packing_num}
        N := 
        \max \left\{ d \in \Z ~:~
        \begin{array}{c}
            \text{There exists }
            \sigma_1 , \ldots , \sigma_d
            \in \Sigma 
            \text{ such that for every } 
            j \in \{1, \ldots , d-1\} \\ 
            \text{we have }
            \sigma_{j+1} \notin 
            \Reach_{||\cdot||_{l^1(\R^{j})}} 
            \left( \{\sigma_1, \ldots , \sigma_{j}\} , 
            2C , \ep , \ep_0 \right) 
        \end{array}
        \right\},
\eeq
for which the following is true. 

Suppose $N \leq M := \min \left\{ \n - 1 , M \right\}$ and
consider applying the \textbf{Banach GRIM} algorithm 
to approximate $\vph$ on $\Sigma$ with
$\ep$ as the accuracy threshold, $\ep_0$ as
the acceptable recombination error bound, 
$M$ as the maximum number of steps, 
$s_1 = \ldots = s_M = 1$ as the shuffle numbers, and with 
the integers $k_1 , \ldots , k_M$ in Step \ref{Banach_GRIM_i}
of the \textbf{Banach GRIM} algorithm all being chosen 
equal to $1$
(cf. \textbf{Banach GRIM} \ref{Banach_GRIM_i}).
Then, after at most $N$ steps the algorithm terminates.
That is, there is some integer $n \in \{ 1 , \ldots , N\}$
for which the \textbf{Banach GRIM} algorithm terminates 
after completing $n$ steps.
Consequently, there are coefficients
$c_1 , \ldots , c_{n+1} \in \R$ and 
indices 
$e(1) , \ldots , e(n+1) \in \{1, \ldots , \n\}$
with
\beq
    \label{banach_conv_thm_coeff_sum}
        \sum_{s=1}^{n+1} \left|c_s\right| 
        || f_{e(s)}||_X
        = C,
\eeq
and such that the element $u \in \Span(\cf)$ 
defined by
\beq
    \label{banach_conv_thm_uM_def_&_acc}
        u := \sum_{s=1}^{n+1}
        c_s f_{e(s)} 
        \qquad \text{satisfies, for every } 
        \sigma \in \Sigma, \text{ that}
        \qquad 
        |\sigma(\vph - u)| \leq \ep. 
\eeq
Further, given any $A > 1$, 
we have,
for every $\sigma \in 
\Reach_{||\cdot||_{l^1(\R^{\Lambda})}} 
\left( \Sigma, 2C, A \ep , \ep \right)$, that 
\beq
    \label{banach_conv_thm_robust_est}
        |\sigma(\vph - u)| \leq 
        A \ep.
\eeq
Finally, if the coefficients 
$a_1 , \ldots , a_{\n} \in \R \setminus \{0\}$ 
corresponding to $\vph$ 
(cf. (\bI) of \eqref{banach_conv_thm_varphi_C})
are all positive (i.e. $a_1 , \ldots , a_{\n} > 0$) then 
the coefficients $c_{1} , \ldots , c_{n+1} \in \R$ 
corresponding to $u$ 
(cf. \eqref{banach_conv_thm_uM_def_&_acc}) are 
all non-negative (i.e. $c_1 , \ldots , c_{n+1} \geq 0$).
\end{theorem}

\begin{remark}
\label{section2_notation_recap}
For the readers convenience, we recall the specific
notation from Subsection \ref{notation} utilised in 
the \textbf{Banach GRIM Convergence} Theorem \ref{banach_conv}.
Firstly, given $l \in \Z_{\geq 1}$, a subset 
$L = \{ \sigma_1 , \ldots , \sigma_l \} \subset X^{\ast}$, 
and $s \in \R_{\geq 0} $ we have 
\beq
    \label{ovSpan_def_recall_l_statement}
        \ovSpan_{||\cdot||_{l^1(\R^{l})}}  
        \left( L , s \right)
        := 
        \left\{ \sum_{s=1}^{l} c_s \sigma_s ~:~ 
        c = (c_1 , \ldots , c_{l}) \in \R^{l} \text{ and } 
        ||c||_{l^1(\R^{l})} \leq s \right\},
\eeq
whilst 
$\ovSpan_{||\cdot||_{l^1(\R^{l})} } \left(
L , \infty \right)
:= \Span \left( L \right)$.
Moreover, given $r \in \R_{\geq 0}$ we define  
$\ovSpan_{||\cdot||_{l^1(\R^{l})}} \left( L , s\right)_r$ to be 
the $r$-fattening of 
$\ovSpan_{||\cdot||_{l^1(\R^{l})}} \left( L , s\right)$ 
in $X^{\ast}$. That is
\beq
    \label{ovSpan_r_fat_def_recall_l_statement}
        \ovSpan_{||\cdot||_{l^1(\R^{l})}}
        \left( L , s \right)_r 
        := 
        \left\{ \sigma \in X^{\ast} :
        \text{There exists }
            \sigma' \in 
            \ovSpan_{||\cdot||_{l^1(\R^{l})}}
            \left( L , s \right)
            \text{ with }
            || \sigma - \sigma' ||_{X^{\ast}} 
            \leq r  
        \right\}.
\eeq
In particular, we have 
$\ovSpan_{||\cdot||_{l^1(\R^{l})}} \left( L , s \right)_0
:= \ovSpan_{||\cdot||_{l^1(\R^{l})}} \left( L , s \right)$.
Finally, given constants $c>0$ and $\al \geq \be \geq 0$ 
we have
\beq
    \label{Reach_def_recap}
        \Reach_{||\cdot||_{l^1(\R^l)}} 
        \left( L , c , \al , \be \right)
        :=
        \twopartdef{\bigcup_{0 \leq s \leq \frac{\al}{\be}}
        \ovSpan_{||\cdot||_{l^1(\R^l)}} 
        \left( L , s \right)_{\frac{\al - s \be}{c}}}
        {\be > 0}
        {\Span(L)_{\frac{\al}{c}}}
        {\be=0.}
\eeq
\end{remark}
\vskip 10pt
\noindent
The \textbf{Banach GRIM Convergence} Theorem \ref{banach_conv}
tells us that the maximum number of steps that the 
\textbf{Banach GRIM} algorithm can complete before 
terminating is 
\beq
    \label{max_step_num_from_thm}
        N := 
        \max \left\{ d \in \Z ~:~
        \begin{array}{c}
            \text{There exists }
            \sigma_1 , \ldots , \sigma_d \in \Sigma 
            \text{ such that for every } 
            j \in \{1, \ldots , d-1\} \\ 
            \text{we have }
            \sigma_{j+1} \notin 
            \Reach_{||\cdot||_{l^1(\R^{j})}} 
            \left( \{\sigma_1, \ldots , \sigma_{j}\} , 
            2C , \ep , \ep_0 \right) 
        \end{array}
        \right\}.
\eeq
Recall from Theorem \ref{banach_conv} 
(cf. \eqref{banach_conv_thm_uM_def_&_acc}) that
the number of elements from $\cf$ required to 
yield the desired approximation of $\vph$ is 
no greater than $N + 1$.
Consequently, via Theorem \ref{banach_conv},
not only can we conclude that we theoretically can use no more 
than $N + 1$
elements from $\cf$ to approximate $\vph$ throughout 
$\Sigma$ with an accuracy of $\ep$, but also that the 
\textbf{Banach GRIM} algorithm
will actually construct such an approximation.
In particular, if $N < \n - 1$ then we are guaranteed 
that the algorithm will find a 
linear combination of \emph{fewer} than $\n$
of the elements $f_1 , \ldots , f_{\n}$ that is 
within $\ep$ of $\varphi$ throughout $\Sigma$.

Further we observe that $N$ defined in 
\eqref{max_step_num_from_thm} depends on both the 
features $\cf$ through the constant $C$ defined 
in \eqref{banach_conv_thm_varphi_C} and the data 
$\Sigma$. The constant $C$ itself depends only on the 
weights $a_1 , \ldots , a_{\n} \in \R$ and 
the values $||f_1||_X , \ldots , ||f_{\n}||_X \in \R_{>0}$.
No additional constraints are imposed on the collection of 
features $\cf$; in particular, we do not assume 
the existence of a linear combination of fewer than $\n$ 
of the features in $\cf$ giving a good approximation of 
$\vph$ throughout $\Sigma$.

We now relate the quantity $N$ defined in 
\eqref{max_step_num_from_thm} to other geometric 
properties of the subset $\Sigma \subset X^{\ast}$.

For this purpose, let 
$\sigma_1 , \ldots , \sigma_N \in \Sigma$ and, 
for each $j \in \{1, \ldots , N\}$, define
$\Sigma_j := \{ \sigma_1 , \ldots , \sigma_j \}$.
Suppose $N \geq 2$ and that 
for every $j \in \{1, \ldots , N-1\}$ we have 
$\sigma_{j+1} \notin 
\Reach_{||\cdot||_{l^1(\R^{j})}} 
\left( \Sigma_j , 2C , \ep , \ep_0 \right)$.

In the case that $\ep_0 = 0$ this means, 
recalling \eqref{Reach_def_recap}, that
$\sigma_{j+1} \notin \Reach_{||\cdot||_{l^1(\R^{j})}} 
\left( \Sigma_j , 2C , \ep , 0 \right) 
= \Span( \Sigma_j )_{\ep/2C}$.
Hence, when $\ep_0 = 0$, the integer $N$ defined in 
\eqref{max_step_num_from_thm} is given by
\beq
    \label{max_step_num_from_thm_ep0=0}
        N = 
        \max \left\{ d \in \Z ~:~
        \begin{array}{c}
            \text{There exists }
            \sigma_1 , \ldots , \sigma_d \in \Sigma 
            \text{ such that for every } 
            j \in \{1, \ldots , d-1\} \\ 
            \text{we have }
            \sigma_{j+1} \notin \Span(\Sigma_j)_{\frac{\ep}{2C}} 
        \end{array}
        \right\}.
\eeq
Thus the maximum number of steps before the \textbf{Banach GRIM}
algorithm terminates is determined by the maximum 
dimension of a linear subspace $\Pi \subset X^{\ast}$ 
that has a basis consisting of elements in $\Sigma$, and yet
its $\ep/2C$-fattening fails to capture $\Sigma$ in the sense 
that $\Sigma \cap \Pi_{\ep/2C} \neq \Sigma$.

Now suppose that $\ep_0 > 0$. In this case we have that 
(cf. \eqref{Reach_def_recap})
\beq
    \label{Reach_when_ep0>0}
        \Reach_{||\cdot||_{l^1(\R^j)}} 
        \left( \Sigma_j , 2C , \ep , \ep_0 \right)
        =
        \bigcup_{0 \leq t \leq \frac{\ep}{\ep_0}}
        \ovSpan_{||\cdot||_{l^1(\R^j)}} 
        \left( \Sigma_j , t \right)_{
        \frac{\ep - t \ep_0}{2C}}
\eeq
which no longer contains the full linear subspace 
$\Span(\Sigma_j)$.
However, by appealing to Lemma \ref{conv_hull_equiv_form_lemma}, 
we conclude that (cf. \eqref{conv_hull_in_reach})
\beq
    \label{Reach_info_ep0>0}
        \conv(V_j)_{\frac{\ep-\ep_0}{2C}} \subset 
        \Reach_{||\cdot||_{l^1(\R^j)}} 
        \left( \Sigma_j , 2C , \ep , \ep_0 \right)
        \quad \text{for} \quad
        V_j := \left\{ -\sigma_j , \ldots , -\sigma_1 , 0 , 
        \sigma_1 , \ldots , \sigma_j \right\} \subset X^{\ast}
\eeq
and we use the notation $\conv(V_j)$ to denote the 
closed convex hull of $V_j$ in $X^{\ast}$.
Consequently, if we define
\beq
    \label{N_conv_def}
        N_{\conv} :=
        \max \left\{ 
        d \in \Z ~:~ 
        \begin{array}{c}
        \text{There exists }
            \sigma_1 , \ldots , \sigma_d \in \Sigma 
            \text{ such that } \forall
            j \in \{1, \ldots , d-1\} 
            \text{ we have} \\
            \sigma_{j+1} \notin 
            \conv (V_j)_{\frac{\ep-\ep_0}{2C}} 
            \text{ for }
            V_j := \left\{ -\sigma_j , \ldots , -\sigma_1 , 0 , 
            \sigma_1 , \ldots , \sigma_j \right\}
            \subset X^{\ast}
        \end{array}
        \right\}, 
\eeq
then the integer $N$ defined in \eqref{max_step_num_from_thm}
is bounded above by $N_{\conv}$, i.e. $N \leq N_{\conv}$.

We can control $N_{\conv}$
defined in \eqref{N_conv_def} by particular 
\emph{packing} and \emph{covering} numbers of $\Sigma$.
To elaborate, if $Z$ is a Banach space and $\cU \subset Z$, 
then for any $r > 0$ the $r$-packing number
of $\cU$, denoted by $N_{\pack}(\cU,Z,r)$, and 
the $r$-covering number of $\cU$, denoted by 
$N_{\cov}(\cU,Z,r)$, are given by
\beq
    \label{intro_gen_pack&cov_num_def}
        \begin{aligned}
            &N_{\pack}(\cU,Z,r) := 
            \max \left\{ d \in \Z : \exists 
            z_1 , \ldots , z_d \in \cU \text{ such that }
            || z_a - z_b ||_Z > r \text{ whenever }
            a \neq b \right\} \quad \text{and} \\
            &N_{\cov}(\cU,Z,r) := 
            \min \left\{ d \in \Z : \exists 
            z_1 , \ldots , z_d \in \cU 
            \text{ such that we have the inclusion }
            \cU \subset \bigcup_{j=1}^d 
            \ovB_Z (z_j , r) \right\}
        \end{aligned}
\eeq
respectively.
Our convention throughout this subsection is that balls 
denoted by $\B$ are taken to be open, whilst those denoted
by $\ovB$ are taken to be closed.

It is now an immediate consequence of \eqref{N_conv_def} 
and \eqref{intro_gen_pack&cov_num_def} that 
$N_{\conv}$ is no greater than the $(\ep - \ep_0)/2C$-packing 
number of $\Sigma$, i.e. 
$N_{\conv} \leq N_{\pack} \left( \Sigma , X^{\ast} , 
(\ep - \ep_0)/2C \right)$. 
Moreover, using the well-known fact that 
the quantities defined in 
\eqref{intro_gen_pack&cov_num_def} satisfy, for any $r > 0$, 
that
$N_{\pack}(\cU,Z,2r) \leq N_{\cov}(\cU,Z,r) 
\leq N_{\pack}(\cU,Z,r)$, we 
may additionally conclude that 
$N_{\conv} \leq 
N_{\pack} \left( \Sigma , X^{\ast} , 
(\ep - \ep_0)/2C \right) \leq 
N_{\cov} \left( \Sigma , X^{\ast} , 
(\ep - \ep_0)/4C \right)$.
Consequently, both the 
$(\ep - \ep_0)/2C$-packing number and the 
$(\ep - \ep_0)/4C$-covering number of $\Sigma$ provide
upper bounds on the number of steps that the 
\textbf{Banach GRIM} algorithm can complete before 
terminating.

In summary, we have that the maximum number of steps 
$N$ that the \textbf{Banach GRIM} algorithm can complete
before terminating (cf. \eqref{max_step_num_from_thm})
satisfies that
\beq
    \label{N_upper_bounds_summary}
        N \leq 
        N_{\conv}
        \leq 
        N_{\pack} \left( \Sigma , X^{\ast} , 
        \frac{\ep - \ep_0}{2C} \right)
        \leq
        N_{\cov} \left( \Sigma , X^{\ast} , 
        \frac{\ep-\ep_0}{4C} \right).
\eeq
Thus the geometric quantities related to $\Sigma$ 
appearing in \eqref{N_upper_bounds_summary} provide 
upper bounds on the number of elements from $\cf$ 
appearing in the approximation found by the 
\textbf{Banach GRIM algorithm}.
Explicit estimates for these geometric quantities allow us
to deduce \emph{worst-case} bounds for the performance 
of the \textbf{Banach GRIM} algorithm; we will 
momentarily establish such bounds for the task of 
\emph{kernel quadrature} via estimates for the covering 
number of the unit ball of a \emph{Reproducing Kernel Hilbert 
Space} (RKHS) covered in \cite{JJWWY20}.

Before considering the particular setting of 
kernel quadrature, we observe how the direction of the 
bounds in Theorem \ref{banach_conv} can be, in some sense, 
reversed. To make this precise, observe that 
Theorem \ref{banach_conv} fixes an $\ep > 0$ as the 
accuracy we desire for the approximation, and then provides
an upper bound on the number of features from the collection 
$\cf$ that are used by the \textbf{Banach GRIM} algorithm 
to construct an approximation of $\vph$ that is within
$\ep$ of $\vph$ throughout $\Sigma$. 
However it would be useful to know, for a given fixed 
$n_0 \in \{2 , \ldots , \n\}$, how well the 
\textbf{Banach GRIM} algorithm can approximate $\vph$ 
throughout $\Sigma$ using no greater than $n_0$ of the 
features from $\cf$.
We now illustrate how Theorem \ref{banach_conv} provides
such information.

Consider a fixed $n_0 \in \{2 , \ldots , \n\}$ and let 
$\be_0 = \be_0(n_0 , C , \Sigma,\ep_0) > 0$ be defined by
\beq
    \label{intro_opp_implication_req}
        \be_0 := 
        \min \left\{ \lambda >  \ep_0 ~:~
        N_{\lambda} \leq n_0 - 1
        \right\}
\eeq
where $N_{\lambda}$ denotes the integer defined in \eqref{max_step_num_from_thm} 
for the constant $\ep > \ep_0$ there as $\lambda$ here.
Then, by applying the \textbf{Banach GRIM} algorithm under 
the same assumptions as in Theorem \ref{banach_conv} and the 
additional choice that $\ep := \be_0$, we may conclude from 
Theorem \ref{banach_conv} that the \textbf{Banach GRIM} 
algorithm terminates after
no more than $n_0 - 1$ steps. Consequently, the algorithm 
returns an approximation $u$ of $\vph$ that is a linear
combination of at most $n_0$ of the features in $\cf$, 
and that is within $\be_0$ of $\vph$ on $\Sigma$ in the 
sense that for every $\sigma \in \Sigma$ we have 
$|\sigma(\vph - u)| \leq \be_0$. 
Hence the relation given in \eqref{intro_opp_implication_req}
provides a guaranteed accuracy 
$\be_0 = \be_0(n_0,C,\Sigma,\ep_0) > 0$ for how well the 
\textbf{Banach GRIM} algorithm can approximate $\vph$ 
with the additional constraint that the approximation is 
a linear combination of no greater than $n_0$ of the 
features in $\cf$. This guarantee ensures both that 
there is a linear combination of at most $n_0$ of the 
features in $\cf$ that is within $\be_0$ of $\vph$ 
throughout $\Sigma$ and that the \textbf{Banach GRIM} 
algorithm will find such a linear combination.

For the remainder of this subsection we turn our attention
to the particular setting of kernel quadrature.
assume that $\cX = \{ x_1 , \ldots , x_{\n} \}$ 
is a finite set of some finite dimensional Euclidean space
$\R^d$, and that $k : \cX \times \cX \to \R$
is a continuous symmetric positive semi-definite kernel that
is bounded above by $1$. 
For each $i \in \{1, \ldots , \n\}$ define a continuous function
$k_{x_i} : \cX \to \R$ by setting, for $z \in \cX$, 
$k_{x_i}(z) := k(z,x_i)$. 
Define 
$K := \left\{ k_{x_i} : i \in \{1 , \ldots , \n\} \right\} 
\subset C^{0}(\cX)$
and let $\cH_k$ denote the RKHS associated to $k$.
In this case it is known that
$\cH_k = \Span (K)$ and hence 
$K \subset \ovB_{\cH_k}(0,1)$.

Suppose $a_1 , \ldots , a_{\n} > 0$ so that 
$\vph := \sum_{i=1}^{\n} a_i \de_{x_i} \in \bbP[\cX]$.
Under the choice that $X := \m[\cX]$, we observe 
that the constant $C$ corresponding to the definition
in \eqref{banach_conv_thm_varphi_C} 
satisfies that $C = 1$. 
Recall that $C^{0}(\cX) \subset \m[\cX]^{\ast}$ via 
the identification of an element $\psi \in C^{0}(\cX)$
with the linear functional $\m[\cX] \to \R$ given by
$\nu \mapsto \nu[\psi] := \int_{\cX} \psi(z) d\nu(z)$.
We abuse notation slightly by referring to both the continuous
function in $C^0(\cX)$ and the associated linear functional
$\m[\cX] \to \R$ as $\psi$.
By defining $\Sigma := K$, 
the kernel quadrature problem in this setting is to 
find an empirical measure $u \in \bbP[\cX]$ whose support
is a strict subset of the support of $\vph$ and such that
$\left| f(\vph - u) \right| \leq \ep$ for every
$f \in \Sigma$.

Recalling \eqref{N_upper_bounds_summary}, the 
performance of the \textbf{Banach GRIM} algorithm for this task
is controlled by the pointwise $(\ep-\ep_0)/4$-covering number of 
$\Sigma$, i.e. by 
$N_{\cov} ( \Sigma, C^0(\cX) , (\ep-\ep_0)/4)$.
That is, there is an integer $M \leq
N_{\cov} ( \Sigma, C^0(\cX) , (\ep - \ep_0)/4)$
such that the \textbf{Banach GRIM} algorithm finds weights 
$c_1 , \ldots , c_{M+1} \geq 0$
and indices $e(1) , \ldots , e(M+1) \in \{1,\ldots ,\n\}$
such that $u := \sum_{s=1}^{M+1} c_s \de_{x_{e(s)}}$
is a probability measure in $\bbP[\cX]$
satisfying, for every 
$f \in \Sigma$, that 
$|\vph(f) - u(f)| \leq \ep$.
A consequence of $\Sigma \subset \ovB_{\cH_k}(0,1)$ is that
the $(\ep-\ep_0)/4$-covering number of 
$\Sigma$ is itself controlled by the 
$(\ep-\ep_0)/4$-covering number of $\ovB_{\cH_k}(0,1)$, 
denoted by
$N_{\cov} \left( \ovB_{\cH_k}(0,1), C^0(\cX) , (\ep-\ep_0)/4
\right)$.

Many authors have considered estimating the covering 
number of the unit ball of a RKHS, see the works 
\cite{Zho02,CSW11,Kuh11,SS13,HLLL18,Suz18,JJWWY20,FS21} 
for example. In our particular setting, the covering 
number estimates in \cite{JJWWY20} yield explicit 
performance guarantees for the \textbf{Banach GRIM}
as we now illustrate.

In our setting the 
kernel has a pointwise convergent Mercer decomposition
$k(x,y) = \sum_{m=1}^{\infty} \lambda_m e_m(x)e_m(y)$
with $\lambda_1 \geq \lambda_2 \geq \ldots \geq 0$
and $\{ e_m \}_{m=1}^{\infty} \subset L^2(\cX)$
being orthonormal \cite{CS08,SS12}. The pairs 
$\left\{ (\lambda_m , e_m) \right\}_{m=1}^{\infty}$
are the eigenpairs for the operator
$T_k : L^2(\cX) \to L^2(\cX)$ defined for $f \in L^2(\cX)$
by $T_k[f] := \int_{\cX} k(\cdot,y) f(y) dy$.
We assume that the eigenfunctions 
$\{e_m\}_{m=1}^{\infty}$ are uniformly bounded in the 
sense that for every $m \in \Z_{\geq 1}$ and every 
$x \in \cX$ we have $|e_m(x)| \leq C_0$ for some constant
$C_0 > 0$. 
Finally, we assume that the eigenvalues 
$\{ \lambda_m\}_{m=1}^{\infty}$ decay exponentially as
$m$ increases in the sense that for every $m \in \Z_{\geq 1}$
we have $\lambda_m \leq C_1 e^{-C_2j}$
for constants $C_1 , C_2 > 0$.
These assumptions are satisfied, for example, by the 
\emph{squared exponential} (\emph{Radial Basis Function}) 
kernel $k(s,t) := e^{-(s-t)^2}$ \cite{JJWWY20}; more 
explicit estimates for this particular choice of kernel 
may be found in \cite{FS21}.

Given any $r \in (0,1)$, it is established in \cite{JJWWY20}
(cf. Lemma D.2 of \cite{JJWWY20}) that under these 
assumptions we have
\beq
    \label{intro_RKHS_cov_num_bd_example}
        \log N_{\cov}(\ovB_{\cH_k}(0,1), C^0(\cX), r)
        \leq C_3 \left( \log \left( \frac{1}{r} \right)
        + C_4 \right)^{2}
\eeq
for constants $C_3 = C_3(C_0,C_1,C_2) > 0$
and $C_4 = C_4(C_0,C_1,C_2) > 0$.
Assuming that $\ep < 4$, by appealing to 
\eqref{intro_RKHS_cov_num_bd_example} for the choice 
$r:= (\ep - \ep_0)/4$ we may conclude that if
$C_3 \left( \log \left( \frac{4}{\ep-\ep_0} \right)
+ C_4 \right)^{2} < \log (\n - 1)$
then the \textbf{Banach GRIM} algorithm will 
return a probability measure $u \in \bbP[\cX]$ given by 
a linear combination of fewer than $\n$ of the point 
masses $\de_{x_1} , \ldots , \de_{x_{\n}}$ 
satisfying, for every $f \in \Sigma$,
that $|\vph(f) - u(f)| \leq \ep$.

Alternatively, given $n_0 \in \{2 , \ldots , \n \}$
define $\be_0 = \be_0 (C_0,C_1,C_2,n_0,\ep_0) > 0$ by
\beq
    \label{intro_kernel_quadrature_beta_0}
        \be_0 := 4e^{C_4} 
        e^{
        -\left(\frac{\log(n_0-1)}{C_3}\right)^{
        \frac{1}{2}
        } } + \ep_0 = 4e^{C_4}
        \left( n_0 - 1\right)^{-\left(
        C_3 \log(n_0 -1)
        \right)^{-\frac{1}{2}}} + \ep_0 > \ep_0.
\eeq
Provided $n_0 > 1 + e^{C_3C_4^2}$ we have that
$\frac{\be_0 - \ep_0}{4} \in (0,1)$.
In this case we observe that 
\eqref{intro_RKHS_cov_num_bd_example} and 
\eqref{intro_kernel_quadrature_beta_0}
yield that 
\beq
    \label{partic_cov_num_ub}
        \log N_{\cov} \left( \ovB_{\cH_k}(0,1) , 
        C^{0}(\cX) , \frac{\be_0 - \ep_0}{4}
        \right) 
        \leq 
        \log (n_0-1).
\eeq
We deduce from Theorem \ref{banach_conv}, with 
$\ep := \be_0$, that the algorithm finds a probability 
measure $u \in \bbP[\cX]$ given by 
a linear combination of no more than $n_0$ of the point 
masses $\de_{x_1} , \ldots , \de_{x_{\n}}$ 
satisfying, for every $f \in \Sigma$,
that $|\vph(f) - u(f)| \leq \be_0$.

As $n_0$ increases, $\be_0$ defined in 
\eqref{intro_kernel_quadrature_beta_0} eventually decays 
slower than $n_0^{-a}$ for any $a > 0$.
This poor asymptotic behaviour is not unexpected for 
an estimate that is itself a combination of worst-case 
scenario estimates. However, we may still observe that
for any integer $A \in \Z_{\geq 1}$ large enough to 
ensure that $1 + e^{C_3C_4^2} < n_0 \leq 1 + e^{A^2/C_3}$
(which in particular requires $A > C_3C_4$), 
that $\be_0 \leq 4e^{C_4} (n_0-1)^{-1/A}$.

.

\subsection{Supplementary Lemmata}
\label{Supplementary_Lemmata}
In this subsection we record several lemmata that will be
used during our proof of the
\emph{Banach GRIM Convergence} Theorem \ref{banach_conv} in 
Subsection \ref{Main_Theoretical_Result_Proof}.
The following result records the consequences arising 
from knowing that an approximation $u \in \Span(\cf)$
is \emph{close} to $\vph$ at a finite set of linear 
functionals in $X^{\ast}$. 

\begin{lemma}[\textbf{Finite Influenced Set}]
\label{banach_robust_lemma}
Assume that $X$ is a Banach space with dual-space $X^{\ast}$. 
Let $\th > \th_0 \geq 0$, $\n \in \Z_{\geq 1}$, and
$\cf = \{f_1 , \ldots , f_\n\} \subset X \setminus \{0\}$.
Let $a_1 , \ldots , a_{\n} \in \R \setminus \{0\}$  
and define $\vph \in \Span(\cf)$ and $C > 0$ by
\beq
    \label{banach_robust_lemma_varphi_const}
        (\bI) \quad \varphi := 
        \sum_{i=1}^{\n} a_i f_i
        \qquad \text{and} \qquad
        (\bII) \quad C := \sum_{i=1}^{\n} 
        |a_i| || f_i ||_X > 0.
\eeq
Suppose that $d \in \Z_{\geq 1}$ with
$L = \left\{ \sigma_1 , \ldots , \sigma_d \right\} 
\subset X^{\ast}$, 
and that $u \in \Span(\cf)$ satisfies both that
$||u||_X \leq C$ and, for every $\sigma \in L$, that
$|\sigma(\vph - u)| \leq \th_0$. Then, using 
our notation conventions from Section \ref{notation} 
(see also Remark \ref{section2_notation_recap}), for every 
$\sigma \in \Reach_{||\cdot||_{l^1(\R^d)}} 
\left( L , 2C , \th , \th_0 \right)$ we have that
\beq
    \label{banach_robust_lemma_conc}
        |\sigma(\vph - u)| \leq \th.
\eeq 
\end{lemma}

\begin{proof}[Proof of Lemma \ref{banach_robust_lemma}]
Assume that $X$ is a Banach space with dual-space
$X^{\ast}$. 
Let $\th > \th_0 \geq 0$, $\n \in \Z_{\geq 1}$, and
$\cf = \{f_1 , \ldots , f_\n\} \subset X \setminus \{0\}$.
Let $a_1 , \ldots , a_{\n} \in \R \setminus \{0\}$  
and define $\vph \in \Span(\cf)$ and $C > 0$ by 
(\bI) and (\bII) in 
\eqref{banach_robust_lemma_varphi_const} respectively.
Suppose that $d \in \Z_{\geq 1}$ and that 
$L = \left\{ \sigma_1 , \ldots , \sigma_d \right\} 
\subset X^{\ast}$.
Let $u \in \Span(\cf)$ and assume both that $||u||_X \leq C$
and that, for every $\sigma \in L$, we have
$|\sigma(\vph - u)| \leq \th_0$.
It follows from (\bI) and (\bII) of 
\eqref{banach_robust_lemma_varphi_const}
that $||\vph||_X \leq C$.
Hence $||\vph - u||_X \leq 2C$.

We deal with the cases $\th_0 = 0$ and 
$\th_0 \in (0,\th)$ separately.

First suppose that $\th_0 = 0$. In this case we have 
that $\sigma_1(\vph-u) = \ldots = \sigma_d(\vph-u)=0$, and
that 
\beq
    \label{th0=0_influence_set_def}
        \Reach_{||\cdot||_{l^1(\R^d)}} 
        \left( L , 2C , \th , 0 \right)
        \stackrel{\text{def}}{=}
        \bigcup_{0 \leq t \leq \infty}
        \ovSpan_{||\cdot||_{l^1(\R^d)}} 
        \left( L , t \right)_{\frac{\th}{2C}} 
        =
        \Span(L)_{\frac{\th}{2C}}.
\eeq
As a consequence of \eqref{th0=0_influence_set_def}, 
we need only establish the estimate in 
\eqref{banach_robust_lemma_conc} for 
$\sigma \in \Span(L)_{\th/2C}$.
Assuming $\sigma \in \Span(L)_{\th/2C}$, then there is some
$\rho \in \Span(L)$ with 
$||\sigma - \rho||_{X^{\ast}} \leq \th/2C$. 
Since $\rho \in \Span(L)$ there are coefficients
$c_1 , \ldots , c_d \in \R$ for which 
$\rho = \sum_{j=1}^d c_j \sigma_j$. 
Evidently we have that $\rho(\vph-u) = \sum_{j=1}^d c_j 
\sigma_j(\vph-u) = 0$.
Consequently we may estimate that
\beq
    \label{influence_est_th0=0}
        | \sigma(\vph-u)|
        \leq
        | (\sigma - \rho)(\vph-u)|
        +
        |\rho(\vph-u)| 
        \leq
        ||\sigma - \rho||_{X^{\ast}}
        ||\vph-u||_X
        \leq 
        2C \left( \frac{\th}{2C} \right) = \th.
\eeq
The arbitrariness of $\sigma$ means we may conclude from 
\eqref{influence_est_th0=0} that 
\eqref{banach_robust_lemma_conc} is valid as claimed 
when $\th_0=0$.

Now we suppose that $\th_0 \in (0,\th)$.

Since $\ovSpan_{||\cdot||_{l^1(\R^d)}} (L,0) = \{0\} \subset X^{\ast}$,
if $\sigma \in \ovSpan_{||\cdot||_{l^1(\R^d)}} (L,0)_{\th/2C}$
then $||\sigma||_{X^{\ast}} \leq \th/2C$.
Consequently
$|\sigma(\vph-u)| \leq ||\sigma||_{X^{\ast}} ||\vph-u||_X 
\leq \th$ as claimed in \eqref{banach_robust_lemma_conc}.

Now consider $t \in (0, \th/\th_0]$ and let 
$\sigma \in \ovSpan_{||\cdot||_{l^1(\R^d)}} 
\left( L , t \right)_{\frac{\th - t\th_0}{2C}}$. 
Then there is some $\rho \in 
\ovSpan_{||\cdot||_{l^1(\R^d)}} \left( L , t \right)$
with $||\sigma - \rho||_{X^{\ast}} \leq \frac{\th - t \th_0}{2C}$.
Further, there are coefficients 
$c_1 , \ldots , c_d \in \R$ with 
$|c_1| + \ldots + |c_d| \leq t$ and such that 
$\rho = \sum_{j=1}^d c_j \sigma_j$. 
Consequently, we may estimate that
\begin{multline}
    \label{influence_est_th0>0_b}
        | \sigma(\vph-u)|
        \leq
        | (\sigma - \rho)(\vph-u)|
        +
        |\rho(\vph-u)| 
        \leq
        ||\sigma - \rho||_{X^{\ast}}
        ||\vph-u||_X
        +
        \sum_{j=1}^d |c_j| |\sigma_j(\vph-u)| \\
        \leq 2C \left( \frac{\th - t \th_0}{2C} \right) + 
        \th_0 \sum_{j=1}^d |c_j|
        \leq \th - t \th_0 + t \th_0
        =
        \th.
\end{multline}
Since $t \in (0, \th/\th_0]$ and 
$\sigma \in \ovSpan_{||\cdot||_{l^1(\R^d)}} 
\left( L , t \right)_{\frac{\th - t\th_0}{2C}}$ 
were both arbitrary, we may conclude that 
\eqref{influence_est_th0>0_b} is valid whenever
$\sigma \in \bigcup_{0 < t \leq \th/\th_0} 
\ovSpan_{||\cdot||_{l^1(\R^d)}} 
\left( L , t \right)_{\frac{\th - t\th_0}{2C}}$.

The combination of the proceeding two paragraphs allows 
us to conclude,
for $\th_0 \in (0,\th)$, that whenever 
$\sigma \in \bigcup_{0 \leq t \leq \th/\th_0} 
\ovSpan_{||\cdot||_{l^1(\R^d)}} 
\left( L , t \right)_{\frac{\th - t\th_0}{2C}}
= \Reach_{||\cdot||_{l^1(\R^d)}} 
\left( L , 2C , \th , \th_0 \right)$
we have 
$|\sigma(\vph-u)| \leq \th$. 
Consequently, \eqref{banach_robust_lemma_conc} is valid
for the case that $\th_0 \in (0,\th)$. 
And having earlier established the validity of 
\eqref{banach_robust_lemma_conc} when $\th_0=0$, this 
completes the proof of Lemma \ref{banach_robust_lemma}.
\end{proof}
\vskip 2mm
\noindent 
We can use the result of Lemma \ref{banach_robust_lemma}
to prove that the linear functionals 
selected during the \textbf{Banach GRIM} algorithm, under the 
choice $M := \min \left\{ \n - 1 , \Lambda \right\}$ and that 
for every $t \in \{1, \ldots , M \}$ we have $k_t := 1$ and 
$s_t := 1$, are
separated by a definite $X^{\ast}$-distance.
The precise result is the following lemma.

\begin{lemma}[\textbf{Banach GRIM Functional 
Separation}]
\label{dist_between_interp_functionals}
Assume $X$ is a Banach space with dual-space $X^{\ast}$.
Let $\th > \th_0 \geq 0$, and $\n , \Lambda \in \Z_{\geq 1}$
such that 
$M := \min \left\{ \n - 1 , \Lambda \right\} \geq 2$.
Suppose that 
$\cf := \{ f_1 , \ldots , f_{\n} \} \subset X \setminus \{0\}$,
and that $\Gamma \subset X^{\ast}$ is finite with cardinality 
$\Lambda$. 
Let $a_1 , \ldots , a_{\n} \in \R \setminus \{0\}$ 
and define $\vph \in \Span(\f)$ and $C > 0$ by
\beq
    \label{banach_separation_const}
        (\bI) \quad 
        \varphi := \sum_{i=1}^{\n} a_i f_i
        \qquad \text{and} \qquad
        (\bII) \quad 
        C := \sum_{i=1}^{\n} |a_i| || f_i ||_X > 0.
\eeq
Consider applying the \textbf{Banach GRIM} algorithm 
to approximate $\vph$ on $\Gamma$ with
$\th$ as the accuracy threshold, $\th_0$ as
the acceptable recombination error bound, 
$M$ as the maximum number of steps, 
$s_1 = \ldots = s_M = 1$ as the shuffle numbers, and with 
the integers $k_1 , \ldots , k_M$ in Step \ref{Banach_GRIM_i}
of the \textbf{Banach GRIM} algorithm all being chosen 
equal to $1$
(cf. \textbf{Banach GRIM} \ref{Banach_GRIM_i}).
Suppose that $m \in \Z_{\geq 2}$ and the algorithm
reaches and carries out the $m^{th}$ step without 
terminating. For each $l \in \{1 , \ldots , m \}$ let
$\sigma_l \in \Gamma$ be the linear functional selected 
at the $l^{\text{th}}$ step, and let $u_l \in \Span(\cf)$
be the element found via recombination 
(cf. the \emph{recombination thinning} Lemma 
\ref{Banach_recombination_lemma}) such that, 
for every $s \in \{1, \ldots , l\}$, we have 
$|\sigma_s(\vph - u_l)| \leq \th_0$
(cf. \textbf{Banach GRIM} \ref{Banach_GRIM_iii}
and \ref{Banach_GRIM_iv}).
Further, for each $l \in \{1, \ldots , m\}$, let
$\Gamma_l := \left\{ \sigma_1 , \ldots , \sigma_l \right\}
\subset \Gamma$ and
$\Omega_l := \Reach_{||\cdot||_{l^1(\R^l)}} \left( \Gamma_l , 2C, \th, \th_0 \right)$
where we use our notation conventions
from Section \ref{notation}
(see also Remark \ref{section2_notation_recap}). 
Then for any $l \in \{1, \ldots , m\}$ we have,
for every $\sigma \in X^{\ast}$, we have that
\beq
    \label{banach_separation_estimate}
        | \sigma( \varphi - u_l ) |
        \leq
        \min \left\{ 2C || \sigma ||_{X^{\ast}} ~,~
        2C \dist_{X^{\ast}}( \sigma, \Omega_l )
        + \th \right\}.
\eeq
In particular, for every $\sigma \in \Omega_l$ we have
\beq
    \label{Omega_l_est_sep_lemma}
        \left| \sigma(\vph-u_l) \right| \leq \th.
\eeq
Finally, as a consequence we have,
for every $l \in \{2, \ldots , m\}$, that
\beq
    \label{banach_separation_dist}
        \sigma_l \notin \Omega_{l-1} :=
        \Reach_{||\cdot||_{l^(\R^{l-1})}}
        \left( \Gamma_{l-1} , 2C , \th, \th_0 \right). 
\eeq
\end{lemma}

\begin{remark}
\label{def_dist_apart_rmk}
Using the same notation as in Lemma 
\ref{dist_between_interp_functionals}, we claim that 
\eqref{banach_separation_dist} ensures, for every 
$j \in \{1, \ldots , l-1\}$, that 
\beq
    \label{def_dist_apart_conc}
        ||\sigma_l - \sigma_j ||_{X^{\ast}} \geq 
        \frac{\th - \th_0}{2C}.
\eeq
To see this, recall that 
(cf. Subsection \ref{notation} or 
Remark \ref{section2_notation_recap})
\beq
    \label{Reach_def_recall_rmk}
        \Reach_{||\cdot||_{l^(\R^{l-1})}}
        \left( \Gamma_{l-1} , 2C , \th, \th_0 \right)
        \stackrel{\text{def}}{=} 
        \bigcup_{0 \leq \tau \leq \frac{\th}{\th_0}}
        \ovSpan_{||\cdot||_{l^(\R^{l-1})}}
        \left( \Gamma_{l-1} , \tau \right)_{
        \frac{\th - \tau \th_0}{2C}
        }.
\eeq
Since $\th > \th_0$ we may take $\tau := 1$ in 
\eqref{Reach_def_recall_rmk} to conclude via 
\eqref{banach_separation_dist} that
\beq
    \label{sigma_l_not_in_tau=1_rmk}
        \sigma_l \notin 
        \ovSpan_{||\cdot||_{l^(\R^{l-1})}}
        \left( \Gamma_{l-1} , 1 \right)_{
        \frac{\th - \th_0}{2C}
        }.
\eeq
Recall that 
(cf. Subsection \ref{notation} or 
Remark \ref{section2_notation_recap})
\beq
    \label{ovSpan_def_recall_rmk}
        \ovSpan_{||\cdot||_{l^(\R^{l-1})}}
        \left( \Gamma_{l-1} , 1 \right) 
        \stackrel{\text{def}}{=}
        \left\{ \sum_{s=1}^{l-1} c_s \sigma_s 
        ~:~ 
        c = (c_1 , \ldots , c_{l-1} ) \in \R^{l-1}
        \text{ with }
        || c ||_{l^1(\R^{l-1})} \leq 1
        \right\}.
\eeq
A consequence of \eqref{ovSpan_def_recall_rmk} is that for 
every $j \in \{1, \ldots , l-1\}$ we have 
$\sigma_j \in \ovSpan_{||\cdot||_{l^(\R^{l-1})}}
\left( \Gamma_{l-1} , 1 \right)$,
hence \eqref{sigma_l_not_in_tau=1_rmk} means 
$||\sigma_l - \sigma_j||_{X^{\ast}} > \frac{\th - \th_0}{2C}$
as claimed in \eqref{def_dist_apart_conc}.
\end{remark}

\begin{proof}[Proof of 
Lemma \ref{dist_between_interp_functionals}]
Assume $X$ is a Banach space with dual-space $X^{\ast}$.
Let $\th > \th_0 \geq 0$, and $\n , \Lambda \in \Z_{\geq 1}$ such that 
$M := \min \left\{ \n - 1 , \Lambda \right\} \geq 2$.
Suppose that $\cf := \{ f_1 , \ldots , f_{\n} \} \subset X \setminus \{0\}$,
and that $\Gamma \subset X^{\ast}$ is finite with cardinality $\Lambda$. 
Let $a_1 , \ldots , a_{\n} \in \R \setminus \{0\}$ 
and define $\varphi \in \Span(\cf) \subset X$ and 
$C > 0$ by (cf. (\bI) and (\bII) of 
\eqref{banach_separation_const} respectively)
\beq
    \label{banach_separation_const_pf}
        (\bI) \quad 
        \varphi := \sum_{i=1}^{\n} a_i f_i
        \qquad \text{and} \qquad
        (\bII) \quad 
        C := \sum_{i=1}^{\n} |a_i| || f_i ||_X > 0.
\eeq
Consider applying the \textbf{Banach GRIM} algorithm 
to approximate $\vph$ on $\Gamma$ with
$\th$ as the accuracy threshold, $\th_0$ as
the acceptable recombination error bound, 
$M$ as the maximum number of steps, 
$s_1 = \ldots = S_M = 1$ as the shuffle numbers, and with 
the integers $k_1 , \ldots , k_M$ in Step \ref{Banach_GRIM_i}
of the \textbf{Banach GRIM} algorithm all being chosen 
equal to $1$
(cf. \textbf{Banach GRIM} \ref{Banach_GRIM_i}).
We now follow the second step of the 
\textbf{Banach GRIM} algorithm, i.e. 
\textbf{Banach GRIM} \ref{Banach_GRIM_ii}. 
For each $i \in \{1, \ldots , \n\}$ let $\tilde{a}_i := |a_i|$
and $\tilde{f}_i$ be given by $f_i$ if $a_i > 0$ and
$-f_i$ if $a_i < 0$. Evidently, for every 
$i \in \{1, \ldots ,\n\}$ we have 
$\left|\left| \tilde{f}_i \right|\right|_X = 
||f_i||_X$. Moreover, we also have that
$\tilde{a}_1 , \ldots , \tilde{a}_{\n} > 0$
and $\vph = \sum_{i=1}^{\n} \tilde{a}_i \tilde{f}_i$.
We additionally rescale $\tilde{f}_i$ for each 
$i \in \{1, \ldots, \n\}$ to have unit $X$ norm. 
That is (cf. \textbf{Banach GRIM} \ref{Banach_GRIM_ii}),
for each $i \in \{1, \ldots , \n \}$ set
$h_i := \frac{\tilde{f}_i}{||f_i||_X}$ and
$\al_i := \tilde{a}_i || f_i ||_X$. Then observe 
both that $C$ satisfies 
\beq
    \label{banach_separation_varphi_const}
        C := \sum_{i=1}^{\n} |a_i||| f_i ||_X
        =
        \sum_{i=1}^{\n} \tilde{a}_i
        || f_i ||_X 
        =
        \sum_{i=1}^{\n} \al_i,
\eeq
and, for every $i \in \{1, \ldots , \n\}$, that
$\al_i h_i = \tilde{a}_i \tilde{f}_i = a_i f_i$.
Therefore the expansion for $\vph$ in (\bI) of 
\eqref{banach_separation_const_pf} is equivalent to 
\beq
    \label{banach_sep_lemma_vph_def&bd} 
        \vph = \sum_{i=1}^{\n} \al_i h_i,
        \qquad \text{ and hence } \qquad 
        ||\vph||_X \leq 
        \sum_{i=1}^{\n} \al_i 
        || h_i ||_X 
        =
        \sum_{i=1}^{\n} \al_i
        \stackrel{
        \eqref{banach_separation_varphi_const}
        }{=} C.
\eeq
Turning our attention to steps 
\textbf{Banach GRIM} \ref{Banach_GRIM_iii} and 
\ref{Banach_GRIM_iv} of the \textbf{Banach GRIM} algorithm,
suppose that $m \in \Z_{\geq 2}$ and that the 
$m^{\text{th}}$ step of the \textbf{Banach GRIM} algorithm 
is completed without triggering the early termination criterion.
For each $l \in \{1 , \ldots , m \}$ let
$\sigma_l \in \Gamma$ be the linear functional selected 
at the $l^{\text{th}}$ step, and let $u_l \in \Span(\cf)$
be the element found via recombination 
(cf. the \emph{recombination thinning} Lemma 
\ref{Banach_recombination_lemma}) such that, 
for every $s \in \{1, \ldots , l\}$, we have 
$|\sigma_s(\vph - u_l)| \leq \th_0$
(cf. \textbf{Banach GRIM} \ref{Banach_GRIM_iii}
and \ref{Banach_GRIM_iv}).
Define $\Gamma_l := \left\{ \sigma_1 , \ldots , \sigma_l 
\right\} \subset \Gamma$.

For each $l \in \{1, \ldots , m\}$ observe that
$\min \left\{ \n , l + 1\right\} = l+1$.
Hence the \emph{recombination thinning} Lemma 
\ref{Banach_recombination_lemma} 
additionally tells us that there are non-negative 
coefficients 
$b_{l,1} , \ldots , b_{l,l+1} \geq 0$ and indices
$e_l(1), \ldots , e_l(l+1) \in \{1 , \ldots , \n\}$
for which 
\beq
    \label{banach_sep_lemma_u_n_coeff_sum}
        u_l = \sum_{s=1}^{l+1} b_{l,s}
        h_{e_l(s)} 
        \qquad \text{and} \qquad
        \sum_{s=1}^{l+1} b_{l,s} = 
        \sum_{i=1}^{\n} \al_i.
\eeq
A consequence of \eqref{banach_sep_lemma_u_n_coeff_sum}
is that
\beq
    \label{banach_sep_lemma_ul_norm_bd}
        ||u_l||_X \leq 
        \sum_{s=1}^{l+1} b_{l,s}
        \left|\left| h_{e_l(s)} \right|\right|_X
        =
        \sum_{s=1}^{l+1} b_{l,s}
        \stackrel{
        \eqref{banach_sep_lemma_u_n_coeff_sum}
        }{=}
        \sum_{i=1}^{\n} \al_i
        \stackrel{
        \eqref{banach_separation_varphi_const}
        }{=} C.
\eeq
Further, for each $l \in \{1, \ldots , m\}$, let
\beq
    \label{Omega_l_def_sep_lemma_pf}
        \Omega_l :=
        \Reach_{||\cdot||_{l^(\R^l)}}
        \left( \Gamma_l , 2C , \th , \th_0 \right)
        \stackrel{\text{def}}{=}
        \bigcup_{0 \leq t \leq \frac{\th}{\th_0}}
        \ovSpan_{||\cdot||_{l^(\R^l)}} 
        \left( \Gamma_l , t \right)_{
        \frac{\th - t \th_0}{2C}
        }
\eeq 
where we use our notation conventions
from Section \ref{notation}. 
With our notation fixed, we turn our attention to 
verifying the claims made in
\eqref{banach_separation_estimate},
\eqref{Omega_l_est_sep_lemma}, and
\eqref{banach_separation_dist}.

We begin by noting that \eqref{Omega_l_est_sep_lemma} is 
an immediate consequence of appealing to Lemma 
\ref{banach_robust_lemma} with $l$ and $\Gamma_l \subset X^{\ast}$ 
playing the roles of integer $d$ and 
finite subset $L \subset X^{\ast}$ there respectively.
Indeed by doing so we may conclude that 
(cf. \eqref{banach_robust_lemma_conc})
\beq
    \label{banach_robust_lemma_conc_implication}
        \text{for every } \sigma \in 
        \Reach_{||\cdot||_{l^1(\R^l)}}
        \left( \Gamma_l , 2C , \th, \th_0 \right)
        \stackrel{
        \eqref{Omega_l_def_sep_lemma_pf}
        }{=:}
        \Omega_l
        \quad \text{we have} \quad
        |\sigma(\vph - u)| \leq \th
\eeq
as claimed in \eqref{Omega_l_est_sep_lemma}.

To establish \eqref{banach_separation_dist},
fix $l \in \{2 , \ldots , m\}$ and consider $\sigma_l \in \Gamma$.
A consequence of the \textbf{Banach GRIM} algorithm 
completing the $l^{\text{th}}$ step without terminating
is that (cf. \textbf{Banach GRIM} algorithm steps 
\ref{Banach_GRIM_iii} and \ref{Banach_GRIM_iv})
\beq
    \label{banach_dist_fact_one}
        |\sigma_l(\vph - u_{l-1})| > \th.
\eeq
However, since we have established that 
\eqref{Omega_l_est_sep_lemma} is true, we know that if
$\sigma \in \Omega_{l-1}$
then $| \sigma(\vph-u_{l-1}) | \leq \th$.
Consequently, \eqref{banach_dist_fact_one} tells us that 
$\sigma_l \notin \Omega_{l-1}$ 
as claimed in \eqref{banach_separation_dist}.

To establish \eqref{banach_separation_estimate},
fix $l \in \{1, \ldots , m\}$ and consider any 
$\sigma \in X^{\ast}$.
Then we may estimate that
\beq
    \label{banach_sep_lemma_un_dist_est_obs}
        \left| \sigma(\varphi - u_l) \right|
        \leq
        ||\sigma||_{X^{\ast}} 
        || \vph - u_l||_X
        \stackrel{
        \eqref{banach_sep_lemma_vph_def&bd} ~\&~
        \eqref{banach_sep_lemma_ul_norm_bd} 
        }{\leq} 
        2C
        || \sigma ||_{X^{\ast}}.
\eeq
Alternatively, let
$\rho \in \Omega_l$ and use that via 
\eqref{Omega_l_est_sep_lemma}
$|\rho(\vph - u_l)| \leq \th$ to compute that
\beq
    \label{banach_sep_lemma_un_dist_est_obs_2}
        |\sigma(\varphi-u_l)|
        \leq 
        || \sigma - \rho ||_{X^{\ast}} 
        || \vph - u_l ||_X  + \th
        \stackrel{
        \eqref{banach_sep_lemma_vph_def&bd} ~\&~
        \eqref{banach_sep_lemma_ul_norm_bd}
        }{\leq} 
        2C || \sigma - \rho ||_{X^{\ast}}
        + \th.
\eeq
Taking the infimum over $\rho \in \Omega_l$
in \eqref{banach_sep_lemma_un_dist_est_obs_2} yields
\beq
    \label{banach_sep_lemma_un_dist_est}
        | \sigma ( \varphi - u_l) |
        \leq
        2C \dist_{X^{\ast}} ( \sigma , \Omega_l )
        + \th.
\eeq
Together, \eqref{banach_sep_lemma_un_dist_est_obs}
and \eqref{banach_sep_lemma_un_dist_est} yield 
that for any $\sigma \in X^{\ast}$ we have 
\beq
    \label{banach_sep_lemma_un_dist_est_gotten}
        | \sigma ( \vph - u_l)|
        \leq 
        \min \left\{ 2C || \sigma ||_{X^{\ast}} ~,~
        2C \dist_{X^{\ast}} ( \sigma , \Omega_l )
        + \th \right\}
\eeq
as claimed in \eqref{banach_separation_estimate}.
This completes the proof of Lemma 
\ref{dist_between_interp_functionals}.
\end{proof}
\vskip 4pt
\noindent
We can use Lemma \ref{dist_between_interp_functionals}
to establish an upper bound on the maximum number of steps
the \textbf{Banach GRIM} algorithm can run for before 
terminating.
The precise statement is the following lemma.

\begin{lemma}[\textbf{Banach GRIM Number of 
Steps Bound}]
\label{banach_GRIM_step_num_bound_lemma}
Assume $X$ is a Banach space with dual-space $X^{\ast}$.
Let $\th > \th_0 \geq 0$, and $\n , \Lambda \in \Z_{\geq 1}$
such that 
$M := \min \left\{ \n - 1 , \Lambda \right\} \geq 1$.
Suppose that 
$\cf := \{ f_1 , \ldots , f_{\n} \} \subset X \setminus \{0\}$,
and that $\Gamma \subset X^{\ast}$ is finite with cardinality 
$\Lambda$. 
Let $a_1 , \ldots , a_{\n} \in \R \setminus \{0\}$ 
and define $\vph \in \Span(\f)$ and $C > 0$ by
\beq
    \label{banach_step_bound_lemma__varphi_const}
        (\bI) \quad \vph := 
        \sum_{i=1}^{\n} a_i f_i
        \qquad \text{and} \qquad
        (\bII) \quad C := \sum_{i=1}^{\n} 
        |a_i||| f_i ||_X > 0.
\eeq
Then there is a non-negative integer
$N = N(\Gamma,C,\th,\th_0) \in \Z_{\geq 0}$, given by 
\beq
    \label{banach_step_bound_lemma_r_packing}
        N :=
        \max \left\{ d \in \Z ~:~
        \begin{array}{c}
            \text{There exists }
            \sigma_1 , \ldots , \sigma_d
            \in \Gamma 
            \text{ such that for every } 
            j \in \{1, \ldots , d-1\} \\ 
            \text{we have }
            \sigma_{j+1} \notin 
            \Reach_{||\cdot||_{l^1(\R^{j})}} 
            \left( \{\sigma_1, \ldots , \sigma_{j}\} , 
            2C , \th , \th_0 \right) 
        \end{array}
        \right\},
\eeq
where we use our notation conventions from Section \ref{notation}
(see also Remark \ref{section2_notation_recap}),
for which the following is true.

Suppose $N \leq M := \min \left\{ \n - 1 , \Lambda \right\}$ and
consider applying the \textbf{Banach GRIM} algorithm 
to approximate $\vph$ on $\Gamma$ with
$\th$ as the accuracy threshold, $\th_0$ as
the acceptable recombination error bound, 
$M$ as the maximum number of steps, 
$s_1 = \ldots = s_M = 1$ as the shuffle numbers, and with 
the integers $k_1 , \ldots , k_M$ in Step \ref{Banach_GRIM_i}
of the \textbf{Banach GRIM} algorithm all being chosen 
equal to $1$
(cf. \textbf{Banach GRIM} \ref{Banach_GRIM_i}).
Then, after at most $N$ steps the algorithm terminates.
That is, there is some integer $n \in \{ 1 , \ldots , N \}$
for which the \textbf{Banach GRIM} algorithm terminates 
after completing $n$ steps.
Consequently, there are coefficients
$c_1 , \ldots , c_{n+1} \in \R$ and 
indices 
$e(1) , \ldots , e(n+1) \in \{1, \ldots , \n\}$
with
\beq
    \label{banach_step_bound_lemma_coeff_sum}
        \sum_{s=1}^{n+1} \left|c_s\right| 
        || f_{e(s)}||_X
        = C,
\eeq
and such that the element $u \in \Span(\cf)$ 
defined by
\beq
    \label{banach_step_bound_lemma_uM_acc}
        u := \sum_{s=1}^{n+1}
        c_s f_{e(s)} 
        \qquad \text{satisfies, for every } 
        \sigma \in \Gamma, \text{ that}
        \qquad 
        |\sigma(\vph - u)| \leq \th. 
\eeq
In fact there are linear functionals 
$\Gamma_n = \left\{ \sigma_1 , \ldots , \sigma_n \right\}
\subset \Gamma$ such that
\beq
    \label{banach_step_bound_lemma_acc_est}
        \text{for every }
        \sigma \in 
        \Omega_n := 
        \Reach_{||\cdot||_{l^1(\R^n)}} 
        \left( \Gamma_n , 2C , \th , \th_0 \right)
        \quad \text{we have} \quad
        \left| \sigma(\vph-u) \right| \leq \th.
\eeq
Moreover, if the coefficients 
$a_1 , \ldots , a_{\n} \in \R \setminus \{0\}$ 
corresponding to $\vph$ 
(cf. (\bI) of \eqref{banach_step_bound_lemma__varphi_const})
are all positive (i.e. $a_1 , \ldots , a_{\n} > 0$) then 
the coefficients $c_{1} , \ldots , c_{n+1} \in \R$ 
corresponding to $u$ 
(cf. \eqref{banach_step_bound_lemma_uM_acc}) are 
all non-negative (i.e. $c_1 , \ldots , c_{n+1} \geq 0$).
\end{lemma}

\begin{remark}
Recall that the \textbf{Banach GRIM} algorithm is 
guaranteed to terminate after $M$ steps. 
Consequently, the restriction to the case that 
$N  \leq M$ is sensible since it is only in this 
case that terminating after no more than $N$ steps 
is a non-trivial statement.

If $N < M$ then  
Lemma \ref{banach_GRIM_step_num_bound_lemma}
guarantees that the 
\textbf{Banach GRIM} algorithm will find an 
approximation $u \in \Span(\cf)$ of $\varphi$
that is a linear combination of \emph{less}
than $\n$ of the elements $f_1 , \ldots , f_{\n}$
but is within $\th$ of $\vph$ throughout 
$\Gamma$ in the sense that 
$|\sigma(\vph - u)| \leq \th$ for every
$\sigma \in \Gamma$.
\end{remark}

\begin{remark}
By invoking Lemma 
\ref{dist_between_interp_functionals},
and using the same notation as in Lemma 
\ref{banach_GRIM_step_num_bound_lemma},
we can additionally conclude that for every
$\sigma \in X^{\ast}$ we have  
\beq
    \label{Banach_GRIM_step_num_lemma_rmk_gen_est}
        |\sigma(\vph - u)| 
        \leq 
        \min \left\{ 2C || \sigma ||_{X^{\ast}} ~,~
        2C \dist_{X^{\ast}} \left( \sigma, 
        \Omega_n \right) + \th
        \right\}.
\eeq 
\end{remark}

\begin{proof}[Proof of Lemma 
\ref{banach_GRIM_step_num_bound_lemma}]
Assume $X$ is a Banach space with dual-space $X^{\ast}$.
Let $\th > \th_0 \geq 0$, and $\n , \Lambda \in \Z_{\geq 1}$
such that 
$M := \min \left\{ \n - 1 , \Lambda \right\} \geq 1$.
Suppose that 
$\cf := \{ f_1 , \ldots , f_{\n} \} \subset X \setminus \{0\}$,
and that $\Gamma \subset X^{\ast}$ is finite with cardinality 
$\Lambda$. 
Let $a_1 , \ldots , a_{\n} \in \R \setminus \{0\}$ 
and define $\vph \in \Span(\f)$ and $C > 0$ by
(\bI) and (\bII) of 
\eqref{banach_step_bound_lemma__varphi_const}
respectively. That is
\beq
    \label{banach_step_bound_lemma__varphi_const_pf}
        (\bI) \quad \vph := 
        \sum_{i=1}^{\n} a_i f_i
        \qquad \text{and} \qquad
        (\bII) \quad C := \sum_{i=1}^{\n} 
        |a_i||| f_i ||_X > 0.
\eeq
With a view to later applying \textbf{Banach GRIM} 
to approximate $\vph$ on $\Gamma$, for each 
$i \in \{1, \ldots , \n\}$ let $\tilde{a}_i := |a_i|$
and $\tilde{f}_i$ be given by $f_i$ if $a_i > 0$ and
$-f_i$ if $a_i < 0$. For every $i \in \{1, \ldots ,\n\}$
we evidently have 
$\left|\left| \tilde{f}_i \right|\right|_X = 
||f_i||_X$. Moreover, we also have that
$\tilde{a}_1 , \ldots , \tilde{a}_{\n} > 0$
and $\vph = \sum_{i=1}^{\n} \tilde{a}_i \tilde{f}_i$.
Further, we rescale $\tilde{f}_i$ for each 
$i \in \{1, \ldots, \n\}$ to have unit $X$ norm. 
That is (cf. \textbf{Banach GRIM} \ref{Banach_GRIM_ii}),
for each $i \in \{1, \ldots , \n \}$ set
$h_i := \frac{\tilde{f}_i}{||f_i||_X}$ and
$\al_i := \tilde{a}_i || f_i ||_X$. Observe 
both that $C$ satisfies 
\beq
    \label{banach_step_num_varphi_const}
        C := \sum_{i=1}^{\n} |a_i||| f_i ||_X
        =
        \sum_{i=1}^{\n} \tilde{a}_i
        || f_i ||_X 
        =
        \sum_{i=1}^{\n} \al_i,
\eeq
and, for every $i \in \{1, \ldots , \n\}$, that
$\al_i h_i = \tilde{a}_i \tilde{f}_i = a_i f_i$.
Therefore the expansion for $\vph$ in (\bI) of 
\eqref{banach_separation_const} is equivalent to 
\beq
    \label{banach_step_num_vph_def&bd} 
        \vph = \sum_{i=1}^{\n} \al_i h_i,
        \qquad \text{ and hence } \qquad 
        ||\vph||_X \leq 
        \sum_{i=1}^{\n} \al_i 
        || h_i ||_X 
        =
        \sum_{i=1}^{\n} \al_i
        \stackrel{
        \eqref{banach_step_num_varphi_const}
        }{=} C.
\eeq
Define a non-negative integer
$N = N(\Gamma,C,\th,\th_0) \in \Z_{\geq 0}$ by
\beq
    \label{banach_step_bound_lemma_r_packing_pf}
        N := 
        \max \left\{ d \in \Z ~:~
        \begin{array}{c}
            \text{There exists } 
            \sigma_1 , \ldots , \sigma_d 
            \in \Gamma 
            \text{ such that for every } 
            j \in \{1, \ldots , d-1\} \\ 
            \text{we have }
            \sigma_{j+1} \notin 
            \Reach_{||\cdot||_{l^1(\R^{j})}} 
            \left( \{\sigma_1, \ldots , \sigma_{j}\} , 
            2C , \th , \th_0 \right) 
        \end{array}
        \right\}.
\eeq
Suppose $N \leq M := \min \left\{ \n - 1 , \Lambda \right\}$ and
consider applying the \textbf{Banach GRIM} algorithm 
to approximate $\vph$ on $\Gamma$ with
$\th$ as the accuracy threshold, $\th_0$ as
the acceptable recombination error bound, 
$M$ as the maximum number of steps, 
$s_1 = \ldots = s_M = 1$ as the shuffle numbers, and with 
the integers $k_1 , \ldots , k_M$ in Step \ref{Banach_GRIM_i}
of the \textbf{Banach GRIM} algorithm all being chosen 
equal to $1$
(cf. \textbf{Banach GRIM} \ref{Banach_GRIM_i}).

We first prove that the algorithm terminates after
at most $N$ steps have been completed.
Let $\sigma_1 \in \Gamma$ be the linear functional 
chosen in the first step 
(cf. \textbf{Banach GRIM} \ref{Banach_GRIM_i}), and 
$u_1 \in \Span(\cf)$ be the approximation found via 
recombination (cf. the \emph{recombination thinning} 
Lemma \ref{Banach_recombination_lemma})
satisfying, in particular, that 
$|\sigma_1(\vph-u_1)| \leq \th_0$.
Define $\Gamma_1 := \{\sigma_1\} \subset \Gamma$.
We conclude, via Lemma \ref{dist_between_interp_functionals}
(cf. \eqref{Omega_l_est_sep_lemma}), that for every 
$\sigma \in \Reach_{||\cdot||_{l^1(\R)}} 
\left( \Gamma_1 , 2C , \th , \th_0 \right)$ we have 
$|\sigma(\vph-u_1)| \leq \th$.

If $N = 1$ then \eqref{banach_step_bound_lemma_r_packing_pf} 
means there is no $\sigma \in \Gamma$ 
for which $\sigma \notin \Reach_{||\cdot||_{l^1(\R)}} 
\left( \Gamma_1 , 2C , \th , \th_0 \right)$ 
Consequently, 
$\Gamma \cap \Reach_{||\cdot||_{l^1(\R)}} 
\left( \Gamma_1 , 2C , \th , \th_0 \right) = \Gamma$, 
and so we have established that for every 
$\sigma \in \Gamma$ we have $|\sigma(\vph-u_1)| \leq \th$.
Recalling \textbf{Banach GRIM} \ref{Banach_GRIM_iv}, this means
that algorithm terminates before step $2$ is completed.
Hence the algorithm terminates after completing $N=1$ steps.

If $N \geq 2$ then we note that if the stopping criterion 
in \textbf{Banach GRIM} \ref{Banach_GRIM_iv} is triggered 
at the start of step $m \in \{2, \ldots , N\}$ then we 
evidently have that the algorithm has terminated after 
carrying out no more than $N$ steps. 
Consequently, we need only deal with the case in which 
the algorithm reaches and carries out step $N$ without terminating.
In this case, we claim that the algorithm terminates after completing step $N$, i.e. that
the termination criterion at the start of step $N+1$ is triggered.

Before proving this we fix some notation.
Recalling \textbf{Banach GRIM} \ref{Banach_GRIM_iv}, 
for $l \in \{2 , \ldots , N\}$ let $\sigma_l \in \Gamma$
denote the new linear functional selected at step $l$, 
define 
$\Gamma_l := \left\{ \sigma_1 , \ldots , \sigma_l \right\}$,
and let $u_l \in \Span(\cf)$ be the approximation 
found by recombination (cf. the \emph{recombination thinning} 
Lemma \ref{Banach_recombination_lemma}) satisfying, 
for every $s \in \{1, \ldots , l\}$, 
that $|\sigma_s ( \vph - u_l ) | \leq \th_0$.

By appealing to Lemma \ref{dist_between_interp_functionals}
we deduce both  that for any $l \in \{2 , \ldots , N \}$ we have 
(cf. \eqref{banach_separation_dist})
\beq
    \label{sigma_l_not_in_prior_reach}
        \sigma_l \notin 
        \Reach_{||\cdot||_{l^(\R^{l-1})}}
        \left( \Gamma_{l-1} , 2C , \th, \th_0 \right),
\eeq
and that for any 
$\sigma \in \Reach_{||\cdot||_{l^1(\R^l)}} 
\left( \Gamma_l , 2C , \th , \th_0 \right)$ we have
(cf. \eqref{Omega_l_est_sep_lemma})
\beq
    \label{est_on_Gamma_n_reach}
        |\sigma(\vph - u_l)| \leq \th.
\eeq
Consider step $N+1$ of the algorithm at which we 
examine 
$K := \max \left\{ 
| \sigma (\vph - u_N) | : \sigma \in \Gamma \right\}$.
If $K \leq \th$ then the algorithm terminates without
carrying out step $N+1$, and thus has terminated after 
carrying out $N$ steps as claimed. 

Assume that $K > \th$ so that 
$\sigma_{N+1} := \argmax \left\{ 
| \sigma (\vph - u_N) | : \sigma \in \Gamma \right\}$ 
satisfies that $|\sigma_{N+1}(\vph-u_N)| > \th$.
It follows from \eqref{est_on_Gamma_n_reach} 
for $l := N$ that 
$\sigma_{N+1} \notin \Reach_{||\cdot||_{l^1(\R^N)}} 
\left( \Gamma_N , 2C , \th , \th_0 \right)$.
But it then follows from this and 
\eqref{sigma_l_not_in_prior_reach} that 
$\sigma_1 , \ldots , \sigma_{N+1} \in \Gamma$ satisfy
that, for every $j \in \{1 , \ldots , N\}$, that
$\sigma_{j} \notin \Reach_{||\cdot||_{l^1(\R^{j-1})}} 
\left( \Gamma_{j-1} , 2C , \th , \th_0 \right)$.
In which case \eqref{banach_step_bound_lemma_r_packing_pf}
yields that 
$$ N \stackrel{
        \eqref{banach_step_bound_lemma_r_packing_pf}
        }{:=} 
    \max \left\{ d \in \Z ~:~
        \begin{array}{c}
            \text{There exists }
            \sigma_1 , \ldots , \sigma_d
            \in \Gamma 
            \text{ such that for every } 
            j \in \{1, \ldots , d-1\} \\ 
            \text{we have }
            \sigma_{j+1} \notin 
            \Reach_{||\cdot||_{l^1(\R^{j})}} 
            \left( \{\sigma_1, \ldots , \sigma_{j}\} , 
            2C , \th , \th_0 \right) 
        \end{array}
        \right\}
        \geq  N+1$$
which is evidently a contradiction.
Thus we must have that $K \leq \th$, and hence that
the algorithm must terminate before carrying out
step $N+1$.

Having established the claimed upper bound on the number of 
steps before the \textbf{Banach GRIM} algorithm terminates,
we turn our attention to the properties claimed for 
approximation returned after the algorithm terminates.
Let $n \in \{1, \ldots , N\}$ be the integer for which
the \textbf{Banach GRIM} algorithm terminates after step $n$.
Recalling \textbf{Banach GRIM} 
\ref{Banach_GRIM_iii} and \ref{Banach_GRIM_iv},
let $\Gamma_n = \{ \sigma_1 , \ldots , \sigma_n \} 
\subset \Gamma$ be the $n$ linear functionals selected by 
the end of the $n^{\text{th}}$ step, and let $u \in \Span(\cf)$
denote the approximation found via recombination 
(cf. the \emph{recombination thinning}
Lemma \ref{Banach_recombination_lemma}) satisfying, 
for every $s \in \{1, \ldots , n\}$, that 
$|\sigma_s(\vph-u)| \leq \th_0$.

A consequence of Lemma \ref{dist_between_interp_functionals} 
is that for every 
$\sigma \in \Reach_{||\cdot||_{l^1(\R^n)}} 
\left( \Gamma_n , 2C , \th , \th_0 \right)$ we have 
(cf. \eqref{Omega_l_est_sep_lemma})
\beq
    \label{}
        | \sigma(\vph - u)| \leq \th
\eeq
as claimed in \eqref{banach_step_bound_lemma_acc_est}.
Moreover, since we have established that the algorithm 
terminates by triggering the stopping criterion after
completing step $n$, we have, for every $\sigma \in \Gamma$,
that $|\sigma(\vph-u)| \leq \th$ as claimed in the 
second part of \eqref{banach_step_bound_lemma_uM_acc}.

To establish \eqref{banach_step_bound_lemma_coeff_sum}
and the first part of 
\eqref{banach_step_bound_lemma_uM_acc}, first note that
$\min \left\{ \n , n \right\} = n$.
Thus Lemma \ref{Banach_recombination_lemma} additionally 
tells us that recombination returns non-negative coefficients 
$b_{1}, \ldots , b_{n+1} \geq 0$, with
\beq
    \label{banach_step_bound_lemma_um_coeff_sum_pf}
        \sum_{s=1}^{n+1} b_{s}
        =
        \sum_{i=1}^{\n} \al_i
        \stackrel{
        \eqref{banach_step_num_varphi_const}
        }{=}
        C,
\eeq
and indices 
$e(1) , \ldots , e(n+1) \in \{1, \ldots , \n\}$
for which 
\beq
    \label{banach_step_bound_lemma_um_def_pf}
        u = \sum_{s=1}^{n+1} b_{s}
        h_{e(s)}
        =
        \sum_{s=1}^{n+1} 
        \frac{b_{s}}{\left|\left|
        f_{e(s)}\right|\right|_X}
        \tilde{f}_{e(s)}.
\eeq 
For each $s \in \{1 , \ldots , n+1\}$,
we define $c_{s} := \frac{b_{s}}
{\left|\left| f_{e(s)}
\right|\right|_{X}}$ if $\tilde{f}_{e(s)} = f_{e(s)}$
(which we recall is the case if $a_{e(s)} > 0$) 
and $c_{s} := -\frac{b_{s}}
{\left|\left| f_{e(s)}
\right|\right|_{X}}$ if $\tilde{f}_{e(s)} = -f_{e(s)}$
(which we recall is the case if $a_{e(s)} < 0$).
Then \eqref{banach_step_bound_lemma_um_def_pf}
gives the expansion for $u \in \Span(\cf) \subset X$
in terms of the elements $f_1 , \ldots , f_{\n}$
claimed in the first part of
\eqref{banach_step_bound_lemma_uM_acc}.
Moreover, from 
\eqref{banach_step_bound_lemma_um_coeff_sum_pf} 
we have that
\beq
    \label{banach_step_bound_lemma_coeff_sum_got}
        \sum_{s=1}^{n+1} \left|c_{s}\right|
        \left|\left| f_{e(s)}
        \right|\right|_{X}
        =
        \sum_{s=1}^{n+1} b_{s}
        \stackrel{
        \eqref{banach_step_bound_lemma_um_coeff_sum_pf}
        }{=}
        C
\eeq
as claimed in 
\eqref{banach_step_bound_lemma_coeff_sum}.

It remains only to prove that if the coefficients
$a_1 , \ldots , a_{\n} \in \R \setminus \{0\}$ are all positive 
(i.e. $a_1 , \ldots , a_{\n} > 0$), then the resulting
coefficients $c_{1} , \ldots , c_{n+1} \in \R$ are
all non-negative (i.e. $c_{1} , \ldots , c_{n+1} \geq 0$).
To see this, observe that if $a_1 , \ldots , a_{\n} > 0$
then, for every $i \in \{1 ,\ldots , \n\}$, we have that
$\tilde{f}_i = f_i$. Consequently, for every 
$s \in \{1 , \ldots , n+1\}$ we have that 
$\tilde{f}_{e(s)} = f_{e(s)}$, and so by definition 
we have $c_{s} = \frac{b_{s}}{\left|\left| 
f_{e(s)} \right|\right|_X}$. Since $b_{s} \geq 0$, 
it follows that $c_{s} \geq 0$.
This completes the proof of Lemma
\ref{banach_GRIM_step_num_bound_lemma}.
\end{proof}

\subsection{Proof of Main Theoretical Result}
\label{Main_Theoretical_Result_Proof}
In this subsection we prove the 
\textbf{Banach GRIM Convergence} Theorem 
\ref{banach_conv} by
combining Lemmas \ref{banach_robust_lemma}, 
\ref{dist_between_interp_functionals}, and
\ref{banach_GRIM_step_num_bound_lemma}.

\begin{proof}[Proof of Theorem \ref{banach_conv}]
Assume $X$ is a Banach space with dual-space $X^{\ast}$.
Let $\ep > \ep_0 \geq 0$.
Let $\n , \Lambda \in \Z_{\geq 1}$ and set 
$M := \min \left\{ \n - 1 , \Lambda \right\}$.
Let
$\cf := \{ f_1 , \ldots , f_{\n} \} \subset X \setminus \{0\}$ 
and $\Sigma \subset X^{\ast}$ be a finite subset with cardinality 
$\Lambda$.
Let $a_1 , \ldots , a_{\n} \in \R \setminus \{0\}$ and define 
$\vph \in \Span(\cf)$ and $C > 0$ by
\beq
    \label{banach_conv_thm_varphi_C_pf}
        (\bI) \quad \vph := \sum_{i=1}^{\n} 
        a_i f_i
        \qquad \text{and} \qquad
        (\bII) \quad C := \sum_{i=1}^{\n}
        |a_i| || f_i ||_X > 0.
\eeq
Define a non-negative integer 
$N = N(\Sigma,C,\ep,\ep_0)\in \Z_{\geq 0}$,
given by 
\beq
    \label{banach_conv_thm_max_step_num_pf}
        N := 
        \max \left\{ d \in \Z ~:~
        \begin{array}{c}
            \text{There exists }
            \sigma_1 , \ldots , \sigma_d
            \in \Sigma 
            \text{ such that for every } 
            j \in \{1, \ldots , d-1\} \\ 
            \text{we have }
            \sigma_{j+1} \notin 
            \Reach_{||\cdot||_{l^1(\R^{j})}} 
            \left( \{\sigma_1, \ldots , \sigma_{j}\} , 
            2C , \ep , \ep_0 \right) 
        \end{array}
        \right\}.
\eeq
Suppose $N \leq M := \min \left\{ \n - 1 , \Lambda \right\}$ and
consider applying the \textbf{Banach GRIM} algorithm 
to approximate $\vph$ on $\Sigma$ with
$\ep$ as the accuracy threshold, $\ep_0$ as
the acceptable recombination error bound, 
$M$ as the maximum number of steps, 
$s_1 = \ldots = s_M = 1$ as the shuffle numbers, and with 
the integers $k_1 , \ldots , k_M$ in Step \ref{Banach_GRIM_i}
of the \textbf{Banach GRIM} algorithm all being chosen 
equal to $1$
(cf. \textbf{Banach GRIM} \ref{Banach_GRIM_i}).
By appealing to Lemma 
\ref{banach_GRIM_step_num_bound_lemma}, 
with $\ep$, $\ep_0$ and $\Sigma$ here as 
the $\th$, $\th_0$ and $\Gamma$ there respectively,
to conclude that there is some integer 
$n \in \{1, \ldots , N  \}$ for which the algorithm 
terminates after step $n$. 
Thus the \textbf{Banach GRIM} algorithm terminates
after completing, at most, $N$ steps as claimed.

Lemma \ref{banach_GRIM_step_num_bound_lemma} additionally 
tells us that there are coefficients 
$c_1 , \ldots , c_{n+1} \in \R$ and indices
$e(1) , \ldots , e(n+1) \in \{1, \ldots , \n\}$
with (cf. \eqref{banach_step_bound_lemma_coeff_sum})
\beq
    \label{banach_conv_thm_coeff_sum_done}
        \sum_{s=1}^{n+1} |c_s| 
        || f_{e(s)}||_X = C
\eeq
and such that the element $u \in \Span(\cf)$ 
defined by (cf. \eqref{banach_step_bound_lemma_uM_acc})
\beq
    \label{banach_conv_thm_vph_approx_done}
        u := \sum_{s=1}^{n+1} c_s f_{e(s)} 
        \quad \text{satisfies, for every }
        \sigma \in \Sigma \text{ that} \quad 
        |\sigma(\vph-u)| \leq \ep.
\eeq
Observe that \eqref{banach_conv_thm_coeff_sum_done} 
is precisely the claim 
\eqref{banach_conv_thm_coeff_sum}, whilst 
\eqref{banach_conv_thm_vph_approx_done} 
is precisely the claim
\eqref{banach_conv_thm_uM_def_&_acc}.

The final consequence of 
Lemma \ref{banach_GRIM_step_num_bound_lemma} that we 
note is that if the coefficients
$a_1 , \ldots , a_{\n} \in \R \setminus \{0\}$
associated to $\vph$ (cf. (\bI) of 
\eqref{banach_conv_thm_varphi_C_pf}) 
are all positive (i.e. $a_1 , \ldots , a_{\n} > 0$)
then the coefficients $c_{1} , \ldots , c_{n+1} \in \R$
associated with $u$ 
(cf. \eqref{banach_conv_thm_vph_approx_done})
are all non-negative (i.e. 
$c_{1} , \ldots , c_{n+1} \geq 0$. 
This is precisely the preservation of non-negative 
coefficients claimed in Theorem \ref{banach_conv}.

It only remains to verify the claim made in 
\eqref{banach_conv_thm_robust_est}.
For this purpose let $A > 1$ and define 
\beq
    \label{Banach_GRIM_conv_extend_to_set_pf}
        \Omega := 
        \Reach_{||\cdot||_{l^1(\R^{\Lambda})}}
        \left( \Sigma , 2C , A\ep , \ep \right)
        \stackrel{\text{def}}{=}
        \bigcup_{0 \leq t \leq A}
        \ovSpan_{||\cdot||_{l^1(\R^{\Lambda})}}
        \left( \Sigma , t \right)_{
        \frac{(A-1)\ep}{2C}
        }.
\eeq
Then observe that $\Sigma \subset X^{\ast}$ is a finite
subset of cardinality $\Lambda$ and that  
\eqref{banach_conv_thm_coeff_sum_done}
and \eqref{banach_conv_thm_vph_approx_done} mean that
$u \in \Span(\cf)$ 
satisfies both that $||u||_X \leq C$
and that, for every $\sigma \in \Sigma$, we have 
$|\sigma(\vph -u)| \leq \ep$. 
Therefore we may apply Lemma \ref{banach_robust_lemma},
with the integer $d$, the finite subset $L \subset X^{\ast}$, 
and the real numbers $\th > \th_0 \geq 0$
of that result as the integer $\Lambda$, the finite 
subset $\Sigma$, and the real numbers 
$A\ep > \ep > 0$ here, to deduce that 
(cf. \eqref{banach_robust_lemma_conc}) for every 
$\sigma \in \Omega $ we have $|\sigma(\vph-u)| \leq A \ep$.
Recalling the definition of the set $\Omega$ in 
\eqref{Banach_GRIM_conv_extend_to_set_pf}, 
this is precisely the estimate claimed in 
\eqref{banach_conv_thm_robust_est}. 
This completes the proof of Theorem \ref{banach_conv}.
\end{proof}

\section{Numerical Examples}
\label{numerical}
In this section we compare GRIM with several existing 
techniques on two different reduction tasks.
The first task we consider is motivated by an example 
appearing in Section 4 of the work \cite{MMPY15} by 
Maday et al. concerning GEIM.

We consider the Hilbert space $X := L^2(0,1)$, 
and given $(a,b) \in [0.01,24.9] \times [0,15]$ we define 
$f_{a,b} \in X$ by 
\beq
    \label{Maday_L2_task_fab}
        f_{a,b}(x) := \frac{1}{\sqrt{1 + (25 + a \cos (bx) )x^2}}.
\eeq
For a chosen integer $\n \in \Z_{\geq 1}$ we consider the 
collection $\cf \subset X$ defined by
\beq
    \label{Maday_L2_task_cf}
        \cf := \left\{ f_{a,b} :
        (a,b) \in [0.01,24.9]_{\n} \times [0,15]_{\n}
        \right\}.
\eeq 
Here, for real numbers $c,d \in \R$, we use the notation 
$[c,d]_{\n}$ to denote a partition of the interval $[c,d]$
by $\n$ equally spaced points.

For a chosen $M \in \Z_{\geq 1}$ and
$s > 0$, we fix an equally spaced partition 
$y_1 , \ldots , y_M \in [0,1]$ and consider the collection
$\Sigma = 
\left\{ \sigma_k : k \in \{1 , \ldots , M\} \right\}
\subset X^{\ast}$
where, for $ k \in \{1, \ldots , M\}$ and $\psi \in X$, 
\beq
    \label{Maday_L2_task_functionals}
        \sigma_k(\psi) := 
        \fint_{0}^{1} \psi(x) d\rho_k(x) 
        \qquad \text{with} \qquad 
        d\rho_k(x) = e^{ 
        -\frac{(x-y_k)^2}{2s^2}
        }dx.
\eeq 
Finally, we fix the choice $a_1 = \ldots = a_{\n} = 1$
for the coefficients.
The task is to approximate the function 
$\vph := \sum_{f \in \cf} f$ over the collection 
$\Sigma$ of linear functionals.

For the tests we fix the values $M := 1000$, 
$s := 5 \times 10^{-4}$, and consider each of the values 
$N = 20, 25, 30$ individually. 
We compare our implementation of GRIM against our own coded
implementation of GEIM \cite{MM13,MMT14,MMPY15} 
(which, in particular, makes use of the recursive relation 
established in \cite{MMPY15})
and the Scikit-learn implementation
of LASSO \cite{BBCDDGGMPPPPTVVW11}.
The results are summarised in Table 
\ref{Maday_L2_Task_Results} below.
For each approximation $u$ we record the $L^2$-norm of the 
difference $\vph - u$ and the $\sup$-norm of the difference
$\vph - u$ over $\Sigma$, i.e. the values 
$\left( \int_0^1 (\vph(x) - u(x))^2 dx \right)^{1/2}$ and 
$\max \left\{ |\sigma(\vph-u)| : \sigma \in \Sigma \right\}$).

\begin{table}[H]
	\centering
	\begin{tabular}{c|c|c|c}
		~& {\bf{GRIM}}& {\bf{GEIM}} & {\bf{LASSO}} \\ 
		\hline	
        $\n = 20$ & 
        $\begin{aligned}
            &16 ~\text{non-zero weights} \\
            &L^2-\text{norm} = 0.15  \\
            &\sup -\text{norm} = 0.49
        \end{aligned}$ 
        & 
        $\begin{aligned}
            &20 ~\text{non-zero weights} \\
            &L^2-\text{norm} = 0.15  \\
            &\sup -\text{norm} = 0.64
        \end{aligned}$
        & 
        $\begin{aligned}
            &90 ~\text{non-zero weights} \\
            &L^2-\text{norm} = 0.19  \\
            &\sup -\text{norm} = 0.66
        \end{aligned}$ \\
        \hline
        $\n = 25$ & 
        $\begin{aligned}
            &20 ~\text{non-zero weights} \\
            &L^2-\text{norm} = 0.04  \\
            &\sup -\text{norm} = 0.16
        \end{aligned}$ 
        & 
        $\begin{aligned}
            &27 ~\text{non-zero weights} \\
            &L^2-\text{norm} = 0.04  \\
            &\sup -\text{norm} = 0.26
        \end{aligned}$
        & 
        $\begin{aligned}
            &135 ~\text{non-zero weights} \\
            &L^2-\text{norm} = 0.30  \\
            &\sup -\text{norm} = 1.04
        \end{aligned}$ \\
        \hline
        $\n = 30$ & 
        $\begin{aligned}
            &19 ~\text{non-zero weights} \\
            &L^2-\text{norm} = 0.07  \\
            &\sup -\text{norm} = 0.23
        \end{aligned}$ 
        & 
        $\begin{aligned}
            &24 ~\text{non-zero weights} \\
            &L^2-\text{norm} = 0.15  \\
            &\sup -\text{norm} = 0.72
        \end{aligned}$
        & 
        $\begin{aligned}
            &176 ~\text{non-zero weights} \\
            &L^2-\text{norm} = 0.43  \\
            &\sup -\text{norm} = 1.51
        \end{aligned}$
	\end{tabular}
	\caption{The number of non-zero weights and the 
        $L^2$ and $\sup$ norms of the difference between 
        $\vph$ and the approximation are recorded for 
        each technique. The $L^2$ and $\sup$ norm values 
        are recorded to 2 decimal places. For each value of 
        $\n$ we first find the LASSO approximation. 
        Then we record the values for the first GEIM 
        approximation that at least matches the LASSO 
        approximation on both $L^2$ and $\sup$ norm values. 
        Finally we record the values for the first GRIM
        approximation that at least matches the GEIM 
        approximation on both $L^2$ and $\sup$ norm values.}
	\label{Maday_L2_Task_Results}
\end{table}
\vskip 4pt
\noindent
In each case the GRIM approximation is a linear combination
of fewer functions in $\cf$ than both the GEIM approximation 
and the LASSO approximation. Moreover, the GRIM approximation 
of $\vph$ is at least as good as both the GEIM approximation and 
the LASSO approximation in both the $L^2$-norm sense and the 
$\sup$-norm sense.

The second task we consider is a kernel quadrature problem 
appearing in Section 3.2 of \cite{HLO21}. In particular, 
we consider the \emph{3D Road Network} data set \cite{JKY13} 
of 434874 elements in $\R^3$ and the 
\emph{Combined Cycle Power Plant} data set \cite{GKT12}
of 9568 elements in $\R^5$.
For the 3D Road Network data set we take a random subset 
$\Omega$ of size $43487 \approx 434874/10$, whilst for the 
Combined Cycle Power Plant data set we take $\Omega$ to be 
the full 9568 elements.

In both cases we consider $X := \m[\Omega]$ to be the 
Banach space of signed measures on $\Omega$, and take
$\cf$ to be the collection
of point masses supported at points in $\Omega$, i.e. 
$\cf := \left\{ \de_p : p \in \Omega \right\} \subset \m[\Omega]$.
For the collection of linear functionals 
$\Sigma \subset \m[\Omega]^{\ast}$ we take 
$\Sigma$ to be the closed unit ball of the 
\emph{Reproducing Kernel Hilbert Space} (RKHS) $\cH_k$ 
associated to the kernel $k : \Omega \times \Omega \to \R$ 
defined by 
\beq
    \label{Satoshi_Kern_Quad_kernel_def}
        k(x,y) := e^{-\frac{||x-y||^2}{2 \lambda^2}}.
\eeq
Here $|| \cdot ||$ denotes the Euclidean norm on the appropriate
Euclidean space, and $\lambda$ is a hyperparameter determined by
\emph{median heuristics} \cite{HLO21}.
We let $\vph$ denote the equally weighted probability measure 
over $\Omega$ and consider the kernel quadrature problem for 
$\vph$ with respect to the RHKS $\cH_k$. By defining 
$\n := \#(\Omega)$ and
$a_1 = \ldots = a_{\n} = 1/\n$, we observe that this problem 
is within our framework.

In addition to implementing GRIM, we additionally implement a 
modified version of GRIM, which we denote GRIM $+$ opt.
The '$+$ opt' refers to applying the convex optimisation 
detailed in \cite{HLO21} to the weights returned by GRIM at 
each step. The performance of GRIM and GRIM $+$ opt is 
compared with the performance of the methods 
N. $+$ emp, N. $+$ emp $+$ opt, Monte Carlo, iid Bayes, Thinning, 
Thin $+$ opt, Herding, and Herd $+$ opt considered in 
\cite{HLO21}. Details of these methods may be found in 
\cite{HLO21} and the references there in.
We make extensive use of the python code associated with 
\cite{HLO21} available via GitHub (\href{
https://github.com/satoshi-hayakawa/kernel-quadrature
}{Convex Kernel Quadrature GitHub}).

We implement GRIM under the condition that, at each step, 
4 new functions from $\Sigma$ are added to the collection 
of functions over which we require the approximation to coincide
with the target $\vph$.
The performance of each approximation is measured by its
\emph{Worst Case Error} (WCE) with respect to $\vph$ over 
the RKHS $\cH_k$. This is defined as 
\beq
    \label{Satoshi_WCE_def}
        \text{WCE}(u,\vph,\cH_k) := 
        \sup_{f \in \ovB_{\cH_k}(0,1)}
        \left| u(f) - \vph(f) \right|
        =
        \sup_{f \in \ovB_{\cH_k}(0,1)}
        \left| \int_{\Omega} f(x) du(x) - 
        \int_{\Omega} f(x) d \vph(x) \right|,
\eeq
and may be explicitly computed in this setting by the 
formula provided in \cite{HLO21}.
For each method we record the average 
$\log( \text{WCE}(u,\vph,\cH_k)^2 )$ over 20 trials.
The results are illustrated in Figures \ref{3Dnet_results}
and \ref{PPlant_results}.

For both the 3D Road Network and the Combined Cycle Power Plant
data sets the novel recombination-based convex kernel quadrature 
method developed by Satoshi Hayakawa, the first author, and 
Harald Oberhauser in \cite{HLO21} and our GRIM approach
comfortably out perform the other methods, each of which is 
either purely growth-based or purely thinning-based. 
Moreover, the convex kernel quadrature method of \cite{HLO21} is 
specifically tailored to the task of kernel quadrature. 
Whilst being significantly slower, GRIM $+$ opt
nevertheless matches the performance of N. $+$ emp $+$ opt
despite not being specially designed for the task of 
kernel quadrature.
Moreover, even without the additional '$+$ opt'
convex optimisation step, GRIM remains within the same
class of performance as the N. $+$ emp $+$ opt method.

\begin{figure}[H]
\center
\includegraphics[width=0.7\textwidth]
{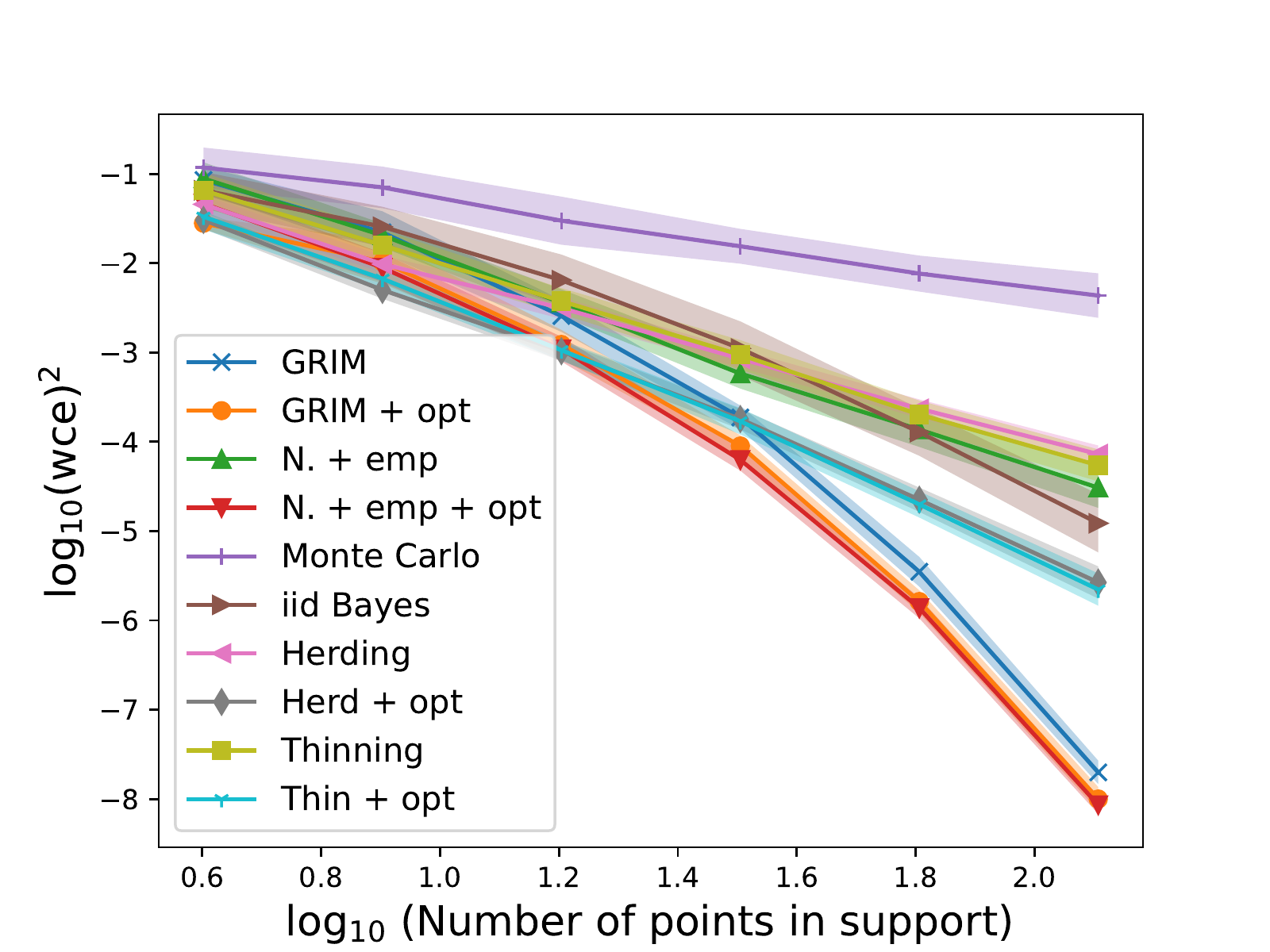}
\caption{3D Road Network Results - The average 
$\log( \text{WCE}(u,\vph,\cH_k)^2 )$ over 20 trials is 
plotted for each method. The shaded regions show their 
standard deviations.}
\label{3Dnet_results}
\end{figure}

\begin{figure}[H]
\center
\includegraphics[width=0.7\textwidth]
{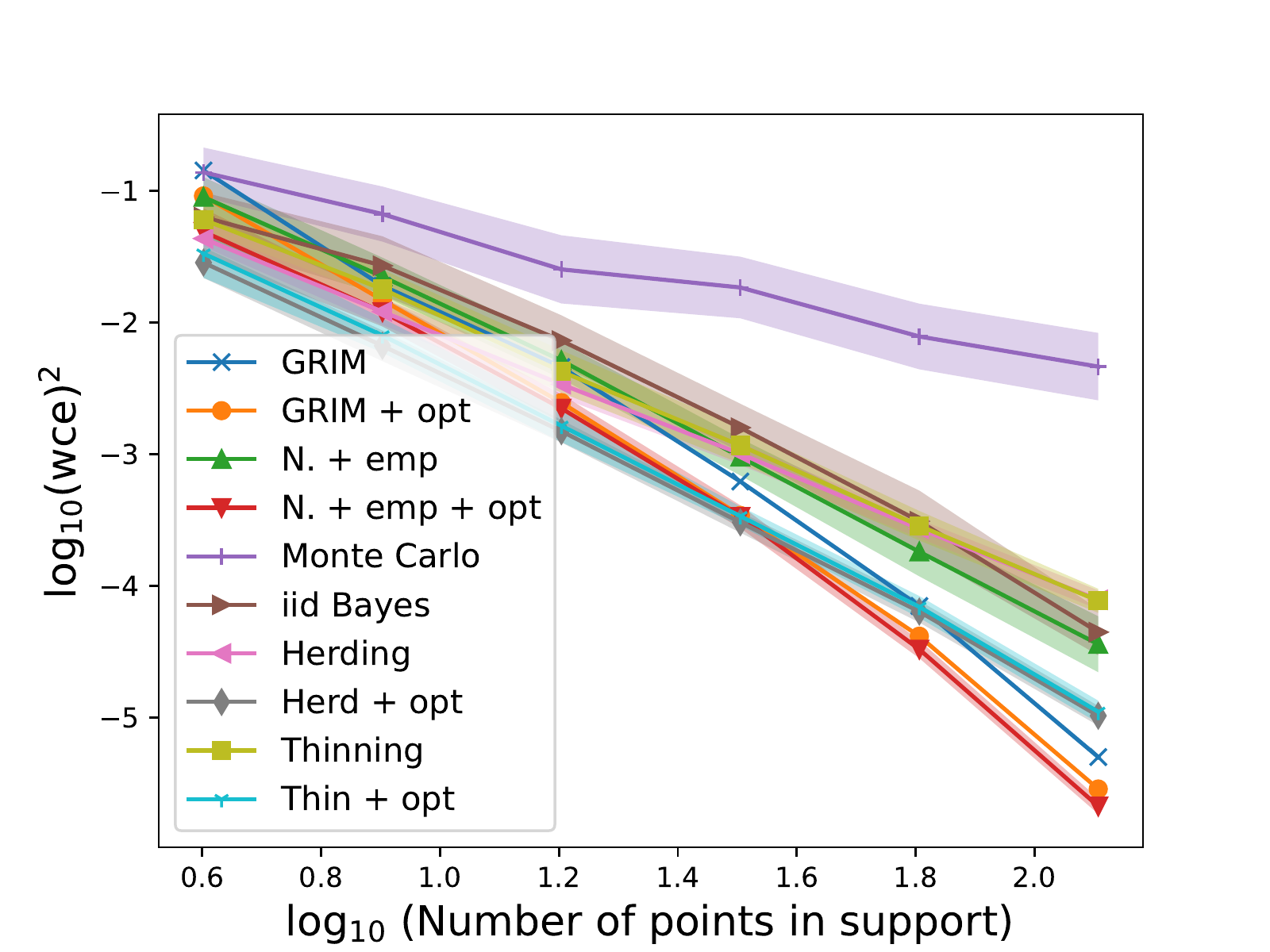}
\caption{Combined Cycle Power Plant Results - The average 
$\log( \text{WCE}(u,\vph,\cH_k)^2 )$ over 20 trials is 
plotted for each method. The shaded regions show their 
standard deviations.}
\label{PPlant_results}
\end{figure}
\vskip 4pt
\noindent
The third and final task we consider is a machine learning 
inference acceleration task motivated by \cite{JLNSY17}. 
The problem is \emph{outside} the Hilbert space framework of the previous examples.
We consider the \emph{path signature} based \emph{Landmark Human Action Recognition} (LHAR) 
model developed in \cite{JLNSY17}.

This model is trained to determine an action from video clips 
of the evolution of 15-40 markers placed on a persons body. 
It utilises signatures of streams generated by the locations of these markers.
An introduction to the use of path signatures in machine 
learning contexts can be found in the survey papers \cite{CK16} and \cite{LM22}.
We do not assume any familiarity with signatures and treat them 
as a 'black box' tool used by the LHAR model of \cite{JLNSY17}.

We restrict consider the LHAR model from \cite{JLNSY17} 
on the JHMDB data set \cite{BGJSZ13} consisting of 928 clips 
of 21 distinct actions, with each clip containing 15-40 frames.
Following \cite{JLNSY17}, the dataset is split into 
660 clips to be used for training and 268 clips to be used for testing.

The pipeline for the LHAR model from \cite{JLNSY17} can be summarised as 
\beq
    \label{LHAR_model}
        \text{Clip} \quad
        \stackrel{\text{Signatures}}{\rightarrow}
        \quad \R^{430920} \quad
        \stackrel{\text{Linear Map}}{\rightarrow}
        \quad \R^{21} \quad
        \stackrel{\text{Bias}}{\rightarrow}
        \quad \R^{21} \quad 
        \stackrel{\text{Softmax}}{\rightarrow}
        \quad \{0, \ldots , 20\}.
\eeq
A variety of truncated path signatures are computed of 
augmentations of the stream of marker locations provided 
by the clip; see \cite{JLNSY17} for full details.
The result is that each clip is transformed into a 
vector of $N := 430920$ features, 
i.e. an element in $\R^{430920} = \R^N$.
Let $\Omega_{\train} \subset \R^{430920}$ denote the collection 
of these vectors for the 660 training clips, 
and $\Omega_{\test} \subset \R^{430920}$ denote the collection 
of these vectors for the 268 testing clips.
A single-hidden-layer neural network is trained as
the classifier outputting a probability distribution 
given by a Softmax over the 21 class labels 
$\{0, \ldots , 20\}$.

Let $A : \R^{N} \to \R^{21}$ denote the Linear Map 
part of the model in \eqref{LHAR_model}.
For later convenience, given any map $v : \R^N \to \R^{21}$
we let $\cl[v]$ denote the model from \cite{JLNSY17} with 
the linear map $A$ replaced by $v$ (i.e. the pipeline 
described in \eqref{LHAR_model} but with the Linear Map 
$A$ replaced by the function $v$). We consider approximating the model $\cl[A]$. 
Our approach is to find an approximation $v$ of $A$, and 
then use the model $\cl[v]$ as the approximation of $\cl[A]$.
For this purpose, we consider both $\R^N$ and $\R^{21}$ to be 
equipped with their respective $l^{\infty}$ norms.
We first observe that the task of approximating $A$ is 
within the mathematical framework of this paper.

Let $(a_{i,j})_{i=1,j=1}^{21,N}$ denote the 
coefficients of the matrix associated to $A$, and 
for each $i \in \{1, \ldots , N\}$ define 
$f_i \in C^0\left( \R^N ; \R^{21} \right)$ by, 
for $x = (x_1 , \ldots , x_N ) \in \R^N$, setting 
$f_i(x) := x_i \left( a_{1,i} , \ldots , a_{21,i} \right)$. 
Let $e_1 , \ldots , e_{21} \in \R^{21}$ be the standard basis of $\R^{21}$ that 
is orthonormal with respect to the standard Euclidean dot product 
$\left< \cdot , \cdot \right>_{\R^{21}}$ on $\R^{21}$.
For each $p \in \Omega_{\train}$ and each $j \in \{1, \ldots , 21\}$ let 
$\de_{p,j} : C^0(\R^N;\R^{21}) \to \R$ be defined by 
$\de_{p,j}[f] := \left< f(p) , e_j \right>_{\R^{21}}$.
Then $A = \sum_{i=1}^N f_i$ and hence,
by choosing $X := C^0 \left( \Omega ; \R^{21} \right)$, 
$\n := N (= 430920)$, $\vph := A$, and 
$\Sigma := \left\{ \de_{p,j} : p \in \Omega_{\train} \text{ and } j \in \{1, \ldots , 21\} \right\}$, 
we observe that we are within the 
framework for which GRIM is designed (cf. Section \ref{Motivation}).

To provide a benchmark we use the Scikit-Learn 
MultiTaskLasso \cite{BBCDDGGMPPPPTVVW11} implementation 
to approximate $\vph$. 
Several different choices for the regularisation 
parameter $\al$ are considered. 
For each $\al$ we consider the Scikit-Learn MultiTaskLasso 
\cite{BBCDDGGMPPPPTVVW11} with the maximum number of iterations
set to 10000.
We record the number of non-zero weights
in the resulting approximation $u$, the training set 
accuracy achieved by $\cl[u]$, and the testing set 
accuracy achieved by $\cl[u]$. These results are 
summarised in Table \ref{LHAR_lasso_results}.

\begin{table}[H]
	\centering
	\begin{tabular}{|c|c|c|c|}
        \hline 
        \textbf{Alpha}
        &
        \textbf{Number of Non-zero Weights} 
        & 
        \textbf{Train Accuracy (\%) (2.d.p)}
        &
        \textbf{Test Accuracy (\%) (2.d.p)}
        \\ 
		\hline	
        0.0001 & 5464 & 100.00 & 81.34 \\
        0.0002 & 4611 & 100.00 & 82.09 \\
        0.0003 & 4217 & 100.00 & 83.21 \\
        0.0004 & 4021 & 100.00 & 82.84 \\
        0.0005 & 4009  & 100.00 & 82.46  \\
        0.001 & 3913 & 100.00 & 82.46  \\
        0.005 & 3157  & 99.70 & 81.72  \\
        0.01 & 2455 & 99.24 & 81.72  \\
        0.05 & 1012 & 95.30 & 79.85 \\
        \hline
	\end{tabular}
	\caption{Each row contains information corresponding
        to one application of MultiTaskLasso 
        \cite{BBCDDGGMPPPPTVVW11}
        to find an approximation $u$ of $\vph$.
        The model $\cl[u]$ then gives an approximation to 
        the model $\cl[\vph]$ from \cite{JLNSY17} on 
        the JHMDB dataset \cite{BGJSZ13}.
        The first column records the value of the regularisation 
        parameter $\al$ for this application of 
        MultiTaskLasso.
        The second column records the number of non-zero 
        weights appearing in the returned approximation $u$.
        The third column records the accuracy achieved 
        on the training set (to 2 decimal places) by
        the model $\cl[u]$.
        The fourth column records the accuracy achieved 
        on the test set (to 2 decimal places) by
        the model $\cl[u]$.}
	\label{LHAR_lasso_results}
\end{table}
\noindent
In terms of the test accuracy score achieved by the model 
$\cl[u]$, the choice of $\al := 3 \times 10^{-4}$ performed best.
For this choice of $\al$ the corresponding model $\cl[u]$ 
achieved an accuracy score of $83.21\%$ (2.d.p) on the test set 
whilst only using 4217 non-zero weights.
In comparison, the original model $\cl[\vph]$ developed in 
\cite{JLNSY17} uses 430920 non-zero weights, and achieves an 
accuracy score of $83.58\%$ (2.d.p) on the test set.

With the Lasso performance acting as a benchmark, we 
consider applying GRIM to approximate $\vph$.
We consider several choices for the number of new points
from $\Sigma$ to be added to the collection of points 
over which we require the approximation to coincide with
the target $\vph$. We fix the shuffle number for each step as 16. 
We apply GRIM to approximate $\vph$ over $\Omega_{\train}$,
i.e. we use only the training set. 
The test set is not used for constructing the approximation
$u$, and is only used to record the accuracy achieved by model 
generated by the approximation $\cl[u]$.

We make the following adaptation to the GRIM algorithm for this task. 
Our aim is to find an approximation $u$ of $\vph$ that is close in a 
pointwise sense to $\vph$ throughout $\Omega_{\train}$.
Instead of considering each linear functional in $\Sigma$ individually, 
we consider each collection $\Delta_{p} := \left\{ \de_{p,j} : j \in \{1, \ldots , 21\} \right\}$ 
for a point $p \in \Omega_{\train}$.
We modify the \textbf{Banach Extension Step} (cf. Section \ref{Banach_GRIM_alg}) 
to the \textbf{Modified Extension Step} below.
\vskip 4pt
\noindent
\textbf{Modified Extension Step}

\noindent
Assume that $L' \subset \Sigma$. Let $u \in \Span(\cf)$. Let $m \in \Z_{\geq 1}$ such that 
$\#(L') + 21m \leq \#(\Sigma)$. Take 
\beq
	\label{mod_banach_ext_step_sigma_1}
		\sigma_1 := \argmax \left\{ |\sigma(\vph-u)| : \sigma \in \Sigma \right\}.
\eeq
Let $p_1 \in \Omega_{\train}$ denote the point for which $\sigma_1 \in \Delta_{p_1}$.

\noindent
Inductively for $j=2,\ldots,m$ take 
\beq
	\label{mod_banach_ext_step_sigma_j}
		\sigma_j := \argmax \left\{ |\sigma(\vph-u)| : \sigma \in \Sigma \setminus (\Delta_{p_1} \cup \dots \cup \Delta_{p_{j-1}}) \right\}
\eeq
and let $p_j \in \Omega_{\train}$ denote the point for which $\sigma_j \in \Delta_{p_j}$.

\noindent 
Once the points $p_1 , \ldots , p_m \in \Omega_{\train}$ have been selected, extend $L'$ to 
$L := L' \cup \Delta_{p_1} \cup \dots \cup \Delta_{p_m}$.
\vskip 4pt
\noindent
Since $||u(p) - \vph(p)||_{l^{\infty}(\R^{21})} = \max \left\{ |\sigma (u-\vph)| : \sigma \in \Delta_p \right\}$, 
we observe that the \textbf{Modified Extension Step} is equivalent to taking 
\beq
	\label{point_p_1}
		p_1 := \argmax \left\{ ||u(p)-\vph(p)||_{l^{\infty}(\R^{21})} : p \in \Omega_{\train} \right\},
\eeq 
then inductively taking for $j=2,\ldots,m$ the point 
\beq
	\label{point_p_j}
		p_j := \argmax \left\{ ||u(p)-\vph(p)||_{l^{\infty}(\R^{21})} : p \in \Omega_{\train} \setminus \{p_1 , \ldots , p_{j-1} \} \right\},
\eeq 
and then extending $L'$ to $L := L' \cup \Delta_{p_1} \cup \dots \cup \Delta_{p_m}$.
Consequently, replacing the \textbf{Banach Extension Step} with the 
\textbf{Modified Extension Step} in the \textbf{Banach GRIM} algorithm in Section \ref{Banach_GRIM_alg} results 
in an algorithm in which a collection of points, at which we require an approximation $u$ of $\vph$ to coincide with 
$\vph$, is inductively greedily grown. 
It is important to keep in mind that each newly selected point $p \in \Omega_{\train}$ corresponds to choosing 21 new linear 
functionals in $\Sigma$.
Thus the restriction on the total number of points $K$ to possibly be selected is that $21K \leq \#(\Sigma)$, 
which is equivalent to requiring $K \leq \#(\Omega_{\train}) = 660$.
It is this modification of the \textbf{Banach GRIM} algorithm (cf. Section \ref{Banach_GRIM_alg}) that we consider for 
the task of approximating the linear map $\vph : \R^N \to \R^{21}$.

Our aim here is that the model $\cl[u]$ achieves a similar accuracy score as the original model $\cl[\vph]$ from 
\cite{JLNSY17}. Consequently, we alter the termination criteria to better reflect this goal.
The model from \cite{JLNSY17} achieves an accuracy score of $83.58\%$ on the test set.
We choose $\de := 10^{-3}$ and terminate if the approximation generates a model achieving an accuracy score
of at least $(1-\de)$ times the accuracy score of the original model on the testing set.
Hence we require the approximation to generate a model achieving a testing set accuracy score of at least 
$83.50\%$ to 2 decimal places.
We further record the accuracy achieved on the training set for comparison with the $100\%$ accuracy achieved by the 
original model here.

For each choice of number of points to add at each step, we run the GRIM algorithm (modified as outlined above) three times. 
We record the minimum number of non-zero weights returned 
after an approximation generating a model with the required test set accuracy score. 
Subsequently, we record the training accuracy score and testing accuracy score of the models found by each run
using this number of non-zero weights. 
In addition to recording these values for the best performing model, we additionally record the variance of the 
training accuracy score and testing accuracy score over the three runs. The results are summarised in Table 
\ref{LHAR_reduction_results} below.

\begin{table}[H]
	\centering
	\begin{tabular}{|c|c|c|c|c|c|}
        \hline
		\textbf{Points Added} 
        & 
        \textbf{Number of} 
        & 
        \multicolumn{2}{c|}{\textbf{Train Accuracy (\%)}}
        &
        \multicolumn{2}{c|}{\textbf{Test Accuracy (\%)}}
        \\ 
        \textbf{Per Step} & \textbf{Non-zero Weights} & 
        Best & Variance & Best & Variance \\
		\hline	
        10 & 2101 & 100.00 & 0.04 & 83.58 & 1.12  \\
        20 & 2941 & 100.00 & 0.00 & 83.58 & 0.09 \\
        30 & 2521 & 100.00 & 0.02 & 83.58 & 0.49 \\
        40 & 2521 & 100.00 & 0.01 & 84.33 & 1.33 \\
        50 & 2101 & 99.55 & 0.04 & 83.96 & 0.65 \\
        \hline
	\end{tabular}
	\caption{Each row contains information corresponding
        to our application of GRIM to approximate the 
        action recognition model from \cite{JLNSY17} on 
        the JHMDB dataset \cite{BGJSZ13}.
        The first column records the number of new points 
        added at each step of GRIM.
        The second column records the minimum number of non-zero 
        weights appearing in a returned approximation $u$ 
        generating a model $\cl[u]$ achieving a testing set 
        accuracy score of at least $83.50\%$ (2.d.p).
        The third column summarises the training accuracy 
        scores achieved during the three runs for each choice 
        of number of points to add per step. The first half 
        of the third column records the best training accuracy 
        score (to 2 decimal places) achieved by
        the model $\cl[u]$ for an approximation $u$ of $\vph$ 
        returned by one of the runs of GRIM. 
        The second half of the third column records the variance
        (to 2 decimal places)
        of the training accuracy scores achieved by the 
        models $\cl[u]$ for the approximations $u$ of $\vph$ 
        returned by each of the three runs of GRIM.
        The fourth column summarises the testing accuracy 
        scores achieved during the three runs for each choice 
        of number of points to add per step.
        The first half of the fourth column records the 
        best test accuracy score (to 2 decimal places) achieved 
        by the model $\cl[u]$ for an approximation $u$ of $\vph$ 
        returned by one of the runs of GRIM.
        The second half of the third column records the variance
        (to 2 decimal places)
        of the test accuracy scores achieved by the 
        models $\cl[u]$ for the approximations $u$ of $\vph$ 
        returned by each of the three runs of GRIM.}
	\label{LHAR_reduction_results}
\end{table}
\noindent
For each choice of the number of new points to be added at 
each step, GRIM successfully finds an approximation $u$ of 
$\vph$ which is both uses fewer non-zero weights than the 
best performing model found via Lasso, and achieves a 
higher test accuracy score than the best performing model 
found via Lasso.
The best performance, in terms of accuracy score on the test
set, was achieved by GRIM with the choice of adding 40 new 
points at each step. 
The best performing model uses 2521 non-zero weights and 
achieves an accuracy score of $84.33\%$ (2.d.p) on the 
test set. 
This accuracy score is actually higher than the accuracy score
of $83.58\%$ (2.d.p) achieved by the original model $\cl[\vph]$
from \cite{JLNSY17} on the test set.

\vskip 4pt 
\noindent
University of Oxford, Radcliffe Observatory,
Andrew Wiles Building, Woodstock Rd, Oxford, 
OX2 6GG, UK.
\vskip 4pt
\noindent
TL: tlyons@maths.ox.ac.uk \\
\url{https://www.maths.ox.ac.uk/people/terry.lyons}
\vskip 4pt
\noindent
AM: andrew.mcleod@maths.ox.ac.uk \\
\url{https://www.maths.ox.ac.uk/people/andrew.mcleod}

\begin{thebibliography}{1}

\bibitem[\textcolor{blue}{ACO20}]{ACO20}
    A. Abate, F. Cosentino, and H. Oberhauser,
    \emph{A randomized algorithm to reduce the support of 
    discrete measures},
    In Adavances in Neural Information Processing Systems,
    Volume \textbf{33}, pages 15100-15110, 
    2020.

\bibitem[\textcolor{blue}{ABDHP21}]{ABDHP21}
    D. Alistarh, T. Ben-Nun, N. Dryden, T. Hoefler 
    and A. Peste,
    \emph{Sparsity in Deep Learning: Pruning and 
    growth for efficient inference and training in 
    neural networks},
    J. Mach. Learn. Res., 
    \textbf{22}(241), 1-124, 2021.
    
\bibitem[\textcolor{blue}{ABCGMM18}]{ABCGMM18}
    J.-P. Argaud, B. Bouriquet, F. de Caso, 
    H. Gong, Y. Maday and O. Mula,
    \emph{Sensor Placement in Nuclear Reactors 
    Based on the Generalized Empirical Interpolation 
    Method},
    J. Comput. Phys., vol. {\bf{363}},
    pp. 354-370, 2018.
    
\bibitem[\textcolor{blue}{ABGMM16}]{ABGMM16}
    J.-P. Argaud, B. Bouriquet, H. Gong, 
    Y. Maday and O. Mula,
    \emph{Stabilization of (G)EIM in Presence
    of Measurement Noise: Application to Nuclear
    Reactor Physics},
    In Spectral and High Order Methods for 
    Partial Differential Equations-ICOSAHOM 2016,
    vol. {\bf{119}} of Lect. Notes Comput. 
    Sci. Eng., pp.133-145, 
    Springer, Cham, 2017.

\bibitem[\textcolor{blue}{AK13}]{AK13}
    M. G. Augasts and T. Kathirvalavakumar,
    \emph{Pruning algorithms of neural networks - 
    A comparative study},
    Open Computer Science \textbf{3}, 
    p.105-115, 2013.

\bibitem[\textcolor{blue}{BLL15}]{BLL15}
    Francis Bach, Simon Lacoste-Julien and 
    Fredrik Lindsten, 
    \emph{Sequential Kernel Herding: Frank-Wolfe
    Optimization for Particle Filtering},
    In Artificial Intelligence and Statistics,
    pages 544-552, PMLR, 2015.
  
\bibitem[\textcolor{blue}{BMPSZ08}]{BMPSZ08}    
    F. R. Bach, J. Mairal, J. Ponce, G. Sapiro and
    A. Zisserman,  
    \emph{Discriminative learned dictionaries for local
    image analysis},
    Compute Vision and Pattern Recognition (CVPR), 
    pp.1-8, IEEE 2008.
    
\bibitem[\textcolor{blue}{BMPSZ09}]{BMPSZ09}    
    F. R. Bach, J. Mairal, J. Ponce, G. Sapiro and
    A. Zisserman,
    \emph{Supervised dictionary learning},
    Advances in Neural Information Processing Systems,
    pp.1033-1040, 2009.
    
\bibitem[\textcolor{blue}{BMNP04}]{BMNP04}
    M. Barrault, Y. Maday, N. C. Nguyen and 
    A. T. Patera, 
    \emph{An Empirical Interpolation Method: 
    Application to Efficient Reduced-Basis 
    Discretization of Partial Differential Equations},
    C. R. Acad. Sci. Paris, 
    S\'{e}rie I., {\bf{339}}, pp.667-672,
    2004
    
\bibitem[\textcolor{blue}{BS18}]{BS18}
    A. G. Barto and R. S. Sutton,
    \emph{Reinforcement learning: an introduction},
    (2nd). Cambridge, MA, USA: MIT Press.

\bibitem[\textcolor{blue}{BT11}]{BT11}
    Alain Belinet and Christine Thomas-Agnan,
    \emph{Reproducing kernel Hilbert spaces in 
    probabiility and statistics},
    Springer Science \& Business Media, 2011.

\bibitem[\textcolor{blue}{BBCDDGGMPPPPTVVW11}]{BBCDDGGMPPPPTVVW11}
    M. Blondel, M. Brucher, D. Cournpeau, 
    V. Dubourg, E. Duchesnay, A. Gramfort, O. Grisel, 
    V. Michel, A. Passos, F. Pedregosa, M. Perrot, 
    P. Prettenhofer, B. Thirion, J. Vanderplas, 
    G. Varoquax and R. Weiss,
    \emph{Scikit-learn: Machine Learning in Python},
    Journal of Machine Learning Research, 
    \textbf{12}, 2825-2830, 
    2011.
    
\bibitem[\textcolor{blue}{BCGMO18}]{BCGMO18}
    F. -X. Briol, W. Y. Chen, J. Gorham,
    L. Mackey and C. J. Oates,
    \emph{Stein Points},
    In Proceedings of the 35th International 
    Conference on Machine Learning,
    Volume {\bf{80}}, pp. 843-852,
    PMLR, 2018.

\bibitem[\textcolor{blue}{BGJSZ13}]{BGJSZ13}
    J. Black, J. Gall, H. Jhuang, C. Schmid and S. Zuffi,
    \emph{Towards understanding action recognition},
    In IEEE International Conference on Computer Vision (ICCV),
    pp. 3192-3199, 2013.

\bibitem[\textcolor{blue}{CT05}]{CT05}    
    E. Cand\'{e}s and T. Tao,
    \emph{Decoding by linear programming},
    IEEE Trans. Inform. Theory, {\bf{51}},
    2005.
    
\bibitem[\textcolor{blue}{CRT06}]{CRT06}
    E. Cand\'{e}s, J. Romberg and T. Tao,
    \emph{Robust Uncertainty Principles:
    Exact Signal Reconstruction from Highly
    Incomplete Frequency Information},
    IEEE Trans. Inform. Theory, {\bf{52}},
    pp. 489-509, 2006.

\bibitem[\textcolor{blue}{Car81}]{Car81}
    B. Carl,
    \emph{Entropy numbers of diagonal operators with 
    an application to eigenvalue problems},
    J. Approx. Theory, 
    \textbf{32}, pp.135-150, 1981.

\bibitem[\textcolor{blue}{CS90}]{CS90}
    B. Carl and I. Stephani,
    \emph{Entropy, Compactness, and the Approximation of
    Operators},
    Cambridge University Press, Cambridge, UK, 1990.

\bibitem[\textcolor{blue}{CHXZ20}]{CHXZ20}
    L. Chen, A. Huang, S. Xu and B. Zhang, 
    \emph{Convolutional Neural Network Pruning: 
    A Survey},
    39th Chinese Control Conference (CCC), 
    IEEE, p.7458-7463, 2020.

\bibitem[\textcolor{blue}{CSW11}]{CSW11}
    Zhixiang Chen, Baohuai Sheng and Jianli Wang,
    \emph{The Covering Number for Some Mercer Kernel 
    Hilbert Spaces on the Unit Sphere},
    Taiwanese J. Math. \textbf{15}(3), 
    p.1325-1340, 2011.

\bibitem[\textcolor{blue}{CSW10}]{CSW10}
    Y. Chen, A. Smola and M. Welling,
    \emph{Super-Samples from Kernel Herding},
    In Proceedings of the Conference on Uncertainty
    in Artificial Intelligence, 2010.

\bibitem[\textcolor{blue}{CS08}]{CS08}
    A. Christmann and I. Steinwart, 
    \emph{Support Vector Machines},
    Springer Science \& Business Media, 2008.
    
\bibitem[\textcolor{blue}{CM73}]{CM73}
    Jon F. Claerbout and Francis Muir,
    \emph{Robust Modeling with Erratic Data},
    Geophysics,
    Vol. {\bf{38}},
    So. 5, pp. 826-844,
    (1973)

\bibitem[\textcolor{blue}{CK16}]{CK16}
    Ilya Chevyrev and Andrey Kormilitzin,
    \emph{A Primer on the Signature Method in Machine Learning},
    arXiv preprint, 2016.
    \url{https://arxiv.org/abs/1603.03788}

\bibitem[\textcolor{blue}{DGOY08}]{DGOY08}
    J. Darbon, D. Goldfarb, S. Osher and W. Yin,
    \emph{Bregman Iterative Algorithms for 
    L1 Minimization with Applications to 
    Compressed Sensing},
    SIAM J. Imaging Sci. {\bf{1}}, pp.143-168,
    2008.
    
\bibitem[\textcolor{blue}{Don06}]{Don06}
    D. Donoho,
    \emph{Compressed Sensing},
    IEEE Trans. Inform. Theory, 
    {\bf{52}}, pp. 1289-1306, 2006.
    
\bibitem[\textcolor{blue}{DE03}]{DE03}
    D. Donoho and M. Elad, 
    \emph{Optimally Sparse Representation 
    in General (Nonorthogonal) Dictionaries
    via $l^1$ Minimization},
    Proc. Natl. Acad. Sci. USA, 
    {\bf{100}}, pp.2197-2202, 2003.

\bibitem[\textcolor{blue}{DET06}]{DET06}
    D. L. Donoho, M. Elad and V. Templyakov,
    \emph{Stable recovery of sparse overcomplete representations 
    in the presence of noise},
    IEEE Transactions on Information Theory, 
    \textbf{52(1)}: 6-18, 
    2006.
    
\bibitem[\textcolor{blue}{DM21a}]{DM21a}
    Raaz Dwivedi and Lester Mackey,
    \emph{Kernel Thinning},
    Proceedings of Machine Learning Research 
    vol. \textbf{134}:1–1, 2021.
    
\bibitem[\textcolor{blue}{DM21b}]{DM21b}
    Raaz Dwivedi and Lester Mackey,
    \emph{Generalized Kernel Thinning},
    Published in ICLR 2022.
    
\bibitem[\textcolor{blue}{DMS21}]{DMS21}
    Raaz Dwivedi, Lester Mackey and
    Abhishek Shetty,
    \emph{Distribution Compression in Near-Linear
    Time},
    Published in ICLR 2022.
    
\bibitem[\textcolor{blue}{DF15}]{DF15} 
    Chris Dyer and Manaal Faruqui,
    \emph{Non-distributional Word Vector 
    Representations},
    Proceedings of the 53rd Annual Meeting of the 
    Association for Computational Linguistics and 
    the 7th International Joint Conference on 
    Natural Language Processing,
    Association for Computational Linguistics,
    Beijing, China,
    Volume {\bf{2}}: Short Papers,
    pp.464-469, 2015.

\bibitem[\textcolor{blue}{ET96}]{ET96}
    D. E. Edmunds and H. Triebel, 
    \emph{Function Spaces, Entropy Numbers, and 
    Differential Operators},
    Cambridge University Press, Cambridge, UK, 1996.

\bibitem[\textcolor{blue}{Ela10}]{Ela10}
    M. Elad, 
    \emph{Sparse and Redundant Representations: From 
    Theory to Applications in Signal and Image Processing},
    Springer, 
    ISBN 978-1441970107,
    2010.
    
\bibitem[\textcolor{blue}{EMS08}]{EMS08}
    M. Elad, J. Mairal and G. Sapiro,
    \emph{Sparse representation for color image
    restoration},
    IEEE Transactions on Image Processing,
    {\bf{17}}(1), 53-69, 2008.

\bibitem[\textcolor{blue}{EJOPR20}]{EJOPR20}
    E. Elsen, S. Jayakumar, S. Osindere, R. Pascanu and 
    J. Rae,
    \emph{Top-kast: Top-k always sparse training},
    Advances in Neural Information Processing Systems, 
    \textbf{33}:20744-20754, 2020.

\bibitem[\textcolor{blue}{EPP00}]{EPP00}
    T. Evgeniou, M. Pontil and T. Poggio,
    \emph{Regularization networks and support vector 
    machines},
    Adv. Comput. Math., \textbf{13}, pp.1-50, 2000.

\bibitem[\textcolor{blue}{FHLLMMPSWY23}]{FHLLMMPSWY23}
    Meng Fang, Tianjin Huang, Gen Li, Shiwei Liu, 
    Xiaolong Ma, Vlado Menkovski, Mykola Pechenizkiy, 
    Li Shen, Zhangyang Wang and Lu Yin, 
    \emph{Dynamic Sparsity in Channel-Level Sparsity Learner},
    NeurIPS 2023.

\bibitem[\textcolor{blue}{FJM19}]{FJM19}
    Dan Feldman, Ibrahim Jubran and Alaa Maalouf,
    \emph{Fast and Accuract Least-Mean-Squares Solvers},
    Advances in Neural Information Processing Systems 
    \textbf{32} (NeurIPS2019), 2019.

\bibitem[\textcolor{blue}{FJM22}]{FJM22}
    Dan Feldman, Ibrahim Jubran and Alaa Maalouf,
    \emph{Fast and Accurate Least-Mean-Squares Solvers for 
    High Dimensional Data},
    in IEEE Transactionf on Pattern Analysis and Machine 
    Intelligence, 
    vol. \textbf{44}, no. 12, pp. 9977-9994, 
    2022.

\bibitem[\textcolor{blue}{FS21}]{FS21}
    S. Fischer and I. Steinwart, 
    \emph{A closer look at covering number bounds 
    for Gaussian kernels},
    Journal of Complexity, 
    Volume \textbf{62}, 2021.

\bibitem[\textcolor{blue}{FMMT14}]{FMMT14}
    Alona Fyshe, Tom M. Mitchell, Brian Murphy 
    and Partha P. Talukdar,
    \emph{Interpretable semantic Vectors from a 
    joint model of brain-and text-based meaning},
    Proceedings of the 52nd Annual Meeting of the 
    Association for Computational Linguistics,
    Association for Computational Linguistics,
    Baltimore, Maryland,
    Volume {\bf{1}}: Long Papers,
    pp.489-499, 2014.

\bibitem[\textcolor{blue}{GMSWY09}]{GMSWY09}
    A. Ganesh, Y. Ma, S. S. Sastry, J. Wright 
    and A. Yang,
    \emph{Robust face recognition via sparse 
    representation},
    TPAMI {\bf{31}}(2), 210-227, 2009.

\bibitem[\textcolor{blue}{GLMNS18}]{GLMNS18}
    M. Gibescu, A. Liotta, D. C. Mocanu, E. Mocanu, 
    P. H. Nguyen and P. Stone,
    \emph{Scalable training of artificial neural networks with 
    adaptive sparse connectivity inspired by network science},
    Nature communications, 
    \textbf{9}(1):2383, 2018.

\bibitem[\textcolor{blue}{GLSWZ21}]{GLSWZ21}
    John Glossner, Tailin Liang, Shaobo Shi, 
    Lei Wang and Xiaotong Zhang,
    \emph{Pruning and quantization for deep neural 
    network acceleration: A survey},
    Neurocomputing, Volume \textbf{461}, 
    pages 370-403, 2021.

\bibitem[\textcolor{blue}{GO09}]{GO09}    
    T. Goldstein and S. Osher,
    \emph{The Split Bregman Method for 
    L1-Regularized Problems},
    SIAM J. Imaging Sci. {\bf{2}},
    pp.323-343, 2009.
    
\bibitem[\textcolor{blue}{GMNP07}]{GMNP07}
    M. A. Grepl, Y. Maday, N. C. Nguyen and 
    A. T. Patera,
    \emph{Efficient Reduced-Basis Treatment of 
    Nonaffine and Nonlinear Partial Differential
    Equations},
    M2AN (Math. Model. Numer. Anal.),
    2007.

\bibitem[\textcolor{blue}{GKT12}]{GKT12}
    F. S. G\"{u}rgen, H. Kaya and P. T\"{u}fekci,
    \emph{Local and global learning methods for predicting
    power of a combined gas \& steam turbine},
    In Proceedings of the Internationl Conference on 
    Emerging Trends in Computer and Electronics Engineering,
    pages 13-18, 
    2012.

\bibitem[\textcolor{blue}{HLO21}]{HLO21}
    Satoshi Hayakawa, Terry Lyons and Harald Oberhauser,
    \emph{Positively weighted kernel Quadrature via
    Subsampling}, 
    In Adavnces in Neural Information Processing Systems 
    (NeurIPS 2022).
    
\bibitem[\textcolor{blue}{HS19}]{HS19}
    J. Hernandez-Garcia and R. Sutton,
    \emph{Learning Sparse Representations Incrementally
    in Deep Reinforcement Learning},
    \url{https://arxiv.org/abs/1912.04002},
    [cs.LG], 9 Dec 2019.

\bibitem[\textcolor{blue}{HLLL18}]{HLLL18}
    J. Huang, H. Lian, H. Lin and S. Lv, 
    \emph{Oracle Inequalities for Sparse Additive Quantile
    Regression in Reproducing Kernel Hilbert Space},
    Annals of Statistics \textbf{46}, p.781-813, 2018.

\bibitem[\textcolor{blue}{JLPST21}]{JLPST21}
    S. Jayakumar, P. E. Latham, R. Pascanu, J. Schwarz and 
    Y. Teh,
    \emph{Powerpropagation: A sparsity inducing weight 
    reparameterisation},
    Advances in Neural Information Processing Systems,
    \textbf{34}:28889-28903, 
    2021.

\bibitem[\textcolor{blue}{JKY13}]{JKY13}
    C. S. Jensen, M. Kaul and B. Yang,
    \emph{Building accurate 3d spatial networks to enable
    next generation intelligent transportation systems},
    In 2013 IEEE 14th International Conference on Mobile Data
    Management, Volume \textbf{1}, pages 137-146, 
    IEEE, 2013.

\bibitem[\textcolor{blue}{JJWWY20}]{JJWWY20}
    C. Jin, M. I. Jordan, M. Wang, Z. Wang and 
    Z. Yang, 
    \emph{On Function Approximation in Reinforcement
    Learning: Optimism in the Face of Large State Spaces},
    34th Conference on Neural Information Processing 
    Systems (NeurIPS 2020), Vancouver, Canada, 2020

\bibitem[\textcolor{blue}{JLNSY17}]{JLNSY17} 
    Lianwen Jin, Terry Lyons, Hao Ni, Cordelia Schmid and 
    Weixin Yang,
    \emph{Developing the Path Signature Methodology
    and its Application to Landmark-based Human 
    Action Recognition},
    In: Yin, G., Zariphopoulou, T. (eds)
    Stochastic Analysis, Filtering and Stochastic Optimization, 
    Springer, Cham. 2022.
    \url{https://doi.org/10.1007/978-3-030-98519-6_18}

\bibitem[\textcolor{blue}{KK20}]{KK20}
    M. Kalini\'{c} and P. K\'{o}m\'{a}r,
    \emph{Denoising DNA encoded library screens
    with sparse learning},
    ACS Combinatorial Science,
    Vol. {\bf{22}}, no. 8, pp.410-421, 
    2020.

\bibitem[\textcolor{blue}{Kol56}]{Kol56}
    A. N. Kolmogorov, 
    \emph{Asymptotic characteristics of some completely 
    bounded metric spaces},
    Dokl. Akad. Nauk. SSSR, \textbf{108}, pp.585-589,
    1956.

\bibitem[\textcolor{blue}{Kuh11}]{Kuh11}
    T. K\"{u}hn, 
    \emph{Covering numbers of Gaussian reproducing 
    kernel Hilbert spaces}, 
    J. Complexity, \textbf{27}, pp.489-499, 2011.

\bibitem[\textcolor{blue}{LY07}]{LY07}
    Y. Lin and M. Yuan,
    \emph{On the Non-Negative Garrote Estimator},
    J. R. Stat. Soc. Ser. B {\bf{69}},
    pp. 143-161, 2007

\bibitem[\textcolor{blue}{LL16}]{LL16}
    W. Lee and T. Lyons,
    \emph{The Adaptive Patched Cubature Filter 
    and its Implementation},
    Communications in Mathematical Sciences,
    {\bf{14}}(3), pp.799-829, 2016.

\bibitem[\textcolor{blue}{LL99}]{LL99}
    W. V. Li and W. Linde, 
    \emph{Approximation, metric entropy and small ball 
    estimates for Gaussian measures},
    Ann. Probab., \textbf{27}, pp.1556-1578, 1999.
    
\bibitem[\textcolor{blue}{LL12}]{LL12}
    C. Litterer and T. Lyons,
    \emph{High order recombination and an 
    application to cubature on wiener space},
    The Annals of Applied Probability,
    Vol. {\bf{22}}, No. 4, 1301-1327,
    2012.

\bibitem[\textcolor{blue}{LMPY21}]{LMPY21}
    S. Liu, D. C. Mocanu, M. Pechenizkiy and L. Yin,
    \emph{Do we actually need dense over-parameterization? 
    in-time over-parameterization in sparse training},
    In Proceedings of the $39^{\text{th}}$ International 
    Conference on Machine Learning, pages 6989-7000, 
    PMLR 2021.

\bibitem[\textcolor{blue}{LP04}]{LP04}
    H. Luschgy and G. Pag\'{e}s,
    \emph{Sharp asymptotics of the Kolmogorov entropy 
    for Gaussian measures},
    J. Funct. Anal., \textbf{212}, pp.89-120, 2004.

\bibitem[\textcolor{blue}{LM22}]{LM22}
    Terry Lyons and Andrew D. McLeod,
    \emph{Signature Methods in Machine Learning},
    arXiv preprint, 2022.
    \url{https://arxiv.org/abs/2206.14674}
    
\bibitem[\textcolor{blue}{MM13}]{MM13}
    Y. Maday, O. Mula, 
    \emph{A generalized empirical interpolation method: 
    Application of reduced basis techniques to data 
    assimilation}, 
    F. Brezzi, P. Colli Franzone, U. Gianazza, 
    G. Gilardi (Eds.), Analysis and Numerics of 
    Partial Differential Equations, Vol. 4 of 
    Springer INdAM Series,
    Springer Milan, 2013, pp. 221–235.

\bibitem[\textcolor{blue}{MMPY15}]{MMPY15}
    Y. Maday, O. Mula, A. T. Patera and M. Yano,
    \emph{The Generalized Empirical Interpolation Method: 
    Stability Theory On Hilbert Spaces With An Application To 
    Stokes Equation},
    Comput. Methods Appl. Mech. Engrg. \textbf{287},
    310-334, 
    2015.

\bibitem[\textcolor{blue}{MMT14}]{MMT14}
    Y. Maday, O. Mula and G. Turinici,
    \emph{Convergence analysis of the Generalized 
    Empirical Interpolation Method},
    SIAM J. Numer. Anal.,
    {\bf{54(3)}} 1713-1731, 2014.
    
\bibitem[\textcolor{blue}{MNPP09}]{MNPP09}
    Y. Maday, N. C. Nguyen, A. T. Patera and 
    G. S. H. Pau,
    \emph{A General Multipurpose Interpolation
    Procedure: The Magic Points},
    Commun. Pure Appl. Anal., 
    {\bf{81}}, pp. 383-404, 
    2009.

\bibitem[\textcolor{blue}{MZ93}]{MZ93}
    S. G. Mallat and Z. Zhang,
    \emph{Matching Pursuits with Time-Frequency Dictionaries},
    IEEE Transactions on Signal Processing, 
    \textbf{12}: 3397-3415, 
    1993.

\bibitem[\textcolor{blue}{MRT12}]{MRT12}
    Mehryar Mohri, Afshin Rostamizadeh and Ameet 
    Talwalkar, 
    \emph{Foundations of Machine Learning},
    Massachusetts: MIT Press, 
    USA, 2012

\bibitem[\textcolor{blue}{NT09}]{NT09}
    D. Needell and J. A. Tropp, 
    \emph{CoSaMP: Iterative signal recovery from 
    incomplete and inaccurate samples},
    Applied and Computational Harmonic Analysis,
    \textbf{26(3)}: 301-321, 
    2009.

\bibitem[\textcolor{blue}{NPS22}]{NPS22}
    M. Nikdast, S. Pasricha and F. Sunny,
    \emph{SONIC: A Sparse Neural Network Inference 
    Accelerator with Silicon Photonics for 
    Energy-Efficient Deep Learning},
    27th Asia and South Pacific Design Automation 
    Conference (ASP-DAC), pp. 214-219, 2022.

\bibitem[\textcolor{blue}{NS21}]{NS21}
    S. Ninomiya and Y. Shinozaki, 
    \emph{On implementation of high-order recombination 
    and its application to weak approximations of 
    stochastic differential equations},
    In Proceedings of the NFA $29^{\text{th}}$ Annual Conference,
    2021.

\bibitem[\textcolor{blue}{PTT22}]{PTT22}
    Sebastian Pokutta, Ken'ichiro Tanaka and
    Kazuma Tsuji, 
    \emph{Sparser Kernel Herding with Pairwise
    Conditional Gradients without Swap Steps},
    \url{https://arxiv.org/abs/2110.12650},
    [math.OC], 8 February 2022.

\bibitem[\textcolor{blue}{Ree93}]{Ree93}
    R. Reed, 
    \emph{Pruning Algorithms - A Survey},
    IEEE Transactions on Neural Networks 
    \textbf{4}, p.74-747, 1993.
    
\bibitem[\textcolor{blue}{SS86}]{SS86}    
    Fadil Santosaf and William W. Symes,
    \emph{Linear Inversion of Band-Limited
    Reflection Seismograms},
    SIAM J. ScI. STAT. COMPUT.
    Vol. {\bf{7}}, No. 4, 
    1986.

\bibitem[\textcolor{blue}{SSW01}]{SSW01}
    B. Sch\"{o}lkopf, A. J. Smola and R. C. Williamson,
    \emph{Generalization performance of regularization 
    networks and support vector machines via entrop numbers 
    of compact operators},
    IEEE Trans. Inform. Theory, \textbf{47}, 
    pp.2516-2532, 2001.

\bibitem[\textcolor{blue}{Ste03}]{Ste03}
    I. Steinwart, 
    \emph{Entropy numbers of convex hulls and an 
    application to learning algorithms},
    Arch. Math., \textbf{80}, pp.310-318, 2003.

\bibitem[\textcolor{blue}{SS07}]{SS07}
    I. Steinwart and C. Scovel, 
    \emph{Fast rates for support vector machines using 
    Gaussian kernels},
    Ann. Statist. \textbf{35} (2), 2007.

\bibitem[\textcolor{blue}{SS12}]{SS12}
    I. Steinwart and C. Scovel, 
    \emph{Mercer's theorem on general domains:
    On the interaction between measures, kernels, and 
    RKHSs},
    Constructive Approximation,
    \textbf{35}(3):363-417, 2012.

\bibitem[\textcolor{blue}{Str71}]{Str71}
    A. H. Stroud,
    \emph{Approximate calculation of multiple 
    integrals},
    Series in Automatic Computation, 
    Englewood Cliffs, NJ: Prentice-Hall, 
    1971.

\bibitem[\textcolor{blue}{Suz18}]{Suz18}
    T. Suzuki, 
    \emph{Fast Learning Rate of Non-Sparse Multiple
    Kernel Learning and Optimal Regularization 
    Strategies},
    Electronic Journal of Statistics \textbf{12},
    p.2141-2192, 2018.

\bibitem[\textcolor{blue}{SS13}]{SS13}
    M. Sugiyama and T. Suzuki, 
    \emph{Fast Learning Rate of Multiple Kernel Learning:
    Tradeoff Between Sparsity and Smoothness},
    Annals of Statistics \textbf{41},
    p.1381-1405, 2013.

\bibitem[\textcolor{blue}{TT21}]{TT21}
    Ken'ichiro Tanaka and Kazuma Tsuji,
    \emph{Acceleration of the Kernel Herding
    Algorithm by Improved Gradient Approximation},
    \url{https://arxiv.org/abs/2105.07900},
    [math.NA], 17 May 2021. 

\bibitem[\textcolor{blue}{Tch15}]{Tch15}
    M. Tchernychova, 
    \emph{Carath\'{e}odry cubature measures},
    PhD thesis, 
    University of Oxford, 
    \url{https://ora.ox.ac.uk/objects/uuid:a3a10980-d35d-467b-b3c0-d10d2e491f2d}
    2015.
    
\bibitem[\textcolor{blue}{TW19}]{TW19}
    GL. Tian and M. Wang,
    \emph{Adaptive Group LASSO for High-Dimensional
    Generalized Linear Models},
    Stat Papers {\bf{60}}, pp. 1469-1486, 
    2019.

\bibitem[\textcolor{blue}{Tib96}]{Tib96}
    Robert Tibshirani,
    \emph{Regression Shrinkage and Selection 
    via the Lasso},
    Journal of the Royal Statistical Society.
    Series B (Methodological), Vol. {\bf{58}},
    No. 1, pp. 267-288, 1996.

\bibitem[\textcolor{blue}{Tro06}]{Tro06}
    J. A. Tropp, 
    \emph{Just relax: Convex programming methods for 
    identifying sparse signals in noise},
    IEEE Transactions on Information Theory,
    \textbf{52(3)}: 1030-1051, 
    2006.
    
\bibitem[\textcolor{blue}{Wel09a}]{Wel09a}
    M. Welling,
    \emph{Herding Dynamical Weights to Learn},
    In Proceedings of the 21st International 
    Conference on Machine Learning, Montreal,
    Quebec, CAN, 2009.
    
\bibitem[\textcolor{blue}{Wel09b}]{Wel09b}
    M. Welling, 
    \emph{Herding Dynamic Weights for Partially 
    Observed Random Field Models}, 
    In Proc. of the Conf. on
    Uncertainty in Artificial Intelligence, 
    Montreal, Quebec, CAN, 2009.
    
\bibitem[\textcolor{blue}{XZ16}]{XZ16}
    Y. Xiang and C. Zhang,
    \emph{On the Oracle Property of Adaptive
    Group LASSO in High-Dimensional Linear Models},
    Stat. Pap. {\bf{57}}, pp. 249-265, 
    2016.

\bibitem[\textcolor{blue}{Zho02}]{Zho02}
    D. -X. Zhou, 
    \emph{The covering number in learning theory},
    J. Complexity, \textbf{18}, pp.739-767, 2002.

\end{thebibliography}
\end{document}